\makeatletter

\documentclass[
    11pt,
    a4paper,
    oneside,
    openright,
    center,
    chapterbib,
    crosshair,
    reqno,
    headcount,
    headline,
    indent,
    indentfirst=false,
    portrait,
    phonetic,
    oldernstyle,
    onecolumn,
    sfbold,
    upper,
]{amsart}

\let\th@plain\relax

\PassOptionsToPackage{T2A}{fontenc} 
\PassOptionsToPackage{utf8}{inputenc} 
\PassOptionsToPackage{cyrpart}{cyrillic}
\PassOptionsToPackage{british,ngerman,russian}{babel}
\PassOptionsToPackage{
    english,
    ngerman,
    russian,
    capitalise,
}{cleveref}
\PassOptionsToPackage{
    margin=10pt,
    position=bottom,
    justification=centering,
    font=small,
    labelformat=simple,
    labelfont={sc},
    labelsep={period},
    textfont={it},
}{caption}
\PassOptionsToPackage{
    margin=10pt,
    position=bottom,
    justification=centering,
    font=small,
    labelformat=parens,
    labelfont={bf},
    labelsep={space},
    textfont={it},
}{subcaption}
\PassOptionsToPackage{framemethod=TikZ}{mdframed}
\PassOptionsToPackage{
    bookmarks=true,
    bookmarksopen=false,
    bookmarksopenlevel=0,
    bookmarkstype=toc,
    raiselinks=true,
    colorlinks=true,
    hyperfigures=true,
    anchorcolor=red,
    citecolor=green,
    linkcolor=blue,
    urlcolor=blue,
}{hyperref} 
\PassOptionsToPackage{normalem}{ulem}
\PassOptionsToPackage{
    amsmath,
    thmmarks,
}{ntheorem}
\PassOptionsToPackage{
    table,
}{xcolor}
\PassOptionsToPackage{
    all,
    color,
    curve,
    frame,
    import,
    knot,
    line,
    movie,
    rotate,
    textures,
    tile,
    tips,
    web,
    xdvi,
}{xy}
\PassOptionsToPackage{
    reset,
    left=1in,
    right=1in,
    top=20mm,
    bottom=20mm,
    heightrounded,
}{geometry}
\PassOptionsToPackage{
    symbol*, 
    multiple, 
}{footmisc}
\PassOptionsToPackage{overload}{textcase}
\PassOptionsToPackage{indentafter}{titlesec}
\PassOptionsToPackage{
    shortcuts, 
}{extdash}

\usepackage{amsfonts}
\usepackage{amsmath}
\usepackage{amssymb}
\usepackage{ntheorem} 
\usepackage{array}
\usepackage{babel}
\usepackage{bbding}
\usepackage{bm}
\usepackage{bbm}
\usepackage{bibentry}
\usepackage{booktabs}
\usepackage{bold-extra} 
\usepackage{calc}
\usepackage{cancel}
\usepackage{caption} 
\usepackage{changepage}
\usepackage{cjhebrew}
\usepackage{cmlgc}
\usepackage{colonequals}
\usepackage{color}
\usepackage{comment}
\usepackage{datetime}
\usepackage{dsfont}
\usepackage{etex}
\usepackage{etoolbox}
\usepackage{eurosym}
\usepackage{extdash}
\usepackage{fancybox}
\usepackage{fancyhdr}
\usepackage{float}
\usepackage{fontenc}
\usepackage{footmisc}
\usepackage{fp}
\usepackage{geometry}
\usepackage{graphicx}
\usepackage{ifpdf}
\usepackage{ifthen}
\usepackage{ifoddpage}
\usepackage{ifnextok}  
\usepackage{inputenc}
\usepackage{latexsym}
\usepackage{lineno}
\usepackage{listings}
\usepackage{longtable}
\usepackage{lscape}
\usepackage{mathtools} 
\usepackage{mathrsfs}
\usepackage{multicol}
\usepackage{multirow}
\usepackage{nameref}
\usepackage{nowtoaux}
\usepackage{paralist}
\usepackage{enumerate} 
\usepackage{pgf}
\usepackage{pgfplots}
\usepackage{phonetic}
\usepackage{proof}
\usepackage{qtree}
\usepackage{refcount}
\usepackage{savesym}
\usepackage{stmaryrd}
\usepackage{synttree}
\usepackage{subcaption}
\usepackage{suffix}
\usepackage{yfonts} 
\usepackage{textcase} 
\usepackage{tikz}
\usepackage{xy}
\usepackage{wrapfig}
\usepackage{xcolor}
\usepackage{xspace}
\usepackage{xstring}
\usepackage{hyperref}
\usepackage{arydshln}
\usepackage{cleveref} 

\pgfplotsset{compat=newest}
\usetikzlibrary{math}

\usetikzlibrary{
    angles,
    arrows,
    automata,
    calc,
    decorations,
    decorations.pathmorphing,
    decorations.pathreplacing,
    positioning,
    patterns,
    quotes,
}

\savesymbol{corresponds}
\savesymbol{Diamond}
\savesymbol{emptyset}
\savesymbol{ggg}
\savesymbol{int}
\savesymbol{lll}
\savesymbol{RectangleBold}
\savesymbol{langle}
\savesymbol{rangle}
\savesymbol{hookrightarrow}
\savesymbol{hookleftarrow}
\savesymbol{Asterisk}
\usepackage{mathabx}
\usepackage{wasysym}

\restoresymbol{x}{corresponds}
\restoresymbol{x}{Diamond}
\restoresymbol{x}{emptyset}
\restoresymbol{x}{ggg}
\restoresymbol{x}{int}
\restoresymbol{x}{lll}
\restoresymbol{x}{RectangleBold}
\restoresymbol{x}{langle}
\restoresymbol{x}{rangle}
\restoresymbol{x}{hookrightarrow}
\restoresymbol{x}{hookleftarrow}
\restoresymbol{x}{Asterisk}

\ifpdf
    \usepackage{pdfcolmk}
\fi

\usepackage{mdframed}

\DeclareFontFamily{U}{MnSymbolA}{}
\DeclareFontShape{U}{MnSymbolA}{m}{n}{
    <-6> MnSymbolA5
    <6-7> MnSymbolA6
    <7-8> MnSymbolA7
    <8-9> MnSymbolA8
    <9-10> MnSymbolA9
    <10-12> MnSymbolA10
    <12-> MnSymbolA12
}{}
\DeclareFontShape{U}{MnSymbolA}{b}{n}{
    <-6> MnSymbolA-Bold5
    <6-7> MnSymbolA-Bold6
    <7-8> MnSymbolA-Bold7
    <8-9> MnSymbolA-Bold8
    <9-10> MnSymbolA-Bold9
    <10-12> MnSymbolA-Bold10
    <12-> MnSymbolA-Bold12
}{}
\DeclareSymbolFont{MnSymA}{U}{MnSymbolA}{m}{n}
\DeclareMathSymbol{\lcirclearrowright}{\mathrel}{MnSymA}{252}
\DeclareMathSymbol{\lcirclearrowdown}{\mathrel}{MnSymA}{255}
\DeclareMathSymbol{\rcirclearrowleft}{\mathrel}{MnSymA}{250}
\DeclareMathSymbol{\rcirclearrowdown}{\mathrel}{MnSymA}{251}

\DeclareFontFamily{U}{MnSymbolC}{}
\DeclareSymbolFont{MnSyC}{U}{MnSymbolC}{m}{n}
\DeclareFontShape{U}{MnSymbolC}{m}{n}{
    <-6>  MnSymbolC5
    <6-7>  MnSymbolC6
    <7-8>  MnSymbolC7
    <8-9>  MnSymbolC8
    <9-10> MnSymbolC9
    <10-12> MnSymbolC10
    <12->   MnSymbolC12%
}{}
\DeclareMathSymbol{\powerset}{\mathord}{MnSyC}{180}
\DeclareMathSymbol{\righthalfcap}{\mathbin}{MnSyC}{186}

\DeclareMathAlphabet{\mathpzc}{OT1}{pzc}{m}{it}

\def\boolwahr{true}
\def\boolfalsch{false}
\def\boolleer{}

\let\boolinappendix\boolfalsch
\let\boolinmdframed\boolfalsch

\newcount\bufferctr
\newcount\bufferreplace

\newlength\rtab
\newlength\gesamtlinkerRand
\newlength\gesamtrechterRand
\newlength\ownspaceabovethm
\newlength\ownspacebelowthm
\def\secnumberingpt{.}
\def\secnumberingseppt{.}
\def\subsecnumberingseppt{}
\def\thmnumberingpt{.}
\def\thmnumberingseppt{}
\def\thmForceSepPt{.}

\definecolor{leer}{gray}{1}
\definecolor{boxgrau}{gray}{0.85}
\definecolor{dunkelgrau}{gray}{0.5}
\definecolor{maroon}{rgb}{0.6901961,0.1882353,0.3764706}
\definecolor{dunkelgruen}{rgb}{0.015625,0.363281,0.109375}
\definecolor{dunkelrot}{rgb}{0.5450980392,0,0}
\definecolor{dunkelblau}{rgb}{0,0,0.5450980392}
\definecolor{blau}{rgb}{0,0,1}
\definecolor{newresult}{rgb}{0.6,0.6,0.6}
\definecolor{improvedresult}{rgb}{0.9,0.9,0.9}
\definecolor{hervorheben}{rgb}{0,0.9,0.7}
\definecolor{starkesblau}{rgb}{0.1019607843,0.3176470588,0.8156862745}
\definecolor{achtung}{rgb}{1,0.5,0.5}
\definecolor{frage}{rgb}{0.5,1,0.5}
\definecolor{schreibweise}{rgb}{0,0.7,0.9}
\definecolor{axiom}{rgb}{0,0.3,0.3}

\definecolor{drawing_light_grey}{gray}{0.85}

\def\let@name#1#2{
    \expandafter\let\csname #1\expandafter\endcsname\csname #2\endcsname\relax
}
\DeclareRobustCommand\crfamily{\fontfamily{ccr}\selectfont}
\DeclareTextFontCommand{\textcr}{\crfamily}

\def\ifthenelseleer#1#2#3{\ifthenelse{\equal{#1}{}}{#2}{#1#3}}
\def\bedingtesspaceexpand#1#2#3{\ifthenelseleer{\csname #1\endcsname}{#3}{#2#3}}

\def\noparskip{\vspace{-3\parskip}}
\def\nvraum{\@ifnextchar\bgroup{\nvraum@c}{\nvraum@bes}}
    \def\nvraum@c#1{\vspace*{-#1\baselineskip}}
    \def\nvraum@bes{\vspace*{-\baselineskip}}
\def\forceaddspace{\relax\ifmmode\else\@\xspace\fi}
\def\forceremovespace{\relax\ifmmode\else\expandafter\@gobble\fi}

\def\send@toaux#1{\@bsphack\protected@write\@auxout{}{\string#1}\@esphack}

\def\rlabel#1[#2]#3#4#5{#5\rlabel@aux{#1}[#2]{#3}{#4}{#5}}
    \def\rlabel@aux#1[#2]#3#4#5{%
        \send@toaux{\newlabel{#1}{{\@currentlabel}{\thepage}{{\unexpanded{#5}}}{#2.\csname the#2\endcsname}{}}}\relax%
    }

\def\tag@rawscheme#1#2[#3]#4#5{\@ifnextchar[{\tag@rawscheme@{#1}{#2}[#3]{#4}{#5}}{\tag@rawscheme@{#1}{#2}[#3]{#4}{#5}[*]}}
    \def\tag@rawscheme@#1#2[#3]#4#5[#6]{\@ifnextchar\bgroup{\tag@rawscheme@@{#1}{#2}[#3]{#4}{#5}[#6]}{\tag@rawscheme@@{#1}{#2}[#3]{#4}{#5}[#6]{}}}
    \def\tag@rawscheme@@#1#2[#3]#4#5[#6]#7{%
        \ifthenelse{\equal{#6}{*}}{%
            \ifthenelse{\equal{#7}{\boolleer}}{\refstepcounter{#3}#4\csname the#3\endcsname#5}{#4#7#5}%
        }{%
            \refstepcounter{#3}#4%
            \ifthenelse{\equal{#7}{\boolleer}}{\rlabel{#6}[#3]{#1}{#2}{\csname the#3\endcsname}}{\rlabel{#6}[#3]{#1}{#2}{#7}}%
            #5%
        }%
    }
\def\tag@scheme#1#2[#3]{\tag@rawscheme{#1}{#2}[#3]{\upshape(}{\upshape)}}

\def\eqtag@post#1{\makebox[0pt][r]{#1}}
\def\eqtag@pre{\tag@scheme{Eq}{Equation}[equation]}
\def\eqtag{\@ifnextchar[{\eqtag@}{\eqtag@[*]}}
    \def\eqtag@[#1]{\@ifnextchar\bgroup{\eqtag@@[#1]}{\eqtag@@[#1]{}}}
    \def\eqtag@@[#1]#2{\eqtag@post{\eqtag@pre[#1]{#2}}}

\def\eqcref#1{\text{(\ref{#1})}}
\def\punktlabel#1{\label{it:#1:\beweislabel}}
\def\punktcref#1{\eqcref{it:#1:\beweislabel}}
\def\punktref#1{\ref{it:#1:\beweislabel}}

\def\Crefit#1#2{\Cref{#1}~\eqcref{it:#2:#1}}

\def\opfromto[#1]_#2^#3{\underset{#2}{\overset{#3}{#1}}}
\def\textoverset#1#2{\overset{\text{#1}}{#2}}

\def\eqcrefoverset#1#2{\textoverset{\eqcref{#1}}{#2}}

\def\mathclap#1{#1}
\def\oberunterset#1{\@ifnextchar^{\oberunterset@oben{#1}}{\oberunterset@unten{#1}}}
    \def\oberunterset@oben#1^#2_#3{\underset{\mathclap{#3}}{\overset{\mathclap{#2}}{#1}}}
    \def\oberunterset@unten#1_#2^#3{\underset{\mathclap{#2}}{\overset{\mathclap{#3}}{#1}}}
    \def\breitunderbrace#1_#2{\underbrace{#1}_{\mathclap{#2}}}
    \def\breitoverbrace#1^#2{\overbrace{#1}^{\mathclap{#2}}}
    \def\breitunderbracket#1_#2{\underbracket{#1}_{\mathclap{#2}}}
    \def\breitoverbracket#1^#2{\overbracket{#1}^{\mathclap{#2}}}

\def\generatenestedsecnumbering#1#2#3{%
    \expandafter\gdef\csname thelong#3\endcsname{%
        \expandafter\csname the#2\endcsname%
        \secnumberingpt%
        \expandafter\csname #1\endcsname{#3}%
    }%
    \expandafter\gdef\csname theshort#3\endcsname{%
        \expandafter\csname #1\endcsname{#3}%
    }%
}
\def\generatenestedthmnumbering#1#2#3{%
    \expandafter\gdef\csname the#3\endcsname{%
        \expandafter\csname the#2\endcsname%
        \thmnumberingpt%
        \expandafter\csname #1\endcsname{#3}%
    }%
    \expandafter\gdef\csname theshort#3\endcsname{%
        \expandafter\csname #1\endcsname{#3}%
    }%
}

\providecommand{\setcounternach}{}
\renewcommand{\setcounternach}[2]{\setcounter{#1}{#2}\addtocounter{#1}{-1}}

\def\forcepunkt#1{#1\IfEndWith{#1}{.}{}{.}}

\def\matrix#1{\left(\begin{array}{#1}}
    \def\endmatrix{\end{array}\right)}
\def\smatrix{\left(\begin{smallmatrix}}
    \def\endsmatrix{\end{smallmatrix}\right)}

\def\multiargrekursiverbefehl#1#2#3#4#5#6#7#8{%
    \expandafter\gdef\csname#1\endcsname #2##1#4{\csname #1@anfang\endcsname##1#3\egroup}
    \expandafter\def\csname #1@anfang\endcsname##1#3{#5##1\@ifnextchar\egroup{\csname #1@ende\endcsname}{#7\csname #1@mitte\endcsname}}
    \expandafter\def\csname #1@mitte\endcsname##1#3{#6##1\@ifnextchar\egroup{\csname #1@ende\endcsname}{#7\csname #1@mitte\endcsname}}
    \expandafter\def\csname #1@ende\endcsname##1{#8}
}
\multiargrekursiverbefehl{svektor}{[}{;}{]}{\begin{smatrix}}{}{\\}{\\\end{smatrix}}
\multiargrekursiverbefehl{vektor}{[}{;}{]}{\begin{matrix}{c}}{}{\\}{\\\end{matrix}}
\multiargrekursiverbefehl{vektorzeile}{}{,}{;}{}{&}{}{}
\multiargrekursiverbefehl{matlabmatrix}{[}{;}{]}{\begin{smatrix}\vektorzeile}{\vektorzeile}{;\\}{;\end{smatrix}}

\def\BeweisRichtung[#1]{\@ifnextchar\bgroup{\@BeweisRichtung@c[#1]}{\@BeweisRichtung@bes[#1]}}
    \def\@BeweisRichtung@bes[#1]{{\bfseries (#1)}}
    \def\@BeweisRichtung@c[#1]#2#3{#2~#1~#3}
\def\erzeugeBeweisRichtungBefehle#1#2{
    \expandafter\gdef\csname #1text\endcsname##1##2{\BeweisRichtung[#2]{##1}{##2}}
    \expandafter\gdef\csname #1\endcsname{%
        \@ifnextchar\bgroup{\csname #1@\endcsname}{\csname #1text\endcsname{}{}}%
    }
    \expandafter\gdef\csname #1@\endcsname##1##2{%
        \csname #1text\endcsname{\punktcref{##1}}{\punktcref{##2}}%
    }
}
\erzeugeBeweisRichtungBefehle{hinRichtung}{$\Rightarrow$}
\erzeugeBeweisRichtungBefehle{herRichtung}{$\Leftarrow$}
\erzeugeBeweisRichtungBefehle{hinherRichtung}{$\Leftrightarrow$}

\def\brkt#1{\langle{}#1{}\rangle}
\def\mathfrak#1{\mbox{\usefont{U}{euf}{m}{n}#1}}

\def\rectangleblack{\text{\RectangleBold}}

\def\squareblack{\blacksquare}

\def\create@abbreviation#1#2{
    \expandafter\gdef\csname #1\endcsname{%
        #2\@ifnextchar.{%
            \relax\ifmmode\else\expandafter\@gobble\fi%
        }{%
            \relax\ifmmode\else\@\xspace\fi%
        }%
    }
}

\create@abbreviation{cf}{\emph{cf.}}
\create@abbreviation{Cf}{\emph{Cf.}}
\create@abbreviation{idest}{\emph{i.e.}}
\create@abbreviation{Idest}{\emph{I.e.}}
\create@abbreviation{exempli}{\emph{e.g.}}
\create@abbreviation{Exempli}{\emph{E.g.}}
\create@abbreviation{etcetera}{\emph{etc.}}
\create@abbreviation{etAlia}{\emph{et al.}}
\create@abbreviation{viz}{\emph{viz.}}

\create@abbreviation{Tfae}{T.\,f.\,a.\,e.}
\create@abbreviation{tfae}{t.\,f.\,a.\,e.}
\create@abbreviation{wrt}{wrt.}
\create@abbreviation{randomvar}{r.\,v.}
\create@abbreviation{iid}{i.\,i.\,d.}
\create@abbreviation{almostEverywhere}{a.\,e.}
\create@abbreviation{almostSurely}{a.\,s.}
\create@abbreviation{withoutlog}{w.\,l.\,o.\,g.}
\create@abbreviation{Withoutlog}{W.\,l.\,o.\,g.}
\create@abbreviation{respectively}{resp.}

\def\crefname@full#1#2#3#4#5{%
    \crefname{#1}{#2}{#3}
    \Crefname{#1}{#4}{#5}
}
\def\crefname@fullmod#1#2#3#4#5{%
    \crefname@full{#1}{#2}{#3}{#4}{#5}
    \crefname@full{#1@basic}{#2}{#3}{#4}{#5}
    \crefname@full{#1@withName}{#2}{#3}{#4}{#5}
}
\crefname@full{chapter}{chapter}{chapters}{Chapter}{Chapters}
\crefname@full{appendix}{appendix}{appendices}{Appendix}{Appendices}
\crefname@full{section}{section}{sections}{Section}{Sections}
\crefname@full{subsection}{section}{sections}{Section}{Sections}
\crefname@full{subsubsection}{section}{sections}{Section}{Sections}
\crefname@full{subsubsubsection}{section}{sections}{Section}{Sections}
\crefname@full{table}{table}{tables}{Table}{Tables}
\crefname@full{figure}{figure}{figures}{Figure}{Figures}
\crefname@full{subfigure}{figure}{figures}{Figure}{Figures}

\crefname@fullmod{thm}{theorem}{theorems}{Theorem}{Theorems}
\crefname@fullmod{thmStar}{theorem}{theorems}{Theorem}{Theorems}
\crefname@fullmod{conj}{conjecture}{conjectures}{Conjecture}{Conjectures}
\crefname@fullmod{cor}{corollary}{corollaries}{Corollary}{Corollaries}
\crefname@fullmod{defn}{definition}{definitions}{Definition}{Definitions}
\crefname@fullmod{conv}{convention}{conventions}{Convention}{Conventions}
\crefname@fullmod{e.g.}{example}{examples}{Example}{Examples}
\crefname@fullmod{prop}{proposition}{propositions}{Proposition}{Propositions}
\crefname@fullmod{proof}{proof}{proofs}{Proof}{Proofs}
\crefname@fullmod{lemm}{lemma}{lemmata}{Lemma}{Lemmata}
\crefname@fullmod{problem}{problem}{problems}{Problem}{Problems}
\crefname@fullmod{qstn}{question}{questions}{Question}{Questions}
\crefname@fullmod{rem}{remark}{remarks}{Remark}{Remarks}

\def\qedEIGEN#1{\@ifnextchar[{\qedEIGEN@c{#1}}{\qedEIGEN@bes{#1}}}
\def\qedEIGEN@bes#1{%
    \parfillskip=0pt
    \widowpenalty=10000
    \displaywidowpenalty=10000
    \finalhyphendemerits=0
    \leavevmode
    \unskip
    \nobreak
    \hfil
    \penalty50
    \hskip.2em
    \null
    \hfill
    #1
    \par%
}
\def\qedEIGEN@c#1[#2]{%
    \parfillskip=0pt
    \widowpenalty=10000
    \displaywidowpenalty=10000
    \finalhyphendemerits=0
    \leavevmode
    \unskip
    \nobreak
    \hfil
    \penalty50
    \hskip.2em
    \null
    \hfill
    {#1~{\small\bfseries\upshape (#2)}}%
    \par%
}
\def\qedVARIANT#1#2{
    \expandafter\def\csname ennde#1Sign\endcsname{#2}
    \expandafter\def\csname ennde#1\endcsname{\@ifnextchar[{\qedEIGEN@c{#2}}{\qedEIGEN@bes{#2}}} 
}
\qedVARIANT{OfProof}{$\squareblack$}
\qedVARIANT{OfWork}{\rectangleblack}
\qedVARIANT{OfSomething}{$\dashv$}
\qedVARIANT{OnNeutral}{} 

\def\ra@pretheoremwork{
    \setlength{\theorempreskipamount}{\ownspaceabovethm}
    \setlength{\theorempostskipamount}{\ownspacebelowthm}
}
\def\rathmtransfer#1#2{
    \expandafter\def\csname #2\endcsname{\csname #1\endcsname}
    \expandafter\def\csname end#2\endcsname{\csname end#1\endcsname}
}

\def\ranewthm#1#2#3[#4]{
    \theoremstyle{\current@theoremstyle}
    \theoremseparator{\current@theoremseparator}
    \theoremprework{\ra@pretheoremwork}
    \@ifundefined{#1@basic}{\newtheorem{#1@basic}[#4]{#2}}{\renewtheorem{#1@basic}[#4]{#2}}
    \theoremstyle{\current@theoremstyle}
    \theoremseparator{\thmForceSepPt}
    \theoremprework{\ra@pretheoremwork}
    \@ifundefined{#1@withName}{\newtheorem{#1@withName}[#4]{#2}}{\renewtheorem{#1@withName}[#4]{#2}}
    \theoremstyle{nonumberplain}
    \theoremseparator{\thmForceSepPt}
    \theoremprework{\ra@pretheoremwork}
    \@ifundefined{#1@star@basic}{\newtheorem{#1@star@basic}[#4]{#2}}{\renewtheorem{#1@star@basic}[#4]{#2}}
    \theoremstyle{nonumberplain}
    \theoremseparator{\thmForceSepPt}
    \theoremprework{\ra@pretheoremwork}
    \@ifundefined{#1@star@withName}{\newtheorem{#1@star@withName}[#4]{#2}}{\renewtheorem{#1@star@withName}[#4]{#2}}
    \umbauenenv{#1}{#3}[#4]
    \umbauenenv{#1@star}{#3}[#4]
    \rathmtransfer{#1@star}{#1*}
}

\def\umbauenenv#1#2[#3]{%
    \expandafter\def\csname #1\endcsname{\relax%
        \@ifnextchar[{\csname #1@\endcsname}{\csname #1@\endcsname[*]}%
    }
    \expandafter\def\csname #1@\endcsname[##1]{\relax%
        \@ifnextchar[{\csname #1@@\endcsname[##1]}{\csname #1@@\endcsname[##1][*]}%
    }
    \expandafter\def\csname #1@@\endcsname[##1][##2]{%
        \ifx*##1%
            \def\enndeOfBlock{\csname end#1@basic\endcsname}
            \csname #1@basic\endcsname%
        \else%
            \def\enndeOfBlock{\csname end#1@withName\endcsname}
            \csname #1@withName\endcsname[##1]%
        \fi%
        \def\makelabel####1{%
            \gdef\beweislabel{####1}%
            \label{\beweislabel}%
        }%
        \ifx*##2%
            \def\enndeSymbol{\qedEIGEN{#2}}
        \else%
            \def\enndeSymbol{\qedEIGEN{#2}[##2]}
        \fi
    }
    \expandafter\gdef\csname end#1\endcsname{\enndeSymbol\enndeOfBlock}
}

    \def\current@theoremstyle{plain}
    \def\current@theoremseparator{\thmnumberingseppt}
    \theoremstyle{\current@theoremstyle}
    \theoremseparator{\current@theoremseparator}
    \theoremsymbol{}

\generatenestedthmnumbering{arabic}{section}{equation}

\generatenestedthmnumbering{arabic}{section}{X}
\generatenestedthmnumbering{Roman}{section}{Xsp}

    \theoremheaderfont{\upshape\bfseries}
    \theorembodyfont{\slshape}

\ranewthm{thm}{Theorem}{\enndeOnNeutralSign}[X]
\ranewthm{lemm}{Lemma}{\enndeOnNeutralSign}[X]
\ranewthm{cor}{Corollary}{\enndeOnNeutralSign}[X]
\ranewthm{prop}{Proposition}{\enndeOnNeutralSign}[X]

    \theorembodyfont{\upshape}

\ranewthm{defn}{Definition}{\enndeOnNeutralSign}[X]
\ranewthm{conv}{Convention}{\enndeOnNeutralSign}[X]
\ranewthm{e.g.}{Example}{\enndeOnNeutralSign}[X]
\ranewthm{fact}{Fact}{\enndeOnNeutralSign}[X]
\ranewthm{problem}{Problem}{\enndeOnNeutralSign}[X]
\ranewthm{qstn}{Question}{\enndeOnNeutralSign}[X]
\ranewthm{rem}{Remark}{\enndeOnNeutralSign}[X]

    \theoremheaderfont{\itshape}
    \theorembodyfont{\upshape}

\ranewthm{proof@tmp}{Proof}{\enndeOfProofSign}[Xdisplaynone]
\rathmtransfer{proof@tmp*}{proof}

\def\shortclaim@claim{%
    \iflanguage{british}{Claim}{%
    \iflanguage{english}{Claim}{%
    \iflanguage{ngerman}{Behauptung}{%
    \iflanguage{russian}{Утверждение}{%
    Claim%
    }}}}%
}
\def\shortclaim@pf@kurz{%
    \iflanguage{british}{Pf}{%
    \iflanguage{english}{Pf}{%
    \iflanguage{ngerman}{Bew}{%
    \iflanguage{russian}{Доказательство}{%
    Pf%
    }}}}%
}
\def\shortclaim{\@ifnextchar\bgroup{\shortclaim@c}{\shortclaim@bes}}
    \def\shortclaim@c#1{\item[{\bfseries \shortclaim@claim\forceaddspace #1.}]}
    \def\shortclaim@bes{\item[{\bfseries \shortclaim@claim.}]}
\def\proofofshortclaim{\item[{\bfseries\itshape\shortclaim@pf@kurz.}]}

\newdateformat{standardshort}{\oldstylenums{\THEYEAR}.\oldstylenums{\THEMONTH}.\oldstylenums{\THEDAY}}
\newdateformat{standardcompact}{\THEYEAR\twodigit{\THEMONTH}\twodigit{\THEDAY}}
\newdateformat{standardlong}{\THEYEAR\ \monthname\ \THEDAY}
\newcolumntype{\RECHTS}[1]{>{\raggedleft}p{#1}}
\newcolumntype{\LINKS}[1]{>{\raggedright}p{#1}}
\newcolumntype{m}{>{$}l<{$}}
\newcolumntype{C}{>{$}c<{$}}
\newcolumntype{L}{>{$}l<{$}}
\newcolumntype{R}{>{$}r<{$}}
\newcolumntype{0}{@{\hspace{0pt}}}
\newcolumntype{\LINKSRAND}{@{\hspace{\@totalleftmargin}}}
\newcolumntype{h}{@{\extracolsep{\fill}}}
\newcolumntype{i}{>{\itshape}}
\newcolumntype{t}{@{\hspace{\tabcolsep}}}
\newcolumntype{q}{@{\hspace{1em}}}
\newcolumntype{n}{@{\hspace{-\tabcolsep}}}
\newcolumntype{M}[2]{%
    >{\begin{minipage}{#2}\begin{math}}%
    {#1}%
    <{\end{math}\end{minipage}}%
}
\newcolumntype{T}[2]{%
    >{\begin{minipage}{#2}}%
    {#1}%
    <{\end{minipage}}%
}
\def\punkteumgebung@genbefehl#1#2#3{
    \punkteumgebung@genbefehl@{#1}{#2}{#3}{}{}
    \punkteumgebung@genbefehl@{multi#1}{#2}{#3}{
        \setlength{\columnsep}{10pt}%
        \setlength{\columnseprule}{0pt}%
        \begin{multicols}{\thecolumnanzahl}%
    }{\end{multicols}\nvraum{1}}
}
\def\punkteumgebung@genbefehl@#1#2#3#4#5{
    \expandafter\gdef\csname #1\endcsname{
        \@ifnextchar\bgroup{\csname #1@c\endcsname}{\csname #1@bes\endcsname}
    }
        \expandafter\def\csname #1@c\endcsname##1{
            \@ifnextchar[{\csname #1@c@\endcsname{##1}}{\csname #1@c@\endcsname{##1}[\z@]}
        }
        \expandafter\def\csname #1@c@\endcsname##1[##2]{
            \@ifnextchar[{\csname #1@c@@\endcsname{##1}[##2]}{\csname #1@c@@\endcsname{##1}[##2][\z@]}
        }
        \expandafter\def\csname #1@c@@\endcsname##1[##2][##3]{
            \let\alterlinkerRand\gesamtlinkerRand
            \let\alterrechterRand\gesamtrechterRand
            \addtolength{\gesamtlinkerRand}{##2}
            \addtolength{\gesamtrechterRand}{##3}
            \advance\linewidth -##2%
            \advance\linewidth -##3%
            \advance\@totalleftmargin ##2%
            \parshape\@ne \@totalleftmargin\linewidth%
            #4
            \begin{#2}[\upshape ##1]%
                \setlength{\parskip}{0.5\baselineskip}\relax%
                \setlength{\topsep}{\z@}\relax%
                \setlength{\partopsep}{\z@}\relax%
                \setlength{\parsep}{\parskip}\relax%
                \setlength{\itemsep}{#3}\relax%
                \setlength{\listparindent}{\z@}\relax%
                \setlength{\itemindent}{\z@}\relax%
        }
        \expandafter\def\csname #1@bes\endcsname{
            \@ifnextchar[{\csname #1@bes@\endcsname}{\csname #1@bes@\endcsname[\z@]}
        }
        \expandafter\def\csname #1@bes@\endcsname[##1]{
            \@ifnextchar[{\csname #1@bes@@\endcsname[##1]}{\csname #1@bes@@\endcsname[##1][\z@]}
        }
        \expandafter\def\csname #1@bes@@\endcsname[##1][##2]{
            \let\alterlinkerRand\gesamtlinkerRand
            \let\alterrechterRand\gesamtrechterRand
            \addtolength{\gesamtlinkerRand}{##1}
            \addtolength{\gesamtrechterRand}{##2}
            \advance\linewidth -##1%
            \advance\linewidth -##2%
            \advance\@totalleftmargin ##1%
            \parshape\@ne \@totalleftmargin\linewidth%
            #4
            \begin{#2}%
                \setlength{\parskip}{0.5\baselineskip}\relax%
                \setlength{\topsep}{\z@}\relax%
                \setlength{\partopsep}{\z@}\relax%
                \setlength{\parsep}{\parskip}\relax%
                \setlength{\itemsep}{#3}\relax%
                \setlength{\listparindent}{\z@}\relax%
                \setlength{\itemindent}{\z@}\relax%
        }
    \expandafter\gdef\csname end#1\endcsname{%
        \end{#2}#5
        \setlength{\gesamtlinkerRand}{\alterlinkerRand}
        \setlength{\gesamtlinkerRand}{\alterrechterRand}
    }
}

\def\ritempunkt{{\Large \textbullet}} 
\setdefaultitem{\ritempunkt}{\ritempunkt}{\ritempunkt}{\ritempunkt}
\punkteumgebung@genbefehl{itemise}{compactitem}{\parskip}{}{}
\punkteumgebung@genbefehl{kompaktitem}{compactitem}{\z@}{}{}
\punkteumgebung@genbefehl{enumerate}{compactenum}{\parskip}{}{}
\punkteumgebung@genbefehl{kompaktenum}{compactenum}{\z@}{}{}

\def\matrix#1{\left(\begin{array}[mc]{#1}}
    \def\endmatrix{\end{array}\right)}
\def\smatrix{\left(\begin{smallmatrix}}
    \def\endsmatrix{\end{smallmatrix}\right)}

\def\multiargrekursiverbefehl#1#2#3#4#5#6#7#8{%
    \expandafter\gdef\csname#1\endcsname #2##1#4{\csname #1@anfang\endcsname##1#3\egroup}
    \expandafter\def\csname #1@anfang\endcsname##1#3{#5##1\@ifnextchar\egroup{\csname #1@ende\endcsname}{#7\csname #1@mitte\endcsname}}
    \expandafter\def\csname #1@mitte\endcsname##1#3{#6##1\@ifnextchar\egroup{\csname #1@ende\endcsname}{#7\csname #1@mitte\endcsname}}
    \expandafter\def\csname #1@ende\endcsname##1{#8}
}
\multiargrekursiverbefehl{svektor}{[}{;}{]}{\begin{smatrix}}{}{\\}{\\\end{smatrix}}
\multiargrekursiverbefehl{vektor}{[}{;}{]}{\begin{matrix}{c}}{}{\\}{\\\end{matrix}}
\multiargrekursiverbefehl{vektorzeile}{}{,}{;}{}{&}{}{}
\multiargrekursiverbefehl{matlabmatrix}{[}{;}{]}{\begin{smatrix}\vektorzeile}{\vektorzeile}{;\\}{;\end{smatrix}}

\def\underbracenodisplay#1{%
    \mathop{\vtop{\m@th\ialign{##\crcr
    $\hfil\displaystyle{#1}\hfil$\crcr
    \noalign{\kern3\p@\nointerlineskip}%
    \upbracefill\crcr\noalign{\kern3\p@}}}}\limits%
}

\def\einzug{\@ifnextchar[{\indents@}{\indents@[\z@]}}
    \def\indents@[#1]{\@ifnextchar[{\indents@@[#1]}{\indents@@[#1][\z@]}}
    \def\indents@@[#1][#2]{%
        \begin{list}{}{\relax
            \setlength{\topsep}{\z@}\relax
            \setlength{\partopsep}{\z@}\relax
            \setlength{\parsep}{\parskip}\relax
            \setlength{\listparindent}{\z@}\relax
            \setlength{\itemindent}{\z@}\relax
            \setlength{\leftmargin}{#1}\relax
            \setlength{\rightmargin}{#2}\relax
            \let\alterlinkerRand\gesamtlinkerRand
            \let\alterrechterRand\gesamtrechterRand
            \addtolength{\gesamtlinkerRand}{#1}
            \addtolength{\gesamtrechterRand}{#2}
        }\relax
            \item[]\relax
    }
        \def\endeinzug{%
            \setlength{\gesamtlinkerRand}{\alterlinkerRand}
            \setlength{\gesamtlinkerRand}{\alterrechterRand}
            \end{list}%
        }

\def\escapeeinzug{\begin{einzug}[-\gesamtlinkerRand][-\gesamtrechterRand]}
    \def\endescapeeinzug{\end{einzug}}

\def\programmiercode{
    \modulolinenumbers[1]
    \begin{einzug}[\rtab][\rtab]%
    \begin{linenumbers}%
        \fontfamily{cmtt}\fontseries{m}\fontshape{u}\selectfont%
        \setlength{\parskip}{1\baselineskip}%
        \setlength{\parindent}{0pt}%
}
    \def\endprogrammiercode{
        \end{linenumbers}
        \end{einzug}
    }

\def\schattiertebox@genbefehl#1#2#3{
    \expandafter\gdef\csname #1\endcsname{%
        \@ifnextchar[{\csname #1@args\endcsname}{\csname #1@args\endcsname[#3]}
    }
        \expandafter\def\csname #1@args\endcsname[##1]{%
            \@ifnextchar[{\csname #1@args@l\endcsname[##1]}{\csname #1@args@n\endcsname[##1]}
        }
        \expandafter\def\csname #1@args@l\endcsname[##1][##2]{%
            \@ifnextchar[{\csname #1@args@l@r\endcsname[##1][##2]}{\csname #1@args@l@n\endcsname[##1][##2]}
        }
        \expandafter\def\csname #1@args@n\endcsname[##1]{%
            \let\boolinmdframed\boolwahr
            \begin{mdframed}[#2leftmargin=0,rightmargin=0,outermargin=0,innermargin=0,##1]
        }
        \expandafter\def\csname #1@args@l@n\endcsname[##1][##2]{%
            \let\boolinmdframed\boolwahr
            \begin{mdframed}[#2leftmargin=##2/2,rightmargin=##2/2,outermargin=##2/2,innermargin=##2/2,##1]
        }
        \expandafter\def\csname #1@args@l@r\endcsname[##1][##2][##3]{%
            \let\boolinmdframed\boolwahr
            \begin{mdframed}[#2leftmargin=##2,rightmargin=##3,outermargin=##2,innermargin=##3,##1]
        }
    \expandafter\gdef\csname end#1\endcsname{%
        \end{mdframed}
        \let\boolinmdframed\boolfalsch
    }
}
    \schattiertebox@genbefehl{schattiertebox}{
        splittopskip=0,%
        splitbottomskip=0,%
        frametitleaboveskip=0,%
        frametitlebelowskip=0,%
        skipabove=1\baselineskip,%
        skipbelow=1\baselineskip,%
        linewidth=2pt,%
        linecolor=black,%
        roundcorner=4pt,%
    }{
        backgroundcolor=leer,%
        nobreak=true,%
    }

    \schattiertebox@genbefehl{schattierteboxdunn}{
        splittopskip=0,%
        splitbottomskip=0,%
        frametitleaboveskip=0,%
        frametitlebelowskip=0,%
        skipabove=1\baselineskip,%
        skipbelow=1\baselineskip,%
        linewidth=1pt,%
        linecolor=black,%
        roundcorner=2pt,%
    }{
        backgroundcolor=leer,%
        nobreak=true,%
    }

\def\tikzsetzepfeil#1{%
    \begin{tikzpicture}[remember picture,overlay,>=latex]%
        \draw #1;%
    \end{tikzpicture}%
}

\def\tikzsetzekreise[#1]#2#3{%
    \tikzsetzepfeil{%
    [rounded corners,#1]%
        ([shift={(-\tabcolsep,0.75\baselineskip)}]#2)%
        rectangle%
        ([shift={(\tabcolsep,-0.5\baselineskip)}]#3)
    }%
}

\tikzset{
    >=stealth,
    auto,
    node distance=1cm,
    thick,
    main node/.style={
        circle,draw,font=\sffamily\Large\bfseries,minimum size=0pt
    },
    state/.style={minimum size=0pt}
    loop above right/.style={loop,out=30,in=60,distance=0.5cm},
    loop above left/.style={above left,out=150,in=120,loop},
    loop below right/.style={below right,out=330,in=300,loop},
    loop below left/.style={below left,out=240,in=210,loop},
    itria/.style={
        draw,dashed,shape border uses incircle,
        isosceles triangle,shape border rotate=90,yshift=-1.45cm
    },
    rtria/.style={
        draw,dashed,shape border uses incircle,
        isosceles triangle,isosceles triangle apex angle=90,
        shape border rotate=-45,yshift=0.2cm,xshift=0.5cm
    },
    ritria/.style={
        draw,dashed,shape border uses incircle,
        isosceles triangle,isosceles triangle apex angle=110,
        shape border rotate=-55,yshift=0.1cm
    },
    litria/.style={
        draw,dashed,shape border uses incircle,
        isosceles triangle,isosceles triangle apex angle=110,
        shape border rotate=235,yshift=0.1cm
    }
}

\renewenvironment{cases}[0]{\left\{\begin{array}[c]{0lcl}}{\end{array}\right.}

\providecommand{\usesinglequotes}{}
\renewcommand{\usesinglequotes}[1]{`#1'}

\providecommand{\onetoone}{}
\renewcommand{\onetoone}[0]{\ensuremath{1\!\!:\!\!1}\forceaddspace}
\providecommand{\envPreMathsLong}{}
\renewcommand{\envPreMathsLong}[0]{%
    \bgroup\relax%
    \let\old@arraystretch\arraystretch\relax%
    \renewcommand\arraystretch{1.2}\relax\relax%
}
\providecommand{\envPostMathsLong}{}
\renewcommand{\envPostMathsLong}[0]{%
    \renewcommand\arraystretch{\old@arraystretch}\relax%
    \egroup\relax%
}
\providecommand{\complex}{}
\renewcommand{\complex}[0]{\mathbb{C}}
\providecommand{\Torus}{}
\renewcommand{\Torus}[0]{\mathbb{T}}

\providecommand{\reals}{}
\renewcommand{\reals}[0]{\mathbb{R}}
\providecommand{\realsNonNeg}{}
\renewcommand{\realsNonNeg}[0]{\reals_{\geq 0}}
\providecommand{\realsPos}{}
\renewcommand{\realsPos}[0]{\reals_{> 0}}
\providecommand{\rationals}{}
\renewcommand{\rationals}[0]{\mathbb{Q}}

\providecommand{\integers}{}
\renewcommand{\integers}[0]{\mathbb{Z}}
\providecommand{\naturals}{}
\renewcommand{\naturals}[0]{\mathbb{N}}
\providecommand{\naturalsZero}{}
\renewcommand{\naturalsZero}[0]{\mathbb{N}_{0}}
\providecommand{\naturalsPos}{}
\renewcommand{\naturalsPos}[0]{\mathbb{N}}
\providecommand{\HilbertRaum}{}
\renewcommand{\HilbertRaum}[0]{\mathcal{H}}
\providecommand{\BanachRaum}{}
\renewcommand{\BanachRaum}[0]{\mathcal{E}}
\providecommand{\Heisenberg}{}
\renewcommand{\Heisenberg}[0]{\mathcal{H}}

\providecommand{\topInterior}{}
\renewcommand{\topInterior}[1]{\textup{int}(#1)}

\providecommand{\topSOT}{}
\renewcommand{\topSOT}[0]{\text{\upshape\scshape sot}}
\providecommand{\topWOT}{}
\renewcommand{\topWOT}[0]{\text{\upshape\scshape wot}}

\providecommand{\card}{}
\renewcommand{\card}[1]{\lvert #1 \rvert}

\providecommand{\einser}{}
\renewcommand{\einser}[0]{\mathbf{1}}

\providecommand{\iunit}{}
\renewcommand{\iunit}[0]{\imath}
\providecommand{\abs}{}
\renewcommand{\abs}[1]{\lvert #1 \rvert}
\providecommand{\absLong}{}
\renewcommand{\absLong}[1]{\Big| #1 \Big|}

\providecommand{\linspann}{}
\renewcommand{\linspann}[0]{\textup{lin}}

\providecommand{\onematrix}{}
\renewcommand{\onematrix}[0]{\text{\upshape\bfseries I}}
\providecommand{\zeromatrix}{}
\renewcommand{\zeromatrix}[0]{\mathbf{0}}

\providecommand{\zerovector}{}
\renewcommand{\zerovector}[0]{\mathbf{0}}

\providecommand{\brkt}{}
\renewcommand{\brkt}[2]{\langle{}#1,\:#2{}\rangle}

\providecommand{\brktLong}{}
\renewcommand{\brktLong}[2]{\Big\langle{}#1,\:#2{}\Big\rangle}

\providecommand{\norm}{}
\renewcommand{\norm}[1]{\lVert #1 \rVert}
\providecommand{\normLong}{}
\renewcommand{\normLong}[1]{\Big\| #1 \Big\|}

\providecommand{\normCb}{}
\renewcommand{\normCb}[1]{\lVert #1 \rVert_{\textup{cb}}}

\providecommand{\opDomain}{}
\renewcommand{\opDomain}[1]{\mathcal{D}(#1)}
\providecommand{\opSpectrum}{}
\renewcommand{\opSpectrum}[1]{\sigma(#1)}

\providecommand{\opResolvent}{}
\renewcommand{\opResolvent}[2]{R(#2, #1)}
\providecommand{\BoundedOpsSymbol}{}
\renewcommand{\BoundedOpsSymbol}[0]{\mathfrak{L}}
\providecommand{\Cnought}{}
\renewcommand{\Cnought}[0]{\ensuremath{C_{0}}}

\providecommand{\isPartition}{}
\renewcommand{\isPartition}[2]{#1 \in \mathrm{Part}(#2)}

\providecommand{\sotInt}{}
\renewcommand{\sotInt}[0]{\text{\topSOT-}{}\int}
\providecommand{\dee}{}
\renewcommand{\dee}[0]{\mathop{\textup{d}}\!}

\providecommand{\funcCalcCts}{}
\renewcommand{\funcCalcCts}[0]{\Psi_{\textup{P-lM}}}
\providecommand{\funcCalcDiscr}{}
\renewcommand{\funcCalcDiscr}[0]{\Psi_{\textup{disc}}}
\providecommand{\Prob}{}
\renewcommand{\Prob}[0]{\mathbb{P}}
\providecommand{\Expected}{}
\renewcommand{\Expected}[0]{\mathbb{E}}
\providecommand{\Var}{}
\renewcommand{\Var}[0]{\mathbf{\mathrm{Var}}}
\providecommand{\distributedAs}{}
\renewcommand{\distributedAs}[0]{\mathrel{\sim}}
\providecommand{\PushForward}{}
\renewcommand{\PushForward}[2]{f_{\ast}{}#2}
\providecommand{\DistExp}{}
\renewcommand{\DistExp}[1]{\mathrm{Exp}(#1)}

\providecommand{\DistPoissScale}{}
\renewcommand{\DistPoissScale}[2]{\mathrm{Poiss}(#1, #2)}
\providecommand{\DistPoissAux}{}
\renewcommand{\DistPoissAux}[2]{\mathrm{Poiss}^{2}_{#1}(#2)}
\providecommand{\Fourier}{}
\renewcommand{\Fourier}[0]{\mathcal{F}}

\providecommand{\LpPlusCpct}{}
\renewcommand{\LpPlusCpct}[3]{L^{#1}_{\textup{c},#2}(#3)}

\def\Cts{\@ifnextchar_{\Cts@tief}{\Cts@tief_{}}}
    \def\Cts@tief_#1#2{\@ifnextchar\bgroup{\Cts@two_{#1}{#2}}{\Cts@one_{#1}{#2}}}
    \def\Cts@one_#1#2{C_{#1}\big(#2\big)}
    \def\Cts@two_#1#2#3{C_{#1}\big(#2,~#3\big)}

\def\BoundedOps#1{\@ifnextchar\bgroup{\BoundedOps@two{#1}}{\mathop{\BoundedOpsSymbol}(#1)}}
    \def\BoundedOps@two#1#2{\mathop{\BoundedOpsSymbol}(#1,#2)}
\def\BoundedOpsInv#1{\@ifnextchar\bgroup{\BoundedOps@two{#1}}{\mathop{\BoundedOpsSymbol}(#1)^{\times}}}
    \def\BoundedOpsInv@two#1#2{\mathop{\BoundedOpsSymbol}(#1,#2)^{\times}}

\def\id{\mathrm{\textit{id}}}

\def\restr#1{\vert_{#1}}
\def\without{\mathbin{\setminus}}

\def\eps{\varepsilon}
\let\altphi\phi
\let\altvarphi\varphi
    \def\phi{\altvarphi}
    \def\varphi{\altphi}

\def\quer#1{\overline{#1}}

\def\lim{\mathop{\ell\mathrm{im}}}
\def\supp{\mathop{\textup{supp}}}
\def\dim{\mathop{\textup{dim}}}
\def\dom{\mathop{\textup{dom}}}

\def\ran{\mathop{\textup{ran}}}

\def\Re{\mathop{\mathfrak{R}\mathrm{e}}}

\def\tinytopSOT{\text{\scriptsize\upshape \scshape sot}}

\@ifundefined{setcitestyle}{%
}{%
    \setcitestyle{numeric-comp,open={[},close={]}}
}

\allowdisplaybreaks 
\renewcommand{\arraystretch}{1}
\def\firstparagraph{\noindent}
\def\continueparagraph{\noindent}

    \generatenestedsecnumbering{arabic}{section}{subsection}
    \generatenestedsecnumbering{arabic}{subsection}{subsubsection}
    \def\theunitnamesection{\thesection}

    \def\sectionname{}

    \let\appendix@orig\appendix
    \def\appendix{%
        \appendix@orig%
        \let\boolinappendix\boolwahr
        \addcontentsline{toc}{part}{\appendixname}%
        \addtocontents{toc}{\protect\setcounter{tocdepth}{0}}
        \def\sectionname{Appendix}%
        \def\theunitnamesection{\Alph{section}}%
    }
    \def\notappendix{%
        \let\boolinappendix\boolfalse
        \addtocontents{toc}{\protect\setcounter{tocdepth}{1 }}
        \def\sectionname{}%
        \def\theunitnamesection{\arabic{section}}%
    }

    \def\@seccntformat#1{%
        \protect\textup{%
            \protect\@secnumfont
            \expandafter\protect\csname format#1\endcsname%
            \csname the#1\endcsname
            \expandafter\protect\csname format#1@pt\endcsname%
            \space
        }%
    }

    \def\formatsection@text{\centering\Large\scshape}
    \def\formatsection@pt{\secnumberingseppt}
    \def\section{\@startsection{section}{1}{\z@}{.7\linespacing\@plus\linespacing}{.5\linespacing}{\formatsection@text}}

    \def\formatsubsection@text{\flushleft\bfseries\scshape}
    \def\formatsubsection@pt{\subsecnumberingseppt}
    \def\subsection{\@startsection{subsection}{2}{\z@}{\z@}{\z@\hspace{1em}}{\formatsubsection@text}}

    \renewcommand{\paragraph}[1]{%
        {\bfseries #1}\:%
    }

\DefineFNsymbols*{customAlphabet}{abcdefghijklmnopqrstuvwxyz}
\setfnsymbol{customAlphabet}

\def\footnotemark[#1]{\text{\textsuperscript{\getrefnumber{#1}}}}

\def\rafootnotectr{20}
\providecommand{\incrftnotectr}{}
\renewcommand{\incrftnotectr}[1]{%
    \addtocounter{#1}{1}%
    \ifnum\value{#1}>\rafootnotectr\relax
        \setcounter{#1}{0}%
    \fi%
}
\providecommand{\footnoteref}{}
\renewcommand{\footnoteref}[1]{\protected@xdef\@thefnmark{\ref{#1}}\@footnotemark}
\let\@old@footnotetext\footnotetext
\def\footnotetext[#1]#2{%
    \incrftnotectr{footnote}%
    \@old@footnotetext[\value{footnote}]{\label{#1}#2}%
}

\def\kopfzeiledefault{
    \lhead[]{}
    \lhead[]{}
    \chead[]{}
    \rhead[]{}
    \lfoot[]{}
    \cfoot{\footnotesize\thepage}
    \rfoot[]{}
}

\def\aktuellesfont{\csname lmodern\endcsname}
\def\documentfont{%
    \gdef\aktuellesfont{\csname lmodern\endcsname}%
    \fontfamily{lmr}\fontseries{m}\selectfont%
    \renewcommand{\sfdefault}{phv}%
    \renewcommand{\ttdefault}{pcr}%
    \renewcommand{\rmdefault}{cmr}
    \renewcommand{\bfdefault}{bx}%
    \renewcommand{\itdefault}{it}%
    \renewcommand{\sldefault}{sl}%
    \renewcommand{\scdefault}{sc}%
    \renewcommand{\updefault}{n}%
}

\allowdisplaybreaks

\def\startdocumentlayoutoptions{
    \selectlanguage{british}
    \setlength{\parskip}{0.25\baselineskip}
    \setlength{\parindent}{2em}
    \kopfzeiledefault
    \documentfont
    \normalsize
}

\providecommand{\highlightTerm}{}
\renewcommand{\highlightTerm}[1]{\emph{#1}}
\providecommand{\highlightForReview}{}
\renewcommand{\highlightForReview}[1]{%
    \bgroup\relax%
    \color{blue}\relax%
    #1\relax%
    \egroup\relax%
}

\def\@adminfootnotes{%
    \let\@makefnmark\relax
    \let\@thefnmark\relax
    \ifx\@empty\@date\else%
        \@footnotetext{\@setdate}%
    \fi%
    \ifx\@empty\@subjclass\else%
        \@footnotetext{\@setsubjclass}%
    \fi
    \ifx\@empty\@keywords\else%
        \@footnotetext{\@setkeywords}%
    \fi
    \ifx\@empty\thankses\else%
        \@footnotetext{\def\par{\let\par\@par}\@setthanks}%
    \fi
}

\def\@settitle{\Large\bfseries\scshape\@title}

\def\@maketitle{%
  \normalfont\normalsize
  \@adminfootnotes
  \@mkboth{\@nx\shortauthors}{\@nx\shorttitle}%
  \global\topskip42\p@\relax
  {\centering\@settitle}
  \ifx\@empty\authors\else{\centering\small\@setauthors}\fi
  \ifx\@empty\@date\else{\vtop{\centering\small\@date\@@par}}\fi
  \ifx\@empty\@dedicatory%
  \else%
    \baselineskip\p@
    \vtop{\centering{\footnotesize\itshape\@dedicatory\@@par}%
    \global\dimen@i\prevdepth}\prevdepth\dimen@i%
  \fi
  \@setabstract
  \normalsize
  \if@titlepage
    \newpage
  \else
    \dimen@34\p@\advance\dimen@-\baselineskip
  \fi
}

\def\addresseshere{%
  \bgroup
  \setlength{\parindent}{0pt}
  \enddoc@text
  \egroup
  \let\enddoc@text\relax
}

\makeatother

\begin{document}
\startdocumentlayoutoptions

\thispagestyle{plain}



\def\abstractname{Abstract}
\begin{abstract}
    We consider characterisations of unitary dilations and approximations
    of irreversible classical dynamical systems on a Hilbert space.
    In the commutative case,
    building on the work in \cite{Chung1962exp}, one can express well known approximants
    (\exempli Hille- and Yosida-approximants)
    via expectations over certain stochastic processes.
    Using this, our first result characterises
    the simultaneous regular unitary dilatability of commuting families of $\Cnought$-semigroups
    via the dilatability of such approximants as well as via regular polynomial bounds.
    This extends the results in \cite{Dahya2023dilation} to the unbounded setting.
    We secondly consider characterisations of unitary and regular unitary dilations
    via two distinct functional calculi.
    Applying these tools to a large class of classical dynamical systems,
    these two notions of dilation exactly characterise
    when a system admits unitary approximations
    under certain distinct notions of weak convergence.
    This establishes a sharp topological distinction between the two notions of unitary dilations.
    Our results are applicable to commutative systems as well as
    non-commutative systems satisfying the \emph{canonical commutation relations} (CCR) in the Weyl form.
\end{abstract}



\title[Characterisations of dilations via approximants, expectations, and functional calculi]{Characterisations of dilations via approximants, expectations, and functional calculi}
\author{Raj Dahya}
\email{raj.dahya@web.de}
\address{Fakult\"at f\"ur Mathematik und Informatik\newline
Universit\"at Leipzig, Augustusplatz 10, D-04109 Leipzig, Germany}

\maketitle



\setcounternach{section}{1}



\section[Introduction]{Introduction}
\label{sec:introduction:sig:article-stochastic-raj-dahya}


\firstparagraph
Classical dynamical systems on Hilbert or Banach spaces,
which are in general irreversible,
can be studied in at least two natural ways
in terms of more ideal systems:
    via approximations
    and
    via embeddings into (or: \usesinglequotes{dilations} to) larger reversible systems.
Towards the former, see \exempli \cite{Hillephillips1957faAndSg,Butzer1967semiGrApproximationsBook,Chung1962exp,Krol2009}.
The study of the latter was in part inspired by a result from Halmos \cite{Halmos1950dilation},
and properly initiated by Sz.-Nagy and Foias in \cite{Nagy1953,Nagy1970}
with their work on unitary (power) dilations of contractions
and of $1$-parameter contractive $\Cnought$-semigroups
over Hilbert spaces.
Remaining in the commutative setting,
various results have been achieved for systems consisting of multiple operators
as well as multi-parameter $\Cnought$-semigroups
(see \exempli
    \cite{Ando1963pairContractions,Slocinski1974,Slocinski1982,Ptak1985,LeMerdy1996DilMultiParam,Shamovich2017dilationsMultiParam}).
For a good overview, see \exempli \cite{Averson2010DilationOverview,Shalit2021DilationBook}.

In the \textbf{first part} of this paper
(\S{}\ref{sec:stochastic:sig:article-stochastic-raj-dahya}--\ref{sec:first-results:sig:article-stochastic-raj-dahya}),
we shall first consider commutative systems of $\Cnought$-semigroups.
Note that a commuting family
    $\{U_{i}\}_{i=1}^{d}$
of unitary $\Cnought$-semigroups on a Hilbert space $\HilbertRaum$
can be uniquely extended to an \topSOT-continuous unitary representation
$U$ of $(\reals^{d},+,\zerovector)$ on $\HilbertRaum$
defined via
    $
        U(\mathbf{t})
        \colonequals
        (
            \prod_{\mathclap{i\in\supp(\mathbf{t}^{-})}}
            U_{i}(t_{i}^{-})
        )^{\ast}
        (
            \prod_{\mathclap{i\in\supp(\mathbf{t}^{+})}}
            U_{i}(t_{i}^{+})
        )
        $
for all $\mathbf{t} = (t_{i})_{i=1}^{d} \in \reals^{d}$,
where
    $t^{-}=\max\{-t,\,0\}$
    and
    $t^{+}=\max\{t,\,0\}$
for $t\in\reals$
(\cf Stone's theorem \cite[Theorem~I.4.7]{Goldstein1985semigroups}).
Bearing this in mind, a commuting family
    $\{T_{i}\}_{i=1}^{d}$
of $\Cnought$-semigroups on $\HilbertRaum$
is said to have a
    \highlightTerm{simultaneous regular unitary dilation}
if

    \noparskip
    $$
        \Big(
            \prod_{i=1}^{d}T(t_{i}^{-})
        \Big)^{\ast}
        \Big(
            \prod_{i=1}^{d}T(t_{i}^{+})
        \Big)
        =
        r^{\ast}\,U(\mathbf{t})\,r
    $$

\continueparagraph
holds for all $\mathbf{t} = (t_{i})_{i=1}^{d} \in \reals^{d}$,
for
    some $\topSOT$-continuous unitary representation $U$ of $(\reals^{d},+,0)$ on a Hilbert space $\HilbertRaum_{1}$
    and
    some isometry ${r\in\BoundedOps{\HilbertRaum}{\HilbertRaum_{1}}}$
(and in this case we shall refer to the data
    $(\HilbertRaum_{1},U,r)$
as the simultaneous regular unitary dilation).%
\footnote{%
    For Hilbert spaces, $\HilbertRaum,\HilbertRaum_{1},\HilbertRaum_{2}$
    the sets
        $\BoundedOps{\HilbertRaum}$
        and
        $\BoundedOps{\HilbertRaum_{1}}{\HilbertRaum_{2}}$
    denote the set of bounded linear operators on $\HilbertRaum$
    and the set of bounded linear operators from $\HilbertRaum_{1}$ to $\HilbertRaum_{2}$
    respectively.
}
A \highlightTerm{simultaneous unitary dilation}
is defined by the above condition restricted to $\mathbf{t} \in \realsNonNeg^{d}$.

For the cases $d\in\{1,2\}$ it was proved in
    \cite[Theorem~I.8.1]{Nagy1970},
    \cite{Slocinski1974},
    \cite[Theorem~2]{Slocinski1982},
    and
    \cite[Theorem~2.3]{Ptak1985},
that all contractive commuting families $\{T_{i}\}_{i=1}^{d}$
have a simultaneous unitary dilation.
In the case of $d=1$, these are in fact regular unitary dilations.
In the general case of $d \geq 1$,
the existence of simultaneous regular unitary dilations was fully characterised
in \cite[Theorem~3.2]{Ptak1985}
via a general condition which we may refer to as \emph{Brehmer positivity}.
Le~Merdy fully classified the existence of simultaneous unitary dilations
as well as dilations \emph{up to similarity}
via the complete boundedness of a certain functional calculus map
(%
    see \cite[Theorems~2.2~and~3.1]{LeMerdy1996DilMultiParam},
    which builds on \cite[Corollaries~4.9~and~4.13]{Pisier2001bookCBmaps}%
),
and successfully applied the latter to commuting families of bounded analytic $\Cnought$-semigroups.
Moreover, Shamovich and Vinnivok provided
sufficient embedding conditions on the generators
for the existence of simultaneous unitary dilations
(see \cite{Shamovich2017dilationsMultiParam}).
Recent results contribute to this history by providing
two further complete characterisations of the existence of simultaneous regular unitary dilations
for commuting families of $\Cnought$-semigroups
under the assumption of bounded generators
(see \cite[Theorems~1.1~and~1.4]{Dahya2023dilation}).
The first characterisation was achieved
via the notion of \emph{complete dissipativity}
(see \cite[Definition~2.8]{Dahya2023dilation}),
which is defined by the positivity of
certain combinatorial expressions involving the generators.
The second characterisation builds on the first
and characterises the existence of regular unitary dilations
via \emph{regular polynomial bounds}
(see \cite[Definition 6.2]{Dahya2023dilation}).
Furthermore, analogue to \cite{LeMerdy1996DilMultiParam},
it was shown that all commuting families of $\Cnought$-semigroups
with bounded generators have regular unitary dilations
up to certain natural modification
(see \cite[Corollary~1.2]{Dahya2023dilation}).

The latter reference left open the question,
whether the characterisation via regular polynomial bounds
could be extended to the unbounded setting
(see \cite[Remark~6.6]{Dahya2023dilation}).
Moreover, the characterisation via complete dissipativity involves a characterisation via approximants
which raises the question, whether for certain \emph{natural choices of approximants},
a commuting family of $\Cnought$-semigroups has a simultaneous regular unitary dilation
if and only if families of their approximants do.
The current paper shall answer both these questions positively
in the general setting without the boundedness assumption
(see \S{}\ref{sec:first-results:sig:article-stochastic-raj-dahya}).

In the \textbf{second part} of this paper
(\S{}\ref{sec:functional-calculus:sig:article-stochastic-raj-dahya}--\ref{sec:second-results:sig:article-stochastic-raj-dahya}),
we consider classical dynamical systems more broadly described by
\emph{homomorphisms} defined over topological monoids.
To motivate this, observe that there is
a natural \onetoone-correspondence between
commuting families
    $\{T_{i}\}_{i=1}^{d}$
of (bounded/contractive/unitary) $\Cnought$-semigroups
and (bounded/contractive/unitary) \topSOT-continuous homomorphisms
    $T$
between $(\realsNonNeg^{d},+,\zerovector)$ and spaces of operators
(\cf \cite[\S{}1]{Dahya2023dilation})
via
    $T(\mathbf{t}) = \prod_{i=1}^{d}T_{i}(t_{i})$
    and
    $T_{i}(t) = T_{i}(t) = T(0,0,\ldots,\underset{i}{t},\ldots,0)$
for $\mathbf{t}\in\realsNonNeg^{d}$, $t\in\realsNonNeg$,
$i\in\{1,2,\ldots,d\}$.
It is thus natural to consider homomorphisms defined over topological monoids.
If $G$ is a topological group and $M \subseteq G$ is a submonoid
and ${T \colon M\to\BoundedOps{\HilbertRaum}}$ is an \topSOT-continuous homomorphism from $M$
to bounded operators on a Hilbert space $\HilbertRaum$,
we define a \highlightTerm{unitary dilation} of $T$
to be a tuple $(\HilbertRaum_{1},U,r)$,
where
    $\HilbertRaum_{1}$ is a Hilbert space,
    $U$ is an $\topSOT$-continuous unitary representation $G$ on $\HilbertRaum_{1}$,
    and
    ${r\in\BoundedOps{\HilbertRaum}{\HilbertRaum_{1}}}$ is an isometry,
such that

    \noparskip
    $$
        T(x) = r^{\ast}\,U(x)\,r
    $$

\continueparagraph
holds for all $x \in M$.
We shall also consider monoids for which one can define
a \emph{positivity structure},
which consists of continuous maps ${\cdot^{+},\cdot^{-} \colon G \to M}$
(see \Cref{defn:positivity-structure-monoids:sig:article-stochastic-raj-dahya}).
Using this concept, we say that $(\HilbertRaum_{1},U,r)$
is a \highlightTerm{regular unitary dilation} of $T$,
if

    \noparskip
    $$
        (T(x^{-}))^{\ast}T(x^{+})
        =
        r^{\ast}\,U(x)\,r
    $$

\continueparagraph
holds for all $x \in G$.
Considering $(G,M) = (\reals^{d},\realsNonNeg^{d})$, $d\in\naturals$,
by the above mentioned \onetoone-correspondence
one can see that these concepts
agree with the definitions of
simultaneous (regular) unitary dilations of commuting families.

We present in \S{}\ref{sec:functional-calculus:sig:article-stochastic-raj-dahya}
characterisations
of unitary and regular unitary dilations
via the means of \emph{functional calculi},
inspired by Sz.-Nagy and Phillips/le~Merdy.
These tools enable us to characterise the existence of unitary approximations
for a broad class of classical systems
(see \S{}\ref{sec:second-results:sig:article-stochastic-raj-dahya}).

Before proceeding,
we recall the afore mentioned results from the bounded setting,
define our terminology
and state the main results of this paper.



\subsection[Characterisation via complete dissipativity]{Characterisation via complete dissipativity}
\label{sec:introduction:dissipativity:sig:article-stochastic-raj-dahya}

\firstparagraph
For a $\Cnought$-semigroup $T$ on a Hilbert space $\HilbertRaum$ with generator $A$,
if $T$ is contractive, then it has a regular unitary dilation
(\cf \cite[Theorem~I.8.1]{Nagy1970}).
And clearly, the latter necessarily requires $T$ to be contractive.
On the other hand, by the Lumer-Phillips form of the Hille-Yosida theorem,
$T$ is contractive if and only if $A$ is dissipative
(see \cite[Theorem~I.3.3]{Goldstein1985semigroups}).
Thus the dissipativity of the generator characterises
the regular unitary dilatability of a $\Cnought$-semigroup.

In the setting of commuting families of semigroups,
dissipativity can be generalised as follows:
For each $k\in\naturalsZero$,
the \highlightTerm{$k$\textsuperscript{th}-order dissipation operators}
are defined by

    \noparskip
    $$
        (-\tfrac{1}{2})^{\card{K}}
        \sum_{\isPartition{(C_{1},C_{2})}{K}}
            \Big(
                \prod_{i\in C_{1}}
                    A_{i}
            \Big)^{\ast}
            \prod_{j\in C_{2}}
                A_{j}
    $$

\continueparagraph
for $K \subseteq \{1,2,\ldots,d\}$ with $\card{K} = k$,
where $\isPartition{(C_{1},C_{2})}{K}$ denotes
that $C_{1},C_{2} \subseteq K$ form a partition of $K$.
We say that the generators $\{A_{i}\}_{i=1}^{d}$ are \highlightTerm{completely dissipative},
if the dissipation operators of all finite orders are positive
(\cf \cite[Definition~2.8]{Dahya2023dilation}).
This notion leads to a characterisation result obtained in
    \cite{Dahya2023dilation},
which for the purposes of this paper we restate as follows:

\begin{thm}[Characterisation via complete dissipativity]
\makelabel{thm:classification:dissipativity:sig:article-stochastic-raj-dahya}
    Let
        $d\in\naturals$,
        $\HilbertRaum$ be a Hilbert space,
        and
        $\{T_{i}\}_{i=1}^{d}$
        be a commuting family of $\Cnought$-semigroups on $\HilbertRaum$,
        with generators $\{A_{i}\}_{i=1}^{d}$.
    If the semigroups have bounded generators,
    then the following are equivalent:

    \begin{kompaktenum}{\bfseries (a)}
        \item\punktlabel{1}
            The family $\{T_{i}\}_{i=1}^{d}$ has a simultaneous regular unitary dilation.
        \item\punktlabel{2}
            The generators $\{A_{i}\}_{i=1}^{d}$ are completely dissipative.
        \item\punktlabel{3}
            There exists a net $(\{T^{(\alpha)}_{i}\}_{i=1}^{d})_{\alpha \in \Lambda}$
            consisting of
            a commuting families of $\Cnought$-semigroups on $\HilbertRaum$,
            which each have simultaneous regular unitary dilations,
            such that
                $$
                    \sup_{\mathbf{t} \in L}
                        \normLong{
                            \Big(
                                \prod_{i=1}^{d}T^{(\alpha)}_{i}(t_{i})
                                -
                                \prod_{i=1}^{d}T_{i}(t_{i})
                            \Big)
                            \xi
                        }
                    \underset{\alpha}{\longrightarrow} 0
                $$
            for all
                $\xi \in \HilbertRaum$
                and
                compact $L \subseteq \realsNonNeg^{d}$.
    \end{kompaktenum}

    \continueparagraph
    Furthermore, the notion of convergence in \punktcref{3}
    can be replaced by pointwise \topSOT-convergence.
    The implication
        \punktcref{3}{}\ensuremath{\implies}{}\punktcref{1}
    also holds without the boundedness assumption.
\end{thm}

See \cite[Theorem~1.1 and Remark~4.4]{Dahya2023dilation} for a proof.
Note that \punktcref{1}{}\ensuremath{\implies}{}\punktcref{3}
is trivial since one can choose a constant net.
However, if one choses particular approximants in advance,
it is not immediate that this implication holds.
This leads to a natural question (in the general setting),
which motivated the research in the present paper:

\begin{qstn}
\makelabel{qstn:classical-approximants:dilation:sig:article-stochastic-raj-dahya}
    Let $\{T_{i}\}_{i=1}^{d}$ be a commuting family of
    $\Cnought$-semigroups on a Hilbert space $\HilbertRaum$.
    For which choices of approximants
        $(\{T^{(\alpha)}_{i}\}_{i=1}^{d})_{\alpha\in\Lambda}$
    does it hold that
    the simultaneous regular unitary dilatability of $\{T_{i}\}_{i=1}^{d}$
    implies that of each family
        $\{T^{(\alpha)}_{i}\}_{i=1}^{d}$?
\end{qstn}



\subsection[Characterisation via polynomial bounds]{Characterisation via polynomial bounds}
\label{sec:introduction:polynomial:sig:article-stochastic-raj-dahya}

\firstparagraph
For commuting operators
    $\{S_{i}\}_{i=1}^{d} \subseteq \BoundedOps{\HilbertRaum}$,
\highlightTerm{regular polynomial evaluation}
is defined as the unique linear map

    \noparskip
    $$
        \complex[X_{1},X_{1}^{-1},X_{2},X_{2}^{-1},\ldots,X_{d},X_{d}^{-1}] \ni p
        \mapsto p(S_{1},S_{2},\ldots,S_{d}) \in \BoundedOps{\HilbertRaum}
    $$

\continueparagraph
satisfying

    \noparskip
    $$
        p(S_{1},S_{2},\ldots,S_{d})
        = \Big(\prod_{\mathclap{i\in\supp(\mathbf{n}^{-})}}S_{i}^{-n_{i}}\Big)^{\ast}
        \Big(\prod_{\mathclap{i\in\supp(\mathbf{n}^{+})}}S_{i}^{n_{i}}\Big)
    $$

\continueparagraph
for all monomials of the form
    $p = \prod_{i=1}^{d}X_{i}^{n_{i}}$
with $\mathbf{n} \in \integers^{d}$
(\cf \cite[Definition~6.1]{Dahya2023dilation}).
We say that $\{T_{i}\}_{i=1}^{d}$ satisfies
\highlightTerm{regular polynomial bounds}
if

    \noparskip
    $$
        \norm{
            p(T_{1}(t_{1}),T_{2}(t_{2}),\ldots,T_{d}(t_{d}))
        } \leq \sup_{\boldsymbol{\lambda} \in \Torus^{d}}
        \abs{
            p(\lambda_{1},\lambda_{2},\ldots,\lambda_{d})
        }
    $$

\continueparagraph
holds for all $\mathbf{t} = (t_{i})_{i=1}^{d} \in \realsNonNeg^{d}$,
where $\Torus$ is the unit circle in the complex plane.
Using these notions, a second characterisation is obtained
in \cite{Dahya2023dilation},
which for the purposes of this paper, we restate as follows:

\begin{thm}[Characterisation via polynomial bounds]
\makelabel{thm:classification-bounded:poly:sig:article-stochastic-raj-dahya}
    Let
        $d\in\naturals$,
        $\HilbertRaum$ be a Hilbert space,
        and
        $\{T_{i}\}_{i=1}^{d}$
        be a commuting family of $\Cnought$-semigroups on $\HilbertRaum$
        with generators $\{A_{i}\}_{i=1}^{d}$.
    If the semigroups have bounded generators,
    then the following are equivalent:

    \begin{kompaktenum}{\bfseries (a)}
        \item\punktlabel{1}
            The family $\{T_{i}\}_{i=1}^{d}$ has a simultaneous regular unitary dilation.
        \item\punktlabel{2}
            The family $\{T_{i}\}_{i=1}^{d}$ satisfies regular polynomial bounds.
        \item\punktlabel{3}
            For each ${K \subseteq \{1,2,\ldots,d\}}$
            and all ${\mathbf{t}=(t_{i})_{i=1}^{d} \in \realsNonNeg^{d}}$
            the operator
                $p_{K}(T_{1}(t_{1}),T_{2}(t_{2}),\ldots,T_{d}(t_{d}))$
            is positive, where

            \noparskip
            \begin{eqnarray*}
                p_{K}
                    \colonequals
                        \displaystyle
                        \sum_{\mathclap{
                            \isPartition{(C_{1},C_{2})}{K}
                        }}
                            \displaystyle
                            \prod_{i \in C_{1}}
                                (1 - X_{i}^{-1})
                            \cdot
                            \displaystyle
                            \prod_{j \in C_{2}}
                                (1 - X_{j}).\\
            \end{eqnarray*}

        \item\punktlabel{4}
            The generators $\{A_{i}\}_{i=1}^{d}$ are completely dissipative.
    \end{kompaktenum}

    \continueparagraph
    Furthermore, the implications
        \punktcref{1}{}\ensuremath{\implies}{}\punktcref{2}{}\ensuremath{\implies}{}\punktcref{3}
    hold without the boundedness assumption.
\end{thm}

Due to the inclusion of the intermediate step \punktcref{3},
we sketch the proof for the reader's convenience.

    \begin{proof}[of \Cref{thm:classification-bounded:poly:sig:article-stochastic-raj-dahya}, sketch]
        The equivalence of
            \punktcref{1}, \punktcref{2}, and \punktcref{4}
        is proved directly in \cite[Theorem~1.4]{Dahya2023dilation} in reliance upon \Cref{thm:classification:dissipativity:sig:article-stochastic-raj-dahya}.
        For \punktcref{1}{}\ensuremath{\implies}{}\punktcref{2}
        the assumption of bounded generators is not required
        (see \cite[Remark~6.6]{Dahya2023dilation}).

        Towards \punktcref{2}{}\ensuremath{\implies}{}\punktcref{3}:
        Suppose (without the boundedness assumption!)
        that the family $\{T_{i}\}_{i=1}^{d}$ satisfies regular polynomial bounds.
        Let
            $K \subseteq \{1,2,\ldots,d\}$
            and
            $\mathbf{t} = (t_{i})_{i=1}^{d} \in \realsNonNeg^{d}$
        be arbitrary.
        Using binomial expansions, one can see that the regular polynomial $p_{K}$
        can be expressed as
            $
                p_{K}
                = \prod_{i=1}^{d}(2 - X_{i} - X_{i}^{-1})
            $,
        and thus
            $
                p_{K}(\lambda_{1},\lambda_{2},\ldots,\lambda_{d})
                = \prod_{i=1}^{d}(2 - \lambda_{i} - \lambda_{i}^{-1})
                = \prod_{i=1}^{d}(2 - 2\Re\lambda_{i})
                \in [0, 4^{d}]
            $
        for all $\boldsymbol{\lambda} = (\lambda_{1},\lambda_{2},\ldots,\lambda_{d}) \in \Torus^{d}$.
        It follows that
            $
                \sup_{\boldsymbol{\lambda}}
                    \abs{1 - \alpha p_{K}(\lambda_{1},\lambda_{2},\ldots,\lambda_{d})}
                \leq 1
            $
        for sufficiently small $\alpha\in\realsPos$.
        Since the family of semigroups satisfies regular polynomial bounds,
        it follows that
            $
                \norm{\onematrix - \alpha p_{K}(T_{1}(t_{1}), T_{2}(t_{2}), \ldots, T_{d}(t_{d}))}
                \leq 1
            $
        for sufficiently small $\alpha\in\realsPos$.
        As argued in the proof of \cite[Theorem~1.4]{Dahya2023dilation}, this implies that
            $p_{K}(T_{1}(t_{1}), T_{2}(t_{2}), \ldots, T_{d}(t_{d}))$
        is a positive operator.

        Finally, the implication \punktcref{3}{}\ensuremath{\implies}{}\punktcref{4}
        (under the assumption of bounded generators)
        is proved exactly as in \cite[Theorem~1.4]{Dahya2023dilation}.
        By taking limits of these positive operators appropriately scaled,
        one obtains that the generators are completely dissipative.
    \end{proof}

This result raises the following question
(\cf \cite[Remark~6.6]{Dahya2023dilation}),
which along with \Cref{qstn:classical-approximants:dilation:sig:article-stochastic-raj-dahya}
also motivated the research in the current paper.

\begin{qstn}
\makelabel{qstn:polynomial:sig:article-stochastic-raj-dahya}
    Does the equivalence of
        \eqcref{it:1:thm:classification-bounded:poly:sig:article-stochastic-raj-dahya}
        and
        \eqcref{it:2:thm:classification-bounded:poly:sig:article-stochastic-raj-dahya}
        in
        \Cref{thm:classification-bounded:poly:sig:article-stochastic-raj-dahya}
    continue to hold without the assumption of bounded generators?
\end{qstn}



\subsection[Characterisation via expectation-approximants]{Characterisation via expectation-approximants}
\label{sec:introduction:results:sig:article-stochastic-raj-dahya}

\firstparagraph
We shall demonstrate a further characterisation related
to the above two results
without the assumption of bounded generators.
The key idea is to make the implication
    \eqcref{it:1:thm:classification:dissipativity:sig:article-stochastic-raj-dahya}{}\ensuremath{\implies}{}\eqcref{it:3:thm:classification:dissipativity:sig:article-stochastic-raj-dahya}
    of
    \Cref{thm:classification:dissipativity:sig:article-stochastic-raj-dahya}
work by considering suitable canonical choices for approximants
(\cf \Cref{qstn:classical-approximants:dilation:sig:article-stochastic-raj-dahya}).

Let $T$ be an arbitrary $\Cnought$-semigroup on a Banach space $\BanachRaum$
with generator
    ${A \colon \opDomain{A} \subseteq \BanachRaum \to \BanachRaum}$.
We now consider two concrete nets of approximants
of the form
    $(
        T^{(\lambda)}
        = (
            e^{tA^{(\lambda)}}
        )_{\mathbf{t} \in \realsNonNeg^{d}}
    )_{\lambda \in I}$
for some $I \subseteq \realsPos$ directed by increasing values of $\lambda$.

If $\omega_{0}(T)\in[-\infty,\:\infty)$ is the \highlightTerm{growth bound} for $T$
(\cf \cite[Proposition~I.5.5 and Definition~I.5.6]{EngelNagel2000semigroupTextBook},
    \cite[Lemma~I.2.12]{Goldstein1985semigroups}%
),
the \highlightTerm{$\lambda$\textsuperscript{th}-Yosida-approximant}
is defined by
    $
        T^{(\lambda)} = (e^{tA^{(\lambda)}})_{t\in\realsNonNeg}
    $
for each $\lambda \in (\omega_{0}(T),\infty)$,
where

    \noparskip
    \begin{equation}
    \label{eq:yosida-approx-generator:sig:article-stochastic-raj-dahya}
        A^{(\lambda)}
        \colonequals
            \lambda A \opResolvent{A}{\lambda}
        = \lambda^{2} \opResolvent{A}{\lambda}
            - \lambda\onematrix
    \end{equation}

\continueparagraph
is a bounded operator,
where
    $
        \opResolvent{A}{\lambda}
        = (\lambda \onematrix - A)^{-1}
    $
denotes the resolvent operator.
The Yosida-approximants satisfy

    \noparskip
    $$
        \sup_{t \in L}\norm{(T^{(\lambda)}(t) - T(t))\xi} \longrightarrow 0
    $$

\continueparagraph
as ${(\omega_{0}(T),\infty) \ni \lambda \longrightarrow \infty}$
for all
    $\xi \in \BanachRaum$
    and
    compact $L \subseteq \realsNonNeg$.
Furthermore, if $T$ is contractive, then each of the Yosida-approximants are contractive.
(For a proof of these classical results, see \exempli
    \cite[Theorem~G.4.3]{HytNeervanMarkLutz2016bookVol2},
    \cite[Theorem~II.3.5, pp.~73--74]{EngelNagel2000semigroupTextBook},
    \cite[(12.3.4), p.~361]{Hillephillips1957faAndSg}.%
)

Now, if a family
    $\{T_{i}\}_{i=1}^{d}$
of commuting $\Cnought$-semigroups
has a simultaneous regular unitary dilation,
then in particular each of the $T_{i}$ must be contractive,
and thus the family
    $\{T^{(\lambda_{i})}_{i}\}_{i=1}^{d}$
of Yosida-approximants consists of contractive $\Cnought$-semigroups
for each $\boldsymbol{\lambda} = (\lambda_{i})_{i=1}^{d} \in \realsPos^{d}$
(\cf the subsequent paragraph below (3.8) in \cite{EngelNagel2000semigroupTextBook}).
Furthermore, the commutativity of the $T_{i}$ implies the commutativity of the resolvents,%
\footnote{%
    see \exempli \cite[Theorem~1]{Abdelaziz1983commutingSemigroups},
    where this is proved under slightly more general assumptions.
    One can also directly verify this
    by relying on the Laplace integral representation of resolvents.
}
which by \eqref{eq:yosida-approx-generator:sig:article-stochastic-raj-dahya}
implies the commutativity of the generators
    $\{A^{(\lambda_{i})}_{i}\}_{i=1}^{d}$,
which in turn implies that
    $\{T^{(\lambda_{i})}\}_{i=1}^{d}$
is a commuting family.

Another appropriate approximation is constructed in \emph{Hille's first exponential formula}.
For $\lambda\in\realsPos$ call $T^{(\lambda)} = (e^{tA^{(\lambda)}})_{t\in\realsNonNeg}$
the \highlightTerm{$\lambda$\textsuperscript{th}-Hille-approximant} for $T$,
where

    \noparskip
    \begin{equation}
    \label{eq:hille-approx-generator:sig:article-stochastic-raj-dahya}
        A^{(\lambda)}
            \colonequals
                \lambda (T(\tfrac{1}{\lambda}) - \onematrix).
    \end{equation}

\continueparagraph
is a bounded operator.
Then again it holds that
    ${T^{(\lambda)}(t) \longrightarrow T(t)}$
for ${\realsPos \ni \lambda \longrightarrow \infty}$
    \wrt the \topSOT-topology
    uniformly in $t$
    on compact subsets of $\realsNonNeg$
(see \cite[\S{}1.2 and Theorem~1.2.2]{Butzer1967semiGrApproximationsBook}).
It is easy to verify that
    $\norm{T^{(\lambda)}(t)} \leq e^{-\lambda t}e^{\lambda t \norm{T(\tfrac{1}{\lambda})}}$.
Thus the Hille-approximants of contractive $\Cnought$-semigroups
are themselves contractive $\Cnought$-semigroups
with bounded generators.
Moreover, if $\{T_{i}\}_{i=1}^{d}$ is a commuting family of (contractive) $\Cnought$-semigroups,
then for each $\boldsymbol{\lambda} = (\lambda_{i})_{i=1}^{d} \in \realsPos^{d}$,
by \eqcref{eq:hille-approx-generator:sig:article-stochastic-raj-dahya},
the generators
    $\{A^{(\lambda_{i})}\}_{i=1}^{d}$
clearly commute,
so that the family of Hille-approximants
    $\{T^{(\lambda_{i})}\}_{i=1}^{d}$
is a commuting family of (contractive) $\Cnought$-semigroups.

It turns out that these classically defined approximants in semigroup theory
can be naturally generalised to a class of approximants,
which we shall call \highlightTerm{expectation-approximants}
(see \Cref{defn:expectation-approximants:sig:article-stochastic-raj-dahya} below).
We now state the first main result of this paper:

\begin{schattierteboxdunn}[backgroundcolor=leer,nobreak=false]
\begin{thm}
\makelabel{thm:classification-unbounded:poly:sig:article-stochastic-raj-dahya}
    Let
        $d\in\naturals$,
        $\HilbertRaum$ be a Hilbert space,
        and
        $\{T_{i}\}_{i=1}^{d}$
        be a commuting family of contractive $\Cnought$-semigroups on $\HilbertRaum$.
    Further let
        $(T^{(\alpha)}_{i})_{\alpha \in \Lambda_{i}}$
    be a net of expectation-approximants for $T_{i}$
    (\exempli Hille- or Yosida-approximants)
    for each $i\in\{1,2,\ldots,d\}$.
    Then the following are equivalent:

    \begin{kompaktenum}{\bfseries (a)}
        \item\punktlabel{1}
            The family $\{T_{i}\}_{i=1}^{d}$ has a simultaneous regular unitary dilation.
        \item\punktlabel{2}
            The family $\{T_{i}\}_{i=1}^{d}$ satisfies regular polynomial bounds.
        \item\punktlabel{3}
            For each ${K \subseteq \{1,2,\ldots,d\}}$
            and all ${\mathbf{t}=(t_{i})_{i=1}^{d} \in \realsNonNeg^{d}}$
            the operator
                $p_{K}(T_{1}(t_{1}),T_{2}(t_{2}),\ldots,T_{d}(t_{d}))$
            is positive, where

            \noparskip
            \begin{eqnarray*}
                p_{K}(X_{1},X_{2},\ldots,X_{d})
                    \colonequals
                        \displaystyle
                        \sum_{\mathclap{
                            \isPartition{(C_{1},C_{2})}{K}
                        }}
                            \displaystyle
                            \prod_{i \in C_{1}}
                                (1 - X_{i}^{-1})
                            \cdot
                            \displaystyle
                            \prod_{j \in C_{2}}
                                (1 - X_{j}).\\
            \end{eqnarray*}

        \item\punktlabel{4}
            The family
                $\{T^{(\alpha_{i})}_{i}\}_{i=1}^{d}$
            of approximants
            has a simultaneous regular unitary dilation
            for each $\boldsymbol{\alpha} = (\alpha_{i})_{i=1}^{d} \in \prod_{i=1}^{d}\Lambda_{i}$.
    \end{kompaktenum}

    \nvraum{1}

\end{thm}
\end{schattierteboxdunn}

By \Cref{thm:classification-unbounded:poly:sig:article-stochastic-raj-dahya},
we shall be able to positively address
    \Cref{%
        qstn:classical-approximants:dilation:sig:article-stochastic-raj-dahya,%
        qstn:polynomial:sig:article-stochastic-raj-dahya%
    }.
Note that the main implication to prove is
    \punktcref{3}{}\ensuremath{\implies}{}\punktcref{4},
as the others have essentially been proved in
\cite{Dahya2023dilation}
(\cf the discussion in \S{}\ref{sec:introduction:dissipativity:sig:article-stochastic-raj-dahya}--\ref{sec:introduction:polynomial:sig:article-stochastic-raj-dahya}).
Now, in the case of bounded generators the spectral-theoretic concept of complete dissipativity
was crucial for the characterisations.
In the unbounded setting, it is unclear how to extend the notion
of dissipation operators and thus of complete dissipativity.
Nonetheless, one possibility and its limitations shall be discussed
(\cf \Cref{rem:complete-dissipativity:sig:article-stochastic-raj-dahya}).
Instead, we make use of stochastic methods,
building on the results in \cite{Chung1962exp}.
Using suitable stochastic processes
and expectations computed strongly via Bochner-integrals,
which we lay out in \S{}\ref{sec:stochastic:sig:article-stochastic-raj-dahya},
we show that expectation-approximants
can be expressed in terms of their original semigroups.
This provides a key ingredient to prove \punktcref{3}{}\ensuremath{\implies}{}\punktcref{4}.
This shall all be covered in
    \S{}\ref{sec:stochastic:sig:article-stochastic-raj-dahya}--\ref{sec:first-results:sig:article-stochastic-raj-dahya}.



\subsection[Special conditions on topological monoids]{Special conditions on topological monoids}
\label{sec:introduction:monoids:sig:article-stochastic-raj-dahya}

\firstparagraph
In the second part of this paper, we concentrate on classical dynamical systems
defined more broadly over topological monoids.
We are particularly interested in topological monoids $M$
which occur as (closed or at least measurable) submonoids of (ideally locally compact) topological groups $G$.

\begin{e.g.}
    Simple examples of locally compact topological groups and closed submonoids thereof
    include any discrete group and submonoid thereof,
    \exempli $(\integers^{d},+,\zerovector)$ and $\naturalsZero^{d}\subseteq\integers^{d}$ for $d\in\naturalsPos$.
    We also consider
        $\realsNonNeg^{d}$
    as a closed submonoid of the locally compact topological group $(\reals^{d},+,\zerovector)$.
\end{e.g.}

\begin{e.g.}
\makelabel{e.g.:heisenberg:sig:article-stochastic-raj-dahya}
    As a non-discrete and non-commutative example,
    consider the Heisenberg group
    $\Heisenberg_{d}$ of order $2d+1$, $d\in\naturalsPos$,
    which can be represented topologically as
    ${\Heisenberg_{d}=\reals^{d}\times\reals^{d}\times\reals}$
    and algebraically via the operation%
    \footnote{%
        There are various different presentations of the Heisenberg group in the literature
        (\cf \cite[\S{}1]{Coburn1999artHeisneberg}, \cite[\S{}10.1]{Deitmar2014bookHarmonicAn}, \cite[\S{}6.7.4]{Folland2015bookHarmonicAnalysis}).
        We choose this particular form for convenience.
    }

        \noparskip
        $$
            (\mathbf{x},\mathbf{p},E)
            (\mathbf{x}^{\prime},\mathbf{p}^{\prime},E^{\prime})
            = \Big(
                \mathbf{x} + \mathbf{x}^{\prime},
                \mathbf{p} + \mathbf{p}^{\prime},
                E + E^{\prime}
                + \frac{1}{2}(
                    \brkt{\mathbf{p}}{\mathbf{x}^{\prime}}
                    - \brkt{\mathbf{p}^{\prime}}{\mathbf{x}}
                )
            \Big)
        $$

    \continueparagraph
    for
        $
            (\mathbf{x},\mathbf{p},E),
            (\mathbf{x}^{\prime},\mathbf{p}^{\prime},E^{\prime})
            \in \Heisenberg_{d}
        $.
    The identity element of $\Heisenberg_{d}$ is clearly $(\zerovector,\zerovector,0)$
    and inverse of $g\in\Heisenberg_{d}$ is given by $g^{-1}=-g$,
    if we view $g$ as a vector in $\reals^{2d+1}$.
    One can readily verify that

        \noparskip
        $$
            \Heisenberg^{+}_{d}
            \colonequals
            \{
                (\mathbf{x},\mathbf{p},E)
                \mid
                \mathbf{x},\mathbf{p}\in\realsNonNeg^{d},
                E\in\reals
            \}
            \subseteq \Heisenberg_{d}
        $$

    \continueparagraph
    is a closed submonoid.
    As a further non-commutative example,
    consider for some antisymmetric matrix $C \in M_{d}(\reals)$
    the closed subgroup
        $
            \Heisenberg_{d,C}
            \colonequals
            \{
                (\mathbf{x},C\mathbf{x},E)
                \mid
                \mathbf{x}\in\reals^{d},
                E\in\reals
            \}
            \subseteq \Heisenberg_{d}
        $
    and the closed submonoid
        $
            \Heisenberg^{+}_{d,C}
            \colonequals
            \{
                (\mathbf{x},C\mathbf{x},E)
                \mid
                \mathbf{x}\in\realsNonNeg^{d},
                E\in\reals
            \}
            \subseteq \Heisenberg_{d,C}
        $.
    Without loss of generality, we can replace the above representations of
        $\Heisenberg_{d,C}$ and $\Heisenberg^{+}_{d,C}$
    by
        $\reals^{d}\times\reals$ and $\realsNonNeg^{d}\times\reals$
    respectively.
    Observe that the group operation satisfies

        \noparskip
        $$
            (\mathbf{x},E)
            (\mathbf{x}^{\prime},E)
            = \Big(
                \mathbf{x} + \mathbf{x}^{\prime},
                E + E^{\prime}
                + \frac{1}{2}(
                    \brkt{C\mathbf{x}}{\mathbf{x}^{\prime}}
                    - \brkt{C\mathbf{x}^{\prime}}{\mathbf{x}}
                )
            \Big)
            = \Big(
                \mathbf{x} + \mathbf{x}^{\prime},
                E + E^{\prime} + \brkt{C\mathbf{x}}{\mathbf{x}^{\prime}}
            \Big)
        $$

    \continueparagraph
    for $(\mathbf{x},E),(\mathbf{x}^{\prime},E^{\prime})\in\Heisenberg_{d,C}$.
    One can thus think of the $i$\textsuperscript{th} and $j$\textsuperscript{th}
    $x$-co-ordinates of elements of $\Heisenberg_{d,C}$
    as being \emph{correlated} with $C_{ij}$ encoding this correlation.
\end{e.g.}

\begin{e.g.}[Non-commuting families, Weyl Form of CCR]
\makelabel{e.g.:non-commuting-family-heisenberg-d-C:sig:article-stochastic-raj-dahya}
    Continuing on from \Cref{e.g.:heisenberg:sig:article-stochastic-raj-dahya},
    let
        $d\in\naturals$,
        $\HilbertRaum$ be a Hilbert space,
        and
        $C\in M_{d}(\reals)$ be an antisymmetric matrix,
    \idest
        $C = D - D^{T}$, where $D \in M_{d}(\reals)$
    is a strictly upper triangular matrix.
    One use of classical dynamical systems defined over $\Heisenberg^{+}_{d,C}$
    is as follows is as follows:
    Consider an arbitrary \topSOT-continuous homomorphism,
        ${T \colon \Heisenberg^{+}_{d,C} \to \BoundedOps{\HilbertRaum}}$.
    Set
        ${
            T_{i} \colonequals (T(t\mathbf{e}_{i},0))_{t\in\realsNonNeg}
        }$
        for each $i\in\{1,2,\ldots,d\}$,
        and
        ${
            U \colonequals (T(\zerovector, E))_{E\in\reals}
        }$,
    whereby each
        $\mathbf{e}_{i} \colonequals (0,0,\ldots,\underset{i}{1},\ldots,0)$
    denotes the canonical $i$\textsuperscript{th} basis vector of $\reals^{d}$.
    One can then readily verify that
        $\{T_{i}\}_{i=1}^{d}$ is a family of $\Cnought$-semigroups on $\HilbertRaum$
        and
        $U$ is a (not necessarily unitary) $\topSOT$-continuous representation
        of $(\reals,+,0)$ on $\HilbertRaum$
        which commutes with each of the $T_{i}$.
    Furthermore the relations

        \noparskip
        \begin{equation}
        \label{eq:commutation-relations:sig:article-stochastic-raj-dahya}
            T_{j}(t) T_{i}(s)
                = U(2stC_{ij}) T_{i}(s) T_{j}(t)
        \end{equation}

    \continueparagraph
    hold for all $i,j\in\{1,2,\ldots,d\}$,
    which is a slight generalisation of
    the semigroup version of the \emph{canonical commutation relations} (CCR)
    in the Weyl form.
    Such systems are of interest in quantum mechanics
    (see \exempli \cite[\S{}5.2.1.2]{Bratteli1997bookQSM}).
    Conversely, one may consider an arbitrary family
        $\{T_{i}\}_{i=1}^{d}$
    of $\Cnought$-semigroups on $\HilbertRaum$
    and an arbitrary $\topSOT$-continuous representation
    $U$ of $(\reals,+,0)$ on $\HilbertRaum$,
    which commutes with each of the $T_{i}$,
    and such that the relations in
    \eqcref{eq:commutation-relations:sig:article-stochastic-raj-dahya} hold.%
    \footnote{%
        For example consider
            $\HilbertRaum = L^{2}(\realsNonNeg^{m})$, $m\in\naturals$.
        Let
            $\mathbf{u}^{(i)}\in\realsNonNeg^{m}$
        and either
            $\lambda\in\iunit\reals$
            and
            $\boldsymbol{\alpha}^{(i)}\in\reals^{m}$,
        or
            $\lambda\in\complex$ with $\Re\lambda < 0$
            and
            $\boldsymbol{\alpha}^{(i)}\in\realsNonNeg^{m}$
        for each $i\in\{1,2,\ldots,d\}$.
        Set
            $U \colonequals (e^{\lambda t}\onematrix)_{t\in\reals}$
        and
            $
                (T_{i}(t)f)(\mathbf{x})
                \colonequals
                e^{
                    \lambda t
                    \brkt{\boldsymbol{\alpha}^{(i)}}{\mathbf{x}}
                }
                f(\mathbf{x} + t\mathbf{u}^{(i)})
            $
        for
            $t\in\realsNonNeg$,
            $f\in L^{2}(\realsNonNeg^{m})$,
            $\mathbf{x}\in\realsNonNeg^{m}$,
            $i\in\{1,2,\ldots,d\}$.
        Then \eqcref{eq:commutation-relations:sig:article-stochastic-raj-dahya} holds with
            $
                C \colonequals (
                    \frac{1}{2}(
                        \brkt{\boldsymbol{\alpha}^{(i)}}{\mathbf{u}^{(j)}}
                        -
                        \brkt{\boldsymbol{\alpha}^{(j)}}{\mathbf{u}^{(i)}}
                    )
                )_{i,j=1}^{d}
                \in M_{d}(\reals)
            $.
    }
    Then, defining
        ${T \colon \Heisenberg^{+}_{d,C} \to \BoundedOps{\HilbertRaum}}$
    via

        \noparskip
        $$
            T(\mathbf{x},E)
            \colonequals
            U(E + \brkt{D\mathbf{x}}{\mathbf{x}})
            \prod_{i=1}^{d}
                T_{i}(x_{i})
        $$

    \continueparagraph
    for each $\mathbf{x}\in\realsNonNeg^{d}$, $E\in\reals$,
    one can verify that $T$ is an \topSOT-continuous homomorphism.
    These two constructions constitute a \onetoone-correspondence
    (this is left as an exercise to the reader)
    between families satisfying
    \eqcref{eq:commutation-relations:sig:article-stochastic-raj-dahya}
    and \topSOT-continuous homomorphisms defined over
    $\Heisenberg^{+}_{d,C}$.
\end{e.g.}

For submonoids of topological groups,
there are natural conditions which will aid us when investigating properties (\viz dilation)
of classical dynamical systems defined over them.
The following definition is due to Mueller \cite[\S{}2]{Mueller1965positiveInIdentity}:%
\footnote{%
    \cf also \cite[Appendix~A]{Dahya2022complmetrproblem},
    where a similar, but stronger, property called \emph{positivity in the identity} is defined.
}

\begin{defn}
\makelabel{defn:monoids:e-jointedness:sig:article-stochastic-raj-dahya}
    Let $G$ be a locally compact topological group
    and $M \subseteq G$ be a measurable submonoid.
    Fix a Haar-measure $\lambda_{G}$ on $G$.
    Say that $M$ is \highlightTerm{$e$-joint}
    if $\lambda_{G}(U \cap M) > 0$
    for all open neighbourhoods $U \subseteq G$ of the identity $e \in G$.%
    \footnote{%
        Note that being a null-set is independent of the particular choice of Haar-measure.
    }
\end{defn}

\begin{e.g.}
\makelabel{e.g.:monoids:e-joint:sig:article-stochastic-raj-dahya}
    Let $d\in\naturals$ and $p\in\mathbb{P}$ be a prime number.

    \begin{kompaktenum}{(i)}
    \item\punktlabel{1}
        Consider $(G,M) = (\reals^{d},\realsNonNeg^{d})$.
        For an open neighbourhood $U \subseteq G$ of $\zerovector$,
        there exists $\eps>0$, such that $U \supseteq (-\eps,\:\eps)^{d}$.
        Thus $M \cap U \supseteq (0,\:\eps)^{d}$,
        so $M \cap U$ is non-null.
        Hence $M$ is an $e$-joint submonoid.
    \item\punktlabel{2}
        Consider
            $
                (G,M)
                = (\Heisenberg_{d},\Heisenberg^{+}_{d})
                = (
                    \reals^{d}\times\reals^{d}\times\reals,
                    \realsNonNeg^{d}\times\realsNonNeg^{d}\times\reals
                )
            $.
        As in \punktcref{1},
        for any open neighbourhood $U \subseteq G$ of $(\zerovector,\zerovector,0)$,
        one has
            $M \cap U \supseteq (0,\:\eps)^{d} \times (0,\:\eps)^{d} \times (-\eps,\:\eps)$
        for some $\eps > 0$,
        so $M \cap U$ is non-null.
        Hence $M$ is an $e$-joint submonoid.
    \item
        Let $C\in M_{d}(\reals)$ be an antisymmetric matrix.
        Similar to \punktcref{2} one can show that
        $\Heisenberg^{+}_{d,C}$ is a (closed) $e$-joint submonoid
        of $\Heisenberg_{d,C}$.
    \item
        The submonoid $\integers_{p} \without \{0\}$ of non-zero $p$-adic integers
        within the locally compact multiplicative group
            $(\rationals_{p} \without \{0\},\cdot,1)$
        of non-zero $p$-adic numbers is clearly $e$-joint,
        since it is a clopen subset.
    \item
        Let $G_{i}$ be a locally compact topological group
        and $M_{i} \subseteq G_{i}$ an $e$-joint measurable submonoid
        for $i\in\{1,2,\ldots,d\}$.
        By simple computations with product measures,
        one can readily verify that
            $\prod_{i=1}^{d}M_{i}$ is a measurable $e$-joint submonoid in $\prod_{i=1}^{d}G_{i}$
        (\cf \cite[Proposition~A.7]{Dahya2022complmetrproblem}).
    \end{kompaktenum}

    \nvraum{1}

\end{e.g.}

\begin{defn}
\makelabel{defn:positivity-structure-monoids:sig:article-stochastic-raj-dahya}
    Let $(G,\cdot,e)$ be a (not necessarly locally compact!) topological group
    and $M \subseteq G$ a submonoid.
    If ${\cdot^{+} \colon G \to M}$ is a continuous function which satisfies

    \begin{kompaktenum}{(i)}
        \item\punktlabel{ax:identity}
            $e^{+} = e$ for the identity element $e \in G$;
        \item\punktlabel{ax:idempotent}
            $x^{++} = x^{+}$ for all $x\in G$,
            \idest $\cdot^{+}$ is idempotent;
            and
        \item\punktlabel{ax:representation}
            $(x^{-})^{-1}x^{+} = x$
            for all $x\in G$,
            where $x^{-} \colonequals (x^{-1})^{+}$,
    \end{kompaktenum}

    \continueparagraph
    then we call $(G,M,\cdot^{+})$
    a \highlightTerm{positivity structure}.
\end{defn}

\begin{e.g.}
\makelabel{e.g.:monoids:positivity:sig:article-stochastic-raj-dahya}
    Let $d\in\naturals$ and $C \in M_{d}(\reals)$ be an antisymmetric matrix.
    The pairs $(G,M)$ of topological groups and submonoids:
        $(\reals^{d},\realsNonNeg^{d})$,
        $(\Heisenberg_{d},\Heisenberg^{+}_{d})$,
        $(\Heisenberg_{d,C},\Heisenberg^{+}_{d,C})$
    admit natural positivity structures.
    Furthermore, if
        $(G_{i},M_{i},\cdot^{+_{i}})$
    are positivity structures,
    then the product
        $(\prod_{i=1}^{d}G_{i},\prod_{i=1}^{d}M_{i})$
    admits a positivity via the pointwise definition.
    The constructions are presented in
        \Cref{table:examples:monoids:positivity:sig:article-stochastic-raj-dahya}
    and left to the reader to verify.
\end{e.g.}

\begin{table}[!htb]
    \begin{tabular}[t]{|p{0.15\textwidth}|p{0.15\textwidth}|p{0.4\textwidth}|}
        \hline
            Group $G$
            &Submonoid $M$
            &Description of ${\cdot^{+} \colon G\to M}$\\
        \hline
        \hline
            $\prod_{i=1}^{d}G_{i}$
                &$\prod_{i=1}^{d}M_{i}$
                &$\mathbf{x} \mapsto (x_{i}^{+_{i}})_{i=1}^{d}$\\
            $(\reals,+,0)$
                &$\realsNonNeg$
                &$t \mapsto \max\{t,0\}$\\
            $(\reals^{d},+,\zerovector)$
                &$\realsNonNeg^{d}$
                    &$\mathbf{t} \mapsto (t_{i}^{+})_{i=1}^{d}$\\
            $\Heisenberg_{d}$
                &$\Heisenberg^{+}_{d}$
                &$(\mathbf{x},\mathbf{p},E)
                    \mapsto
                    (
                        \mathbf{x}^{+},
                        \mathbf{p}^{+},
                        E^{+}
                    )
                $\\
            $\Heisenberg_{d,C}$
                &$\Heisenberg^{+}_{d,C}$
                &$(\mathbf{x},E)
                    \mapsto
                    (
                        \mathbf{x}^{+},
                        E^{+}
                    )
                $\\
        \hline
    \end{tabular}
    \caption{
        Examples of positivity structures $(G,M,\cdot^{+})$.
        Here $d\in\naturals$, $C \in M_{d}(\reals)$ is an antisymmetric matrix,
        and $(G_{i},M_{i},\cdot^{+_{i}})$
        are positivity structures for $i\in\{1,2,\ldots,d\}$.
    }
    \label{table:examples:monoids:positivity:sig:article-stochastic-raj-dahya}
\end{table}

Note that for the subset
    $G^{+} \colonequals \{x^{+}\mid g\in G\} \subseteq M$,
it is neither required that  $G^{+} = M$
nor even that $G^{+}$ be closed under the group operation.
In the case of $\reals^{d}$, this happens to be the case,
but in the case of the Heisenberg group,
neither of these additional properties holds.
We now state some basic facts about positivity structures.

\begin{prop}
    Let $(G,M,\cdot^{+})$ be a positivity structure.
    Then
        $(G^{-})^{-1} \cap G^{+} = (G^{+})^{-1} \cap G^{+} = \{e\}$.
\end{prop}

    \begin{proof}
        First observe that
            $
                (G^{-})^{-1} \cap G^{+}
                = ((G^{-1})^{+})^{-1} \cap G^{+}
                = (G^{+})^{-1} \cap G^{+}
            $.
        Thus it suffices to prove that $(G^{+})^{-1} \cap G^{+} = \{e\}$.
        Since $e^{+} = e$, the $\supseteq$-inclusion holds.
        Towards the $\subseteq$-inclusion,
        let $x \in (G^{+})^{-1} \cap G^{+}$ be arbitrary.
        Then $x = y^{+} = (z^{+})^{-1}$
        for some $y, z \in G$.
        By the idempotence axiom,
        one has that
            $
                x^{+}
                = y^{++}
                = y^{+}
                = x
            $
        and
            $
                x^{-}
                = (x^{-1})^{+}
                = z^{++}
                = z^{+}
                = x^{-1}
            $.
        By axiom \eqcref{it:ax:representation:defn:positivity-structure-monoids:sig:article-stochastic-raj-dahya},
        it follows that
            $
                x = x^{+}
                \eqcrefoverset{it:ax:representation:defn:positivity-structure-monoids:sig:article-stochastic-raj-dahya}{=}
                    x^{-}x
                = x^{-1}x
                = e
            $.
    \end{proof}

\begin{prop}
\makelabel{prop:positivity-structure-basic:sig:article-stochastic-raj-dahya}
    Let $(G,\cdot,e)$ be a topological group
    and $M \subseteq G$ a submonoid.
    Let ${\cdot^{+} \colon G \to M}$
    be an arbitrary map satisfying
    axiom \eqcref{it:ax:representation:defn:positivity-structure-monoids:sig:article-stochastic-raj-dahya}
    of \Cref{defn:positivity-structure-monoids:sig:article-stochastic-raj-dahya}.
    Then axiom
        \eqcref{it:ax:idempotent:defn:positivity-structure-monoids:sig:article-stochastic-raj-dahya}
    is equivalent to
    the condition that $x^{+-}=e$ for all $x\in G$.
\end{prop}

    \begin{proof}
        Let $x \in G$ be arbitrary.
        Then by
            axiom \eqcref{it:ax:representation:defn:positivity-structure-monoids:sig:article-stochastic-raj-dahya}
        one has
            $x^{+} = (x^{+-})^{-1}x^{++}$.
        It follows that
            $x^{++} = x^{+}$
            if and only if
            $x^{+-} = e$.
    \end{proof}



\subsection[Characterisation of unitary approximations]{Characterisation of unitary approximations}
\label{sec:introduction:applications:sig:article-stochastic-raj-dahya}

\firstparagraph
Letting $M$ be a topological monoid and $\HilbertRaum$ a Hilbert space,
a classical dynamical system modelled by a homomorphism,%
\footnote{%
    \idest $T(e) = \onematrix$ and $T(xy) = T(x)T(y)$ for $x,y\in M$.
}
    ${T \colon M \to \BoundedOps{\HilbertRaum}}$,
can be thought of as \emph{reversible}
if it is unitary valued.
As this need not be the case, the question arises whether and in what sense
one can \emph{approximate} $T$ via unitary systems.

In
    \Cref{thm:classification:dissipativity:sig:article-stochastic-raj-dahya}
    and
    \Cref{thm:classification-unbounded:poly:sig:article-stochastic-raj-dahya}
a strong notion of convergence is used for the approximants.
In the literature a similar weak notion of convergence for $\Cnought$-semigroups is studied
(\cf \exempli
    \cite{Krol2009,Eisner2008kato,Eisnersereny2009catThmStableSemigroups,Dahya2022weakproblem,Dahya2022complmetrproblem}
    and
    \cite[\S{}III.6]{Eisner2010buchStableOpAndSemigroups}
    in the case of Hilbert spaces%
).
These notions are defined as follows:
Let
    $d\in\naturals$,
    $\BanachRaum$ be a Banach space,
    $\{T_{i}\}_{i=1}^{d}$
    be a commuting family of $\Cnought$-semigroups on $\BanachRaum$,
    and
    $(\{T^{(\alpha)}_{i}\}_{i=1}^{d})_{\alpha\in\Lambda}$
    be a net of commuting families of $\Cnought$-semigroups on $\BanachRaum$.
We say that
    $(\{T^{(\alpha)}_{i}\}_{i=1}^{d})_{\alpha\in\Lambda}$
    converges to $\{T_{i}\}_{i=1}^{d}$
    \highlightTerm{\wrt the \topSOT-toplogy}
    (\respectively \highlightTerm{\wrt the \topWOT-toplogy})
    \highlightTerm{uniformly on compact subsets of $\realsNonNeg^{d}$},
if
    for all $\xi\in\BanachRaum$
    (\respectively for all $\xi\in\BanachRaum$ and $\eta\in\BanachRaum^{\prime}$)
    and all compact $L\subseteq\realsNonNeg^{d}$
it holds that
        ${
            \sup_{\mathbf{t}\in L}
                \norm{
                    (
                        \prod_{i=1}^{d}
                            T^{(\alpha)}_{i}(t_{i})
                        -
                        \prod_{i=1}^{d}
                            T_{i}(t_{i})
                    )
                    \:\xi
                }
            \underset{\alpha}{\longrightarrow}
            0
        }$
(\respectively
    ${
        \sup_{\mathbf{t}\in L}
            \abs{\brkt{
                (
                    \prod_{i=1}^{d}
                        T^{(\alpha)}_{i}(t_{i})
                    -
                    \prod_{i=1}^{d}
                        T_{i}(t_{i})
                )
                \:\xi
            }{\eta}}
        \underset{\alpha}{\longrightarrow}
        0
    }$
).
For classical systems on Hilbert spaces,
using the concepts in \S{}\ref{sec:introduction:monoids:sig:article-stochastic-raj-dahya}
we may generalise the weak notion of convergence to classical dynamical systems
in general in the following natural manner:

\begin{defn}[Weak topologies]
\makelabel{defn:weak-convergence-over-exact-uniform-ptwise:sig:article-stochastic-raj-dahya}
    Let
        $G$ be a topological group,
        $M \subseteq G$ a submonoid,
        $\HilbertRaum$ a Hilbert space,
        ${T \colon M \to \BoundedOps{\HilbertRaum}}$
        an \topSOT-continuous homomorphism
        and
        ${(T^{(\alpha)} \colon M \to \BoundedOps{\HilbertRaum})_{\alpha\in\Lambda}}$
        a net of \topSOT-continuous homomorphisms.
    We say that
        $(T^{(\alpha)})_{\alpha\in\Lambda}$
        converges to
        $T$

    \begin{kompaktenum}{(i)}
    \item\punktlabel{1}
        \highlightTerm{exactly weakly},
        if for each
            $\xi,\eta\in\HilbertRaum$
        there exists an index $\alpha_{0}\in\Lambda$
        such that for all
            $x \in M$
            and
            $\alpha \geq \alpha_{0}$
            $$
                \brkt{T^{(\alpha)}(x)\xi}{\eta}
                =
                \brkt{T(x)\xi}{\eta};
            $$
    \item\punktlabel{2}
        \highlightTerm{uniformly weakly},
        if for all
            $\xi,\eta\in\HilbertRaum$
            and
            compact $L \subseteq M$
            $$
                \sup_{x \in L}
                    \abs{
                        \brkt{T^{(\alpha)}(x)\xi}{\eta}
                        -
                        \brkt{T(x)\xi}{\eta}
                    }
                \underset{\alpha}{\longrightarrow}
                0;
            $$
    \item\punktlabel{3}
        \highlightTerm{pointwise weakly},
        if for all
            $\xi,\eta\in\HilbertRaum$
            and
            $x \in M$
            $$
                \abs{
                    \brkt{T^{(\alpha)}(x)\xi}{\eta}
                    -
                    \brkt{T(x)\xi}{\eta}
                }
                \underset{\alpha}{\longrightarrow}
                0.
            $$
    \end{kompaktenum}

    \continueparagraph
    If each $T^{(\alpha)} = U^{(\alpha)}\restr{M}$,
    where ${U^{(\alpha)} \colon G \to \BoundedOps{\HilbertRaum}}$
    is an \topSOT-continuous unitary representation of $G$ on $\HilbertRaum$,
    we say that $T$ has an
        \highlightTerm{exact}
        (%
            \respectively
            \highlightTerm{uniform}
            \respectively
            \highlightTerm{pointwise}%
        )
        \highlightTerm{weak unitary approximation},
    if \punktcref{1}
        (%
            \respectively
            \punktcref{2}
            \respectively
            \punktcref{3}%
        )
    holds.
\end{defn}

\begin{defn}[Regular weak topologies]
\makelabel{defn:regular-weak-convergence-over-exact-uniform-ptwise:sig:article-stochastic-raj-dahya}
    Let
        $(G,M,\cdot^{+})$ be a positivity structure,
        $\HilbertRaum$ a Hilbert space,
        ${T \colon M \to \BoundedOps{\HilbertRaum}}$
        an \topSOT-continuous homomorphism
        and
        ${(T^{(\alpha)} \colon M \to \BoundedOps{\HilbertRaum})_{\alpha\in\Lambda}}$
        a net of \topSOT-continuous homomorphisms.
    We say that
        $(T^{(\alpha)})_{\alpha\in\Lambda}$
        converges to
        $T$

    \begin{kompaktenum}{(i)}
    \item\punktlabel{1}
        \highlightTerm{exactly regularly weakly},
        if for each
            $\xi,\eta\in\HilbertRaum$
        there exists an index $\alpha_{0}\in\Lambda$
        such that for all
            $x \in G$
            and
            $\alpha \geq \alpha_{0}$
            $$
                \brkt{T^{(\alpha)}(x^{-})^{\ast}T^{(\alpha)}(x^{+})\xi}{\eta}
                =
                \brkt{T(x^{-})^{\ast}T(x^{+})\xi}{\eta};
            $$
    \item\punktlabel{2}
        \highlightTerm{uniformly regularly weakly},
        if for all
            $\xi,\eta\in\HilbertRaum$
            and
            compact $L \subseteq G$
            $$
                \sup_{x \in L}
                    \abs{
                        \brkt{T^{(\alpha)}(x^{-})^{\ast}T^{(\alpha)}(x^{+})\xi}{\eta}
                        -
                        \brkt{T(x^{-})^{\ast}T(x^{+})\xi}{\eta}
                    }
                \underset{\alpha}{\longrightarrow}
                0;
            $$
    \item\punktlabel{3}
        \highlightTerm{pointwise regularly weakly},
        if for all
            $\xi,\eta\in\HilbertRaum$
            and
            $x \in G$
            $$
                \abs{
                    \brkt{T^{(\alpha)}(x^{-})^{\ast}T^{(\alpha)}(x^{+})\xi}{\eta}
                    -
                    \brkt{T(x^{-})^{\ast}T(x^{+})\xi}{\eta}
                }
                \underset{\alpha}{\longrightarrow}
                0.
            $$
    \end{kompaktenum}

    \continueparagraph
    If each $T^{(\alpha)} = U^{(\alpha)}\restr{M}$,
    where ${U^{(\alpha)} \colon G \to \BoundedOps{\HilbertRaum}}$
    is an \topSOT-continuous unitary representation of $G$ on $\HilbertRaum$,
    we say that $T$ has an
        \highlightTerm{exact}
        (%
            \respectively
            \highlightTerm{uniform}
            \respectively
            \highlightTerm{pointwise}%
        )
        \highlightTerm{regular weak unitary approximation},
    if \punktcref{1}
        (%
            \respectively
            \punktcref{2}
            \respectively
            \punktcref{3}%
        )
    holds.
\end{defn}

Clearly, exact (regular) weak approximations
are uniform (regular) weak approximations,
which in turn are pointwise (regular) weak approximations.
In the case of
    $(G,M)=(\reals^{d},\realsNonNeg^{d})$,
if $d=1$, then
each \emph{regular}-notion of convergence clearly coincides
with the corresponding notion without the \emph{regular} prefix.
In this special case, the following result is known:

\begin{thm}[Kr\'ol, 2009]
\makelabel{thm:unitary-approx:one-param:krol:sig:article-stochastic-raj-dahya}
    Let $T$ be a contractive $\Cnought$-semigroup
    on an infinite dimensional Hilbert space $\HilbertRaum$.%
    \footref{ft:approx-thm:hilbert-space-dim:sig:article-stochastic-raj-dahya}
    Then $T$ has a uniform weak unitary approximation.
\end{thm}

For a proof see \cite[Theorem~2.1 and Remark~2.3]{Krol2009}.
The following results add to this picture.

\begin{schattierteboxdunn}[backgroundcolor=leer,nobreak=true]
\begin{thm}[Characterisation of weak unitary approximations]
\makelabel{thm:unitary-approx:weak:sig:article-stochastic-raj-dahya}
    Let
        $G$ be a locally compact topological group
        and
        $M \subseteq G$ be an $e$-joint closed submonoid.
    Further let
        $\HilbertRaum$ be a Hilbert space
        and
        ${T \colon M \to \BoundedOps{\HilbertRaum}}$
        be an \topSOT-continuous homomorphism.
    If $G$ contains a dense subset $D \subseteq G$
    with $\card{D} \leq \dim(\HilbertRaum)$,%
    \footref{ft:approx-thm:hilbert-space-dim:sig:article-stochastic-raj-dahya}
    then the following are equivalent:

    \begin{kompaktenum}{\bfseries (a)}
        \item\punktlabel{1}
            The classical system $T$ has a unitary dilation.
        \item\punktlabel{2}
            The classical system $T$ has an exact weak unitary approximation.
        \item\punktlabel{3}
            The classical system $T$ has a uniform weak unitary approximation.
    \end{kompaktenum}

    \continueparagraph
    Without the above assumption on the dimension of $\HilbertRaum$,
        \punktcref{2}{}\ensuremath{\implies}{}\punktcref{3}{}\ensuremath{\implies}{}\punktcref{1}
    hold.
\end{thm}
\end{schattierteboxdunn}

\begin{schattierteboxdunn}[backgroundcolor=leer,nobreak=true]
\begin{thm}[Characterisation of regular weak unitary approximations]
\makelabel{thm:unitary-approx:regular-weak:sig:article-stochastic-raj-dahya}
    Let
        $(G,M,\cdot^{+})$ be a positivity structure,
    where
        $G$ is a topological group
        and
        $M \subseteq G$ is a submonoid.%
        \footref{ft:approx-thm:G-not-locally-compact:sig:article-stochastic-raj-dahya}
    Further let
        $\HilbertRaum$ be a Hilbert space
        and
        ${T \colon M \to \BoundedOps{\HilbertRaum}}$
        be an \topSOT-continuous homomorphism.
    If $\HilbertRaum$ is infinite dimensional
    and $G$ contains a dense subset $D \subseteq G$
    with $\card{D} \leq \dim(\HilbertRaum)$,%
        \footref{ft:approx-thm:hilbert-space-dim:sig:article-stochastic-raj-dahya}
    then the following are equivalent:

    \begin{kompaktenum}{\bfseries (a)}
        \item\punktlabel{1}
            The classical system $T$ has a regular unitary dilation.
        \item\punktlabel{2}
            The classical system $T$ has an exact regular weak unitary approximation.
        \item\punktlabel{3}
            The classical system $T$ has a uniform regular weak unitary approximation.
        \item\punktlabel{4}
            The classical system $T$ has a pointwise regular weak unitary approximation.
    \end{kompaktenum}

    \continueparagraph
    Without the above assumption on the dimension of $\HilbertRaum$,
        \punktcref{2}{}\ensuremath{\implies}{}\punktcref{3}{}\ensuremath{\implies}{}\punktcref{4}{}\ensuremath{\implies}{}\punktcref{1}
    hold.
\end{thm}
\end{schattierteboxdunn}

\footnotetext[ft:approx-thm:hilbert-space-dim:sig:article-stochastic-raj-dahya]{
    In the case of separable topological groups,
        \exempli $G=\reals^{d}$, $d\in\naturals$,
    this holds as soon as $\HilbertRaum$ is infinite dimensional.
}

\footnotetext[ft:approx-thm:G-not-locally-compact:sig:article-stochastic-raj-dahya]{
    Note that we neither require $G$ to be locally compact
    nor $M$ to be a measurable subset in this theorem!
}

Now, to prove \Cref{thm:unitary-approx:one-param:krol:sig:article-stochastic-raj-dahya},
Kr\'ol constructs unitary approximants directly from a regular unitary dilation of $T$
and further questions whether a proof is possible without reliance upon dilations
(\cf \cite[Remark~2.2]{Krol2009}).
\Cref{%
    thm:unitary-approx:weak:sig:article-stochastic-raj-dahya,%
    thm:unitary-approx:regular-weak:sig:article-stochastic-raj-dahya,%
}
partially address this by showing
that the existence of simultaneous (regular) dilations is necessary.
Thus any dilation-free proof of
    \Cref{thm:unitary-approx:one-param:krol:sig:article-stochastic-raj-dahya}
might necessarily involve some characterisation of (regular) unitary dilations.

Our results furthermore provide a sharp distinction between the two notions of dilation.
By
    \Cref{%
        e.g.:monoids:e-joint:sig:article-stochastic-raj-dahya,%
        e.g.:monoids:positivity:sig:article-stochastic-raj-dahya,%
    }
these results are immediately applicable
to commuting families of $\Cnought$-semigroups
as well as non-commuting families satisfying the \emph{canonical commutation relations} (CCR) in the Weyl form
(see \Cref{e.g.:non-commuting-family-heisenberg-d-C:sig:article-stochastic-raj-dahya}).

As an example, applying these characterisations to
    $(G,M) = (\reals^{d},\realsNonNeg^{d})$
for any $d \geq 2$ and infinite dimensional Hilbert space $\HilbertRaum$,
we shall construct commuting families of $d$ contractive $\Cnought$-semigroups on $\HilbertRaum$
that admit no regular weak unitary approximations
(see \Cref{cor:counter-examples-unitary-approx:sig:article-stochastic-raj-dahya}).



\subsection[Notation]{Notation}
\label{sec:introduction:notation:sig:article-stochastic-raj-dahya}

\firstparagraph
In this paper we fix the following notation and conventions:

\begin{kompaktitem}
    \item
        We write
            $\naturalsPos = \{1,2,\ldots\}$,
            $\naturalsZero = \{0,1,2,\ldots\}$,
            $\realsNonNeg = \{r\in\reals \mid r\geq 0\}$,
            $\realsPos = \{r\in\reals \mid r > 0\}$,
            and
            $\Torus = \{z\in\complex \mid \abs{z} = 1\}$
            (unit circle in the complex plane).
        To distinguish from indices $i$ we use $\iunit$ for the imaginary unit $\sqrt{-1}$.
    \item
        We write elements of product spaces in bold
        and denote their components in light face fonts with appropriate indices,
        \exempli the $i$\textsuperscript{th} components of
            $\mathbf{t} \in \realsNonNeg^{n}$
            and
            $\boldsymbol{\alpha} \in \prod_{i=1}^{n}\Lambda_{i}$
            are denoted
                $t_{i}$ and $\alpha_{i}$
            respectively.
    \item
        In some instances we shall work with concrete constructions of approximants of
        $\Cnought$-semigroups or families thereof
        (\exempli the Hille- and Yosida-approximants).
        In such cases we use
            $\lambda \in \realsPos$
            and
            $\boldsymbol{\lambda} \in \realsPos^{d}$
        to index the approximants.
        In others instances we work with the generalisation: \highlightTerm{expectation-approximants}
        (defined below in \Cref{defn:expectation-approximants:sig:article-stochastic-raj-dahya}).
        To indicate the abstract setting,
            $\alpha \in \Lambda$
            and
            $\boldsymbol{\alpha} \in \prod_{i=1}^{d}\Lambda_{i}$
        are used to index the approximants.
    \item
        For Banach spaces $\BanachRaum,\BanachRaum_{1},\BanachRaum_{2}$,
        the set
            $\BoundedOps{\BanachRaum}$ and $\BoundedOps{\BanachRaum_{1}}{\BanachRaum_{2}}$
        denote the set of bounded linear operators on $\BanachRaum$
        and the set of bounded linear operators from $\BanachRaum_{1}$ to $\BanachRaum_{2}$
        respectively.
    \item
        For bounded operators $S$ over a Banach space,
        $S^{\prime}$ denotes the adjoint (dual) operator.
        For bounded operators $S$ over a Hilbert space,
        $S^{\ast}$ denotes the Hermitian adjoint.
    \item
        For an (unbounded) linear operator ${A \colon \dom(A)\subseteq\BanachRaum\to\BanachRaum}$
        on a Banach space and $\lambda\in\complex$ in the resolvent set of $A$,
            $
                \opResolvent{A}{\lambda}
                = (\lambda\onematrix - A)^{-1}
            $
        denotes the resolvent operator.
    \item
        For a measure (or probability) space $(X,\Sigma,\mu)$,
        a measurable space $(Y,S)$,
        and
        a measurable function ${f \colon X \to Y}$,
        the push-forward measure $\PushForward{f}{\mu}$,
        which we denote $\mu_{f}$,
        is the measure (\respectively probability measure) on $(Y, S)$
        defined by
            $\PushForward{f}{\mu}[B] = \mu[f^{-1}(B)]$
        for all measurable $B \subseteq Y$.
    \item
        For a probability distribution $\Gamma$ over a set $X$
        we write ${\theta \distributedAs \Gamma}$
        to denote that $\theta$ is an $X$-valued random variable (\randomvar)
        with distribution $\Gamma$.
    \item
        For $t\in\reals$ the distribution $\delta_{t}$ denotes the point distribution concentrated in $t$.
    \item
        For $\lambda\in\realsPos$
        we denote with $\DistExp{\lambda}$
        the exponential distribution with \highlightTerm{rate} $\lambda$.
        For ${\theta \distributedAs \DistExp{\lambda}}$
        it holds that
            ${
                \Prob_{\theta}[B] = \int_{B}\lambda e^{\lambda s}\:\dee s
            }$
        for all measurable $B \subseteq \realsNonNeg$.
    \item
        For $\lambda\in\realsNonNeg$ and $c\in\realsPos$
        we denote with $\DistPoissScale{\lambda}{c}$
        the distribution of a Poisson distributed \randomvar scaled by $c$.
        For $\theta \distributedAs \DistPoissScale{\lambda}{c}$
        it holds that
            ${
                \Prob[\theta=cn] = \frac{\lambda^{n}}{n!}e^{-\lambda}
            }$
        for $n\in\naturals$ with the convention that $0^{0} \colonequals 0$.
        In particular, $\DistPoissScale{\lambda}{c}=\delta_{0}$ if $\lambda=0$.
    \item
        For $\lambda\in\realsPos$ and $t\in\realsNonNeg$,
        we denote with $\DistPoissAux{t}{\lambda}$
        an \highlightTerm{auxiliary Poisson distribution}
        with \highlightTerm{rate} $\lambda$
        over a \highlightTerm{time duration} $t$
        (defined below in \S{}\ref{sec:stochastic:distributions:sig:article-stochastic-raj-dahya},
        see also \cite[\S{}4]{Chung1962exp}).
    \item
        For a (unital) {($C^{\ast}$-)algebra} $\mathcal{A}$,
        $M_{n}(\mathcal{A})$ denotes the (unital) matrix {($C^{\ast}$-)algebra}
        of $\mathcal{A}$-valued $n \times n$-matrices
        for $n\in\naturals$.
        If $\mathcal{A}$ is an algebra of operators over some Hilbert space $\HilbertRaum$,
        then the elements of $M_{n}(\mathcal{A})$
        are viewed as operators acting on $\bigoplus_{i=1}^{n}\HilbertRaum$,
        and we have
            $
                \brkt{(a_{ij})_{ij}\oplus_{i=1}^{d}\xi_{i}}{\oplus_{i=1}^{d}\eta_{i}}
                = \sum_{ij}
                    \brkt{a_{ij}\xi_{j}}{\eta_{i}}
            $
        for \usesinglequotes{matrices} $(a_{ij})_{ij} \in M_{n}(\mathcal{A})$
        and vectors
            $
                \oplus_{i=1}^{n}\xi_{i},\oplus_{i=1}^{n}\eta_{i}
                \in \bigoplus_{i=1}^{n}\HilbertRaum
            $.
    \item
        A map ${\Psi \colon \mathcal{A} \to \mathcal{B}}$
        between (subalgebras of) $C^{\ast}$-algebras is called
        \emph{completely bounded} if
            $
                \normCb{\Psi}
                \colonequals
                    \sup_{n\in\naturalsPos}\norm{\Psi \otimes \id_{M_{n}}}
                < \infty
            $,
        and \emph{completely positive} if
            $
                \Psi \otimes \id_{M_{n}}
            $
            is positive for all $n\in\naturals$.
        Here
            ${\Psi \otimes \id_{M_{n}} \colon M_{n}(\mathcal{A}) \to M_{n}(\BoundedOps{\HilbertRaum})}$
        is defined by
            $(\Psi \otimes \id_{M_{n}})((a_{ij})_{ij}) = (\Psi(a_{ij}))_{ij} \in M_{n}(\BoundedOps{\HilbertRaum})$
        for each
            $(a_{ij})_{ij} \in M_{n}(\mathcal{A})$
            and
            $n\in\naturals$
        (%
            see
            \cite[Chapter~1]{Paulsen1986bookCBmapsAndDilations},
            \cite[Chapter~3]{Pisier2001bookCBmaps}%
        ).
\end{kompaktitem}

For a Banach space $\BanachRaum$,
a measure space $(X,\Sigma,\mu)$
and an operator-valued function ${T \colon X \to \BoundedOps{\BanachRaum}}$,
for which ${T(\cdot)\xi \colon X \to \BanachRaum}$ is
\highlightTerm{strongly measurable} for each $\xi\in\BanachRaum$,
the integral ${\sotInt_{X} T\:\dee\mu}$, when it exists,
denotes the unique bounded operator $\tilde{T}\in\BoundedOps{\BanachRaum}$
that satisfies
    ${\tilde{T}\xi = \int_{X} T(\cdot)\xi\:\dee\mu}$
for all $\xi\in\BanachRaum$,
where $\int_{X} T(\cdot)\xi\:\dee\mu$
is computed as a Bochner-integral.
This holds, for example, if $X$ is a locally compact Polish space
(\exempli $\realsNonNeg^{d}$ for some $d\in\naturals$),
and $T$ is contractive and \topSOT-continuous
(\exempli a product of contractive $\Cnought$-semigroups).
If $(\Omega,\Sigma,\Prob)$ is a probability space
and ${\tau \colon \Omega \to X}$ is an $X$-valued \randomvar (\idest a measurable function),
we refer to
    $
        \Expected[T(\theta)]
        = \sotInt_{\omega\in\Omega} T(\theta(\omega))\:\Prob(\dee\omega)
        = \sotInt_{t \in X} T(t)\:\Prob_{\theta}(\dee t)
    $,
when it exists,
as the \highlightTerm{expectation} (computed strongly via Bochner-integrals).
If $X$ is a locally compact Polish space
and $T$ is a contractive \topSOT-continuous function
then the expectation exists.

The existence and properties of Bochner-integrals
(including linearity, convexity and triangle inequalities, Fubini's theorem for products, \etcetera)
as well as the validity of various computations with Bochner-integrals and expectations
used in the rest of this paper
can be found in or readily derived from the literature.
We refer the reader in particular to
    \cite[\S{}3.7, Theorems~3.7.4--6, and Theorems~3.7.12--13]{Hillephillips1957faAndSg},
    \cite[\S{}II.2, Theorem~2, and Theorem~4]{DiestelUhl1977VectorMeasures},
    \cite[\S{}C.1--4]{EngelNagel2000semigroupTextBook},
    and
    \cite[\S{}1.1.c--\S{}1.2.a]{HytNeervanMarkLutz2016bookVol1}.
For example, using Fubini's theorem one can derive that
    $
        \prod_{i=1}^{n}\Expected[T_{i}(\theta_{i})]
        = \Expected[\prod_{i=1}^{n}T_{i}(\theta_{i})]
    $
for $n\in\naturals$,
independent $\realsNonNeg$-valued \randomvar's
    $\theta_{1},\theta_{2},\ldots,\theta_{n}$
and $\Cnought$-semigroups $T_{1},T_{2},\ldots,T_{n}$ on a Banach space $\BanachRaum$
which are uniformly bounded on the essential ranges of
$\theta_{1},\theta_{2},\ldots,\theta_{n}$ respectively
(\exempli contractive semigroups).
We shall take advantage of this computation throughout.
Further fundamental applications of Bochner-integrals
in the context of $\Cnought$-semigroups
can be found \exempli in \cite{Chung1962exp,Dunfordschwartz1988BookLinOpI,Reissig2005abstractresolvent}.




\section[Stochastic presentation of classical approximants]{Stochastic presentation of classical approximants}
\label{sec:stochastic:sig:article-stochastic-raj-dahya}

\firstparagraph
In this section we provide standalone results for $\Cnought$-semigroups
over Banach spaces and then for (commuting) families.
We assume basic knowledge of stochastic processes
as well as Poisson and exponential distributions.
We shall also rely on the theory of Bochner-integrals,
in particular those that occur in the integral representations of
powers of resolvents of unbounded generators of $\Cnought$-semigroups.


\subsection[Semigroups of distributions]{Semigroups of distributions}
\label{sec:stochastic:distributions:sig:article-stochastic-raj-dahya}

\firstparagraph
A family $(\Gamma(t))_{t\in\realsNonNeg}$ of parameterised distributions
of $\reals$-valued random variables (\randomvar's)
shall be called a \highlightTerm{continuous semigroup of distributions} if

\begin{kompaktenum}{(i)}
    \item
        $\Gamma(0)$ is the point distribution concentrated in $0$ (\idest $\delta_{0}$);
    \item
        $\theta_{1} + \theta_{2} \distributedAs \Gamma(s + t)$
        for $s,t\in\realsNonNeg$
        and any independent \randomvar's
            $\theta_{1} \distributedAs \Gamma(s)$
            and
            $\theta_{2} \distributedAs \Gamma(t)$;
        and
    \item
        ${\Gamma(t) \longrightarrow \Gamma(0) = \delta_{0}}$
        weakly for ${\realsPos \ni t \longrightarrow 0}$,
        \idest
        letting $\theta_{t} \distributedAs \Gamma(t)$ for $t\in\realsNonNeg$,
        it holds that
            ${\Expected[f(\theta_{t})] \longrightarrow \Expected[f(\theta_{0})] = f(0)}$
        for all bounded continuous functions $f\in\Cts{\realsNonNeg}$
\end{kompaktenum}

\continueparagraph
(\cf
    \cite[Definition~14.46 and Example~17.7]{Klenke2008probTheory}
    and
    \cite[\S{}17.E]{Kechris1995BookDST}%
).
Simple examples of this include
    $(\delta_{t})_{t\in\realsNonNeg}$,
    $(\DistPoissScale{\lambda t}{c})_{t\in\realsNonNeg}$
    for $c\in\reals$, $\lambda\in\realsPos$,
    and
    $(\mathcal{N}(\mu t,\sigma^{2}t))_{t\in\realsNonNeg}$
    (the normal distributions)
    for $\mu\in\reals$, $\sigma\in\realsNonNeg$
    with the convention that
    $\mathcal{N}(0,0)$ denotes the point distribution $\delta_{0}$
(\cf \cite[Corollary~15.13]{Klenke2008probTheory}).

We now construct a further family of parameterised distributions.
For $\lambda \in (0,\:\infty)$ we construct a random distribution
via two independent homogeneous point Poisson processes (PPP) as follows
(depicted in \Cref{fig:example-auxiliary-poisson:sig:article-stochastic-raj-dahya}):

\begin{kompaktenum}{\bfseries 1.}
    \item
        Let
            $
                \tilde{\tau}_{0}, \tilde{\tau}_{1}, \tilde{\tau}_{2},\ldots,
                \tau_{0}, \tau_{1}, \tau_{2},\ldots
                \distributedAs
                \DistExp{\lambda}
            $
        be independent identically distributed (\iid) \randomvar's.
    \item
        For each $n\in\naturalsZero$ set
            $
                \tau_{<n} \colonequals \sum_{i=0}^{n-1}\tau_{i}
            $,
        with the convention that the empty sum is just
        the \randomvar equal to $0$ \almostSurely.
        As a sum of $n$ exponentially distributed \randomvar's,
        we have
            $
                \Prob[\tau_{<n} \in B]
                = \int_{s \in B}
                    \lambda\frac{(\lambda s)^{n-1}}{(n-1)!}e^{-\lambda s}
                \:\dee s
            $
        for all measurable $B \subseteq \realsNonNeg$,
        provided $n \geq 1$
        (see \exempli \cite[Theorem~15.12 and Corollary~15.13~(ii)]{Klenke2008probTheory}).
    \item
        For each $t\in\realsNonNeg$, let
            $
                N_{t}
                \colonequals \sup_{[0,\:\infty]}\{
                    n\in\naturalsZero \mid \sum_{i=0}^{n-1}\tilde{\tau}_{i} < t
                \}
            $.
        Then
            $N_{t} < \infty$ \almostSurely
        and
            $N_{t} \distributedAs \DistPoissScale{\lambda t}{1}$,
        including when $t=0$,
        since $\DistPoissScale{0}{1}$ is just the point distribution $\delta_{0}$.
        Moreover, by construction each $N_{t}$ is independent of the $\tau_{i}$.
    \item
        Finally, set
            $
                \theta_{t}
                \colonequals
                    \tau_{< N_{t}}
                = \sum_{i=0}^{N_{t} - 1}\tau_{i}
            $
        for $t\in\realsNonNeg$.
\end{kompaktenum}


\begin{figure}[!ht]

    \begin{subfigure}[m]{\textwidth}
        \centering
        \begin{tikzpicture}[node distance=1cm, thick]
            \pgfmathsetmacro\hunit{3}
            \pgfmathsetmacro\vunit{1}

            \draw [draw=none, fill=none]
                ({-0.1 * \hunit}, {0.5 * \vunit})
                -- ({2.6 * \hunit}, {0.5 * \vunit})
                -- ({2.6 * \hunit}, {-0.5 * \vunit})
                -- ({-0.1 * \hunit}, {-0.5 * \vunit})
                -- cycle;

            \draw [draw=none, fill=drawing_light_grey]
                ({0 * \hunit}, {0 * \vunit})
                -- ({2.5 * \hunit}, {0 * \vunit})
                -- ({2.5 * \hunit}, {0.5 * \vunit})
                -- ({0 * \hunit}, {0.5 * \vunit})
                -- cycle;

            \draw [draw=black, line width=0.5pt, fill=none] ({0.16423594284219328 * \hunit}, {0 * \vunit}) -- ({0.16423594284219328 * \hunit}, {0.5 * \vunit});
            \draw [draw=black, line width=0.5pt, fill=none] ({0.30852786449475916 * \hunit}, {0 * \vunit}) -- ({0.30852786449475916 * \hunit}, {0.5 * \vunit});
            \draw [draw=black, line width=0.5pt, fill=none] ({0.4849472124336579 * \hunit}, {0 * \vunit}) -- ({0.4849472124336579 * \hunit}, {0.5 * \vunit});
            \draw [draw=black, line width=0.5pt, fill=none] ({0.5758038383012576 * \hunit}, {0 * \vunit}) -- ({0.5758038383012576 * \hunit}, {0.5 * \vunit});
            \draw [draw=black, line width=0.5pt, fill=none] ({0.6974385112680267 * \hunit}, {0 * \vunit}) -- ({0.6974385112680267 * \hunit}, {0.5 * \vunit});
            \draw [draw=black, line width=0.5pt, fill=none] ({0.9444818148614873 * \hunit}, {0 * \vunit}) -- ({0.9444818148614873 * \hunit}, {0.5 * \vunit});
            \draw [draw=black, line width=0.5pt, fill=none] ({1.1158979979012023 * \hunit}, {0 * \vunit}) -- ({1.1158979979012023 * \hunit}, {0.5 * \vunit});
            \draw [draw=black, line width=0.5pt, fill=none] ({1.6668098229013812 * \hunit}, {0 * \vunit}) -- ({1.6668098229013812 * \hunit}, {0.5 * \vunit});
            \draw [draw=black, line width=0.5pt, fill=none] ({1.7311098678037797 * \hunit}, {0 * \vunit}) -- ({1.7311098678037797 * \hunit}, {0.5 * \vunit});
            \draw [draw=black, line width=0.5pt, fill=none] ({2.147702862254488 * \hunit}, {0 * \vunit}) -- ({2.147702862254488 * \hunit}, {0.5 * \vunit});

            \draw [decorate, decoration={brace, amplitude=5pt}]
                ({0 * \hunit}, {0.55 * \vunit}) -- ({0.9444818148614873 * \hunit}, {0.55 * \vunit})
                node[midway, above=5pt, align=center]{\footnotesize $N_{t}$};

            \draw [draw=black, line width=2pt, fill=none]
                ({0 * \hunit}, {-0.05 * \vunit})
                node[label=below:{\footnotesize $0$}]{}
                -- ({0 * \hunit}, {0.1 * \vunit});

            \draw [draw=black, line width=2pt, fill=none]
                ({1 * \hunit}, {-0.05 * \vunit})
                node[label=below:{\footnotesize $t$}]{}
                -- ({1 * \hunit}, {0.1 * \vunit});

            \draw [->, draw=black, line width=0.5pt, fill=none]
                ({-0.1 * \hunit}, {0 * \vunit}) -- ({2.6 * \hunit}, {0 * \vunit})
                node [below, align=center]{\footnotesize time};
        \end{tikzpicture}
        \caption{%
            1\textsuperscript{st} homogeneous PPP with rate $\lambda$,
            where $N_{t}$ counts the number of \usesinglequotes{events} in $[0,\:t)$.
        }
        \label{fig:example-auxiliary-poisson:a:sig:article-stochastic-raj-dahya}
    \end{subfigure}

    \begin{subfigure}[m]{\textwidth}
        \centering
        \begin{tikzpicture}[node distance=1cm, thick]
            \pgfmathsetmacro\hunit{3}
            \pgfmathsetmacro\vunit{1}

            \draw [draw=none, fill=none]
                ({-0.1 * \hunit}, {0.5 * \vunit})
                -- ({2.6 * \hunit}, {0.5 * \vunit})
                -- ({2.6 * \hunit}, {-0.5 * \vunit})
                -- ({-0.1 * \hunit}, {-0.5 * \vunit})
                -- cycle;

            \draw [draw=none, fill=drawing_light_grey]
                ({0 * \hunit}, {0 * \vunit})
                -- ({2.5 * \hunit}, {0 * \vunit})
                -- ({2.5 * \hunit}, {0.5 * \vunit})
                -- ({0 * \hunit}, {0.5 * \vunit})
                -- cycle;

            \draw [draw=blue, line width=0.5pt, fill=none] ({0.1704947206981121 * \hunit}, {0 * \vunit}) -- ({0.1704947206981121 * \hunit}, {0.5 * \vunit});
            \draw [draw=blue, line width=0.5pt, fill=none] ({0.986362751410455 * \hunit}, {0 * \vunit}) -- ({0.986362751410455 * \hunit}, {0.5 * \vunit});
            \draw [draw=blue, line width=0.5pt, fill=none] ({1.1292144647056632 * \hunit}, {0 * \vunit}) -- ({1.1292144647056632 * \hunit}, {0.5 * \vunit});
            \draw [draw=blue, line width=0.5pt, fill=none] ({1.2332491071729783 * \hunit}, {0 * \vunit}) -- ({1.2332491071729783 * \hunit}, {0.5 * \vunit});
            \draw [draw=blue, line width=0.5pt, fill=none] ({1.4080395043214222 * \hunit}, {0 * \vunit}) -- ({1.4080395043214222 * \hunit}, {0.5 * \vunit});
            \draw [draw=blue, line width=0.5pt, fill=none] ({1.5578881324962115 * \hunit}, {0 * \vunit}) -- ({1.5578881324962115 * \hunit}, {0.5 * \vunit});
            \draw [draw=black, line width=0.5pt, fill=none] ({2.0546540974785676 * \hunit}, {0 * \vunit}) -- ({2.0546540974785676 * \hunit}, {0.5 * \vunit});

            \draw [decorate, decoration={brace, amplitude=5pt}]
                ({0 * \hunit}, {0.55 * \vunit}) -- ({1.5578881324962115 * \hunit}, {0.55 * \vunit})
                node[midway, above=5pt, align=center]{\footnotesize $N_{t}$};

            \draw [draw=black, line width=2pt, fill=none]
                ({0 * \hunit}, {-0.05 * \vunit})
                node [label=below:{\footnotesize $0$}]{}
                -- ({0 * \hunit}, {0.1 * \vunit});

            \draw [draw=blue, line width=1pt, fill=none]
                ({1.5578881324962115 * \hunit}, {-0.05 * \vunit})
                node [label=below:{\footnotesize $\theta$}]{}
                -- ({1.5578881324962115 * \hunit}, {0.5 * \vunit});

            \draw [->, draw=black, line width=0.5pt, fill=none]
                ({-0.1 * \hunit}, {0 * \vunit}) -- ({2.6 * \hunit}, {0 * \vunit})
                node [below, align=center]{\footnotesize time};
        \end{tikzpicture}
        \caption{%
            2\textsuperscript{nd} homogeneous PPP with rate $\lambda$,
            independent of 1\textsuperscript{st} PPP,
            but with $N_{t}$ determined by (a).
        }
        \label{fig:example-auxiliary-poisson:b:sig:article-stochastic-raj-dahya}
    \end{subfigure}

    \caption{%
        Visualisation of the construction of an \highlightTerm{auxiliary Poisson process},
        $\theta \distributedAs \DistPoissAux{t}{\lambda}$.
    }
    \label{fig:example-auxiliary-poisson:sig:article-stochastic-raj-dahya}
\end{figure}


For each $\lambda\in\realsPos$ and $t\in\realsNonNeg$
let $\DistPoissAux{t}{\lambda}$ denote the distribution of $\theta_{t}$
constructed as above.
We refer to any \randomvar distributed as $\DistPoissAux{t}{\lambda}$
as an \highlightTerm{auxiliary Poisson process}
with \highlightTerm{rate} $\lambda$
over a \highlightTerm{time duration} $t$.
We observe some basic properties of auxiliary Poisson processes.

\begin{prop}
\makelabel{prop:characteristic-function:aux-poisson-process:sig:article-stochastic-raj-dahya}
    Let $\lambda\in\realsPos$ and $t\in\realsNonNeg$.
    The characteristic function
    of a $\DistPoissAux{t}{\lambda}$-distributed \randomvar%
    \footnote{%
        For an $\reals$-valued \randomvar $X$
        the characteristic function is defined as the map
        ${\phi \colon \reals \ni \omega \mapsto \Expected[e^{\iunit \omega X}}]$
        (\cf \cite[Definition~15.7]{Klenke2008probTheory}).
    }
    is given by
        $
            \phi_{t, \lambda}(\omega)
            = e^{\frac{\iunit\omega}{\lambda - \iunit\omega} \lambda t}
        $
    for all $\omega \in \reals$.
    Moreover, the mean and variance of
    $\DistPoissAux{t}{\lambda}$-distributed \randomvar's
    are $t$ and $\frac{2t}{\lambda}$ respectively.
\end{prop}

    \begin{proof}
        Let $\theta \distributedAs \DistPoissAux{t}{\lambda}$.
        Without loss of generality, we may assume that $\theta$
        is given by the above construction, \idest
            $\theta = \theta_{t} = \tau_{<N_{t}}$.
        For $n\in\naturalsPos$
        we have
        $
            \Expected[e^{\iunit\omega\tau_{<n}}]
            = (\frac{\lambda}{\lambda - \iunit\omega})^{n}
        $
        (see \exempli \cite[Theorem~15.12]{Klenke2008probTheory})
        and for $n=0$ clearly
        $
            \Expected[e^{\iunit\omega\tau_{<n}}]
            = \Expected[e^{\iunit\omega\cdot 0}]
            = 1
            \equalscolon (\frac{\lambda}{\lambda - \iunit\omega})^{0}
        $.
        Since by construction, $N_{t}$ is independent of the $\tau_{i}$,
        one thus obtains

            \noparskip
            \begin{eqnarray*}
                \Expected[e^{\iunit\omega\theta}]
                &= &\Expected[\Expected[e^{\iunit\omega\tau_{<N_{t}}}]|N_{t}]\\
                &= &\Expected[(\tfrac{\lambda}{\lambda - \iunit\omega})^{N_{t}}|N_{t}]\\
                &= &\displaystyle
                    \sum_{n=0}^{\infty}
                        e^{-\lambda t}
                        \frac{(\lambda t)^{n}}{n!}
                        \,\Big(
                            \frac{\lambda}{\lambda - \iunit\omega}
                        \Big)^{n}\\
                &= &e^{-\lambda t}e^{\lambda t \, \frac{\lambda}{\lambda - \iunit\omega}}\\
            \end{eqnarray*}

        \continueparagraph
        for all $\omega\in\reals$.
        Hence the characteristic function is as claimed.

        The mean can be computed as follows:
        For $n\in\naturalsPos$
        one has
            $
                \Expected[\tau_{<n}]
                = \sum_{i=1}^{n}\Expected[\tau_{i}]
                = \frac{n}{\lambda}
            $
        and
            $
                \Expected[\tau_{<0}]
                = \Expected[0]
                = 0
            $,
        whence
            $
                \Expected[\theta]
                = \Expected[\Expected[\tau_{<N_{t}} | N_{t}]]
                = \Expected[\frac{N_{t}}{\lambda}]
                = \frac{\lambda t}{\lambda}
                = t
            $.%
        \footnote{%
            Using the characteristic function, one can alternatively compute
            $
                \Expected[\theta]
                = \frac{1}{\iunit}\phi^{\prime}_{t,\lambda}(0)
                =
                    e^{\frac{\iunit\omega}{\lambda - \iunit\omega} \lambda t}
                    \cdot
                    \frac{\lambda^{2}t}{(\lambda - \iunit\omega)^{2}}
                    \rvert_{\:\omega=0}
                = t
            $.
        }
        To compute the variance, we first compute the $2$\textsuperscript{nd} moment of $\theta$.
        Since the characteristic function $\phi_{t,\lambda}$ of $\theta$
        is $2$-times (in fact $\infty$-times) continuously differentiable,
        one may compute (\cf \cite[Theorem~15.34]{Klenke2008probTheory})

            \noparskip
            $$
                \Expected[\theta^{2}]
                = \frac{1}{\iunit^{2}}\phi^{\prime\prime}_{t,\lambda}(0)
                = e^{-\lambda t}e^{\frac{\lambda^{2} t}{\lambda - \iunit\omega}}
                    \Big(
                        \Big(
                            \frac{\lambda^{2} t}{(\lambda - \iunit\omega)^{2}}
                        \Big)^{2}
                        +
                        2\frac{\lambda^{2} t}{(\lambda - \iunit\omega)^{3}}
                    \Big)
                    \Big\rvert_{\:\omega=0}
                = t^{2} + \frac{2t}{\lambda},
            $$

        \continueparagraph
        from which
            $
                \Var(\theta)
                = \Expected[\theta^{2}] - \Expected[\theta]^{2}
                = \frac{2t}{\lambda}
            $
        follows.
    \end{proof}

\begin{prop}
\makelabel{prop:app-semigroup-law:sig:article-stochastic-raj-dahya}
    For each $\lambda\in\realsPos$,
    the family
        $(\DistPoissAux{t}{\lambda})_{t\in\realsNonNeg}$
    of distributions of auxiliary Poisson processes
    constitutes a continuous semigroup of distributions.
\end{prop}

    \begin{proof}
        If $\theta \distributedAs \DistPoissAux{0}{\lambda}$,
        then by the above construction,
            $\theta = \tau_{< N_{0}} = \tau_{< 0} = 0$ \almostSurely.
        We now show for independent \randomvar's
            $t_{1},t_{2} \in \realsNonNeg$
        with
            ${\theta_{1} \distributedAs \DistPoissAux{t_{1}}{\lambda}}$
            and
            ${\theta_{2} \distributedAs \DistPoissAux{t_{2}}{\lambda}}$,
        that
            ${\theta_{1} + \theta_{2} \distributedAs \DistPoissAux{t_{1} + t_{2}}{\lambda}}$.
        To this end let
            $\phi_{\theta_{1}}$, $\phi_{\theta_{2}}$, and $\phi_{\theta_{1} + \theta_{2}}$
        denote the characteristic functions of
            $\theta_{1}$, $\theta_{2}$, and  $\theta_{1}+\theta_{2}$
        respectively.
        By \Cref{prop:characteristic-function:aux-poisson-process:sig:article-stochastic-raj-dahya},
        it holds that
            $\phi_{\theta_{1}} = \phi_{t_{1},\lambda}$
            and
            $\phi_{\theta_{2}} = \phi_{t_{2},\lambda}$
        and by independence
        (\cf \exempli \cite[Lemma~15.11]{Klenke2008probTheory})
            $
                \phi_{\theta_{1} + \theta_{2}}(\omega)
                = \phi_{\theta_{1}}(\omega) \cdot \phi_{\theta_{2}}(\omega)
                = \phi_{t_{1},\lambda}(\omega) \cdot \phi_{t_{2},\lambda}(\omega)
                = e^{\frac{\iunit\omega}{\lambda - \iunit\omega} \lambda t_{1}}
                    e^{\frac{\iunit\omega}{\lambda - \iunit\omega} \lambda t_{2}}
                = \phi_{t_{1}+t_{2},\lambda}(\omega)
            $
        for all $\omega\in\reals$.
        Since the characteristic function uniquely determines the distribution of a \randomvar
        (\cf \cite[Theorem~15.8]{Klenke2008probTheory}),
        it follows by \Cref{prop:characteristic-function:aux-poisson-process:sig:article-stochastic-raj-dahya}
        that $\theta_{1} + \theta_{2} \distributedAs \DistPoissAux{t_{1} + t_{2}}{\lambda}$.

        Towards continuity at $0$, let $f\in\Cts{\realsNonNeg}$ be an arbitrary bounded function.
        Let $t\in\realsNonNeg$ and consider ${\theta_{t} \distributedAs \DistPoissAux{t}{\lambda}}$.
        Without loss of generality, assume that $\theta_{t}$ is constructed as above,
        \idest $\theta_{t} = \tau_{<N_{t}}$.
        We need to show that
            ${\Expected[f(\theta_{t})] \longrightarrow f(0)}$
        for ${\realsPos \ni t \longrightarrow 0}$.
        For $n \in \naturalsPos$
        one has
            $
                \abs{\Expected[f(\tau_{n})]}
                \leq \Expected[\abs{f(\tau_{n})}]
                \leq \norm{f}_{\infty}
            $
        and for $n=0$ one has
            $\Expected[f(\tau_{0})] = \Expected[f(0)] = f(0)$.
        Thus

            \noparskip
            \begin{eqnarray*}
                \absLong{\Expected[f(\theta_{t})] - f(0)}
                &= &\absLong{\Expected[\Expected[f(\theta_{t}) \mid N_{t}]] - f(0)}\\
                &= &\absLong{
                        \displaystyle
                        \sum_{n=0}^{\infty}
                            \Prob[N_{t}=n]
                            \,\Expected[f(\theta_{t}) \mid N_{t}=n]
                        -
                        f(0)
                    }\\
                &= &\absLong{
                        \displaystyle
                        \sum_{n=0}^{\infty}
                            \Prob[N_{t}=n]
                            \underbrace{
                                \Expected[f(\tau_{n})]
                            }_{=f(0)~\text{for $n=0$}}
                        -
                        f(0)
                    }\\
                &\leq &(1 - \Prob[N_{t}=0])\abs{f(0)}
                    + \displaystyle
                    \sum_{n=1}^{\infty}
                        \Prob[N_{t}=n]
                        \norm{f}_{\infty}\\
                &= &(1 - \Prob[N_{t}=0])(\abs{f(0)} + \norm{f}_{\infty})\\
                &= &(1 - e^{-\lambda t})(\abs{f(0)} + \norm{f}_{\infty})\\
                &\longrightarrow &0\\
            \end{eqnarray*}

        \continueparagraph
        for ${\realsPos \ni t \longrightarrow 0}$.
    \end{proof}



\subsection[Approximants as expectations]{Approximants as expectations}
\label{sec:stochastic:one-parameter:sig:article-stochastic-raj-dahya}

\firstparagraph
Our goal is to now associate continuous semigroups of distributions
to $\Cnought$-semigroups of operators via the machinery of expectations
computed strongly via Bochner-integrals.

\begin{defn}[Expectation-approximants]
\makelabel{defn:expectation-approximants:sig:article-stochastic-raj-dahya}
    Let $\Lambda$ be a directed index set
    and ${\Gamma^{(\alpha)} \colonequals (\Gamma^{(\alpha)}(t))_{t \in \realsNonNeg}}$
    be a continuous semigroup of distributions
    for each $\alpha\in \Lambda$.
    Further let $T$ be a contractive $\Cnought$-semigroup on a Banach space $\BanachRaum$
    and define

        \noparskip
        $$
            T^{(\alpha)} \colonequals (\Expected[T(\theta^{(\alpha)}_{t})])_{t\in\realsNonNeg}
        $$

    \continueparagraph
    for each $\alpha \in \Lambda$,
    where ${\theta^{(\alpha)}_{t} \distributedAs \Gamma^{(\alpha)}(t)}$.
    Furthermore, letting
        $\mu^{(\alpha)}_{t}$ and $(\sigma^{(\alpha)}_{t})^{2}$
    denote the mean and variance of $\Gamma^{(\alpha)}(t)$-distributed \randomvar's
    respectively, suppose that

        \noparskip
        $$
            \mu^{(\alpha)}_{t}
                \underset{\alpha}{\longrightarrow} t
            \quad\text{and}\quad
            (\sigma^{(\alpha)}_{t})^{2}
                \underset{\alpha}{\longrightarrow} 0
        $$

    \continueparagraph
    uniformly in $t$ on compact subsets of $\realsNonNeg$.
    In this case, we say that
        $(T^{(\alpha)})_{\alpha\in\Lambda}$
    is a net of \highlightTerm{expectation-approximants}
    with \highlightTerm{associated distribution semigroups}
        $(\Gamma^{(\alpha)})_{\alpha\in\Lambda}$.
    We furthermore refer to $T$ as the \highlightTerm{original semigroup}.
\end{defn}

Our motivation in using expectation-approximants
is that they can be expressed in terms of the original $\Cnought$-semigroup.
Before we proceed with using this broad definition,
we provide some concrete examples.
In particular, we show that the Hille- and Yosida-approximants
satisfy the above definition.

\begin{schattierteboxdunn}[backgroundcolor=leer,nobreak=true]
\begin{lemm}[Hille- and Yosida-approximants as expectations, Chung 1962]
\makelabel{lemm:classical-examples-expectation-approximants:sig:article-stochastic-raj-dahya}
    Let $T$ be a contractive $\Cnought$-semigroup on a Banach space $\BanachRaum$
    with generator $A$.
    The Hille- and Yosida-approximants of $T$
    constitute expectation-approximants.
\end{lemm}
\end{schattierteboxdunn}

For our purposes, it will suffice to assume that $T$ is contractive,
although this result holds without this restriction.
This result can be attributed to Chung
(see \cite[Theorem~1 and Theorem~4]{Chung1962exp}).
For completeness and the reader's convenience, we provide a proof.

    \def\beweislabel{lemm:classical-examples-expectation-approximants:sig:article-stochastic-raj-dahya}
    \begin{proof}[of \Cref{\beweislabel}]
        \paragraph{Hille-approximants:}
            The Hille-approximants are given by the net
                $(T^{(\lambda)})_{\lambda\in\realsPos}$,
            ordered by increasing values of $\lambda$,
            where
                $T^{(\lambda)} = (e^{t A^{(\lambda)}})_{t\in\realsNonNeg}$
                and
                $A^{(\lambda)} = \lambda (T(\tfrac{1}{\lambda}) - \onematrix)$
            for each $\lambda\in\realsPos$.
            For $\lambda\in\realsPos$ the parameterised distributions
                ${
                    \Gamma^{(\lambda)}
                    \colonequals
                    (\DistPoissScale{\lambda t}{\tfrac{1}{\lambda}})_{t\in\realsNonNeg}
                }$
            form a continuous semigroup of distributions (see \S{}\ref{sec:stochastic:distributions:sig:article-stochastic-raj-dahya}),
            with means
                ${\mu^{(\lambda)}_{t} = \frac{1}{\lambda} \cdot \lambda t = t}$
            and variances
                ${
                    (\sigma^{(\lambda)}_{t})^{2}
                    = \frac{1}{\lambda^{2}} \cdot \lambda t
                    = \frac{t}{\lambda}
                    \underset{\lambda}{\longrightarrow} 0
                }$
            for ${\realsPos \ni \lambda \longrightarrow \infty}$,
            and this convergence is clearly uniform in $t$
            on compact subsets of $\realsNonNeg$.
            To satisfy the definition of a net of expectation-approximants,
            it thus suffices to show for
                $t\in\realsNonNeg$,
                $\lambda\in\realsPos$,
                and
                $\theta \distributedAs \DistPoissScale{\lambda t}{\tfrac{1}{\lambda}}$
            that
                $T^{(\lambda)}(t) = \Expected[T(\theta)]$.
            Indeed
                $
                    T^{(\lambda)}(t)
                    = e^{t A^{(\lambda)}}
                    = e^{\lambda t (T(\tfrac{1}{\lambda}) - \onematrix)}
                    = e^{-\lambda t}e^{\lambda t \, T(\tfrac{1}{\lambda})}
                $
            and

                \noparskip
                $$
                    e^{-\lambda t}e^{\lambda t \, T(\tfrac{1}{\lambda})}
                    = \sum_{n=0}^{\infty}
                        \frac{e^{-\lambda t}(\lambda t)^{n}}{n!}
                        \,T(\tfrac{1}{\lambda})^{n}
                    = \sum_{n=0}^{\infty}
                        \Prob[\theta=\tfrac{1}{\lambda}n]\,T(\tfrac{1}{\lambda}n)
                    = \Expected[T(\theta)].
                $$

            \continueparagraph
            Hence the net of Hille-approximants for $T$
            forms a net of expectation-approximants for $T$.

        \paragraph{Yosida-approximants:}
            The Yosida-approximants are given by the net
                $(T^{(\lambda)})_{\lambda\in(\omega_{0}(T)^{+},\:\infty)}$,
            ordered by increasing values of $\lambda$,
            where
                $T^{(\lambda)} = (e^{t A^{(\lambda)}})_{t\in\realsNonNeg}$
                and
                ${
                    A^{(\lambda)}
                    = \lambda \cdot (
                        \lambda \opResolvent{A}{\lambda}
                        - \onematrix
                    )
                }$
            for each $\lambda\in(\omega_{0}(T)^{+},\:\infty)$.
            For $\lambda\in(\omega_{0}(T)^{+},\:\infty)$
            the parameterised distributions
                ${
                    \Gamma^{(\lambda)}
                    \colonequals
                    (\DistPoissAux{t}{\lambda})_{t\in\realsNonNeg}
                }$
            form a continuous semigroup of distributions
            (see \Cref{prop:app-semigroup-law:sig:article-stochastic-raj-dahya}),
            with means
                ${\mu^{(\lambda)}_{t} = \frac{1}{\lambda} \cdot \lambda t = t}$
            and variances
                ${
                    (\sigma^{(\lambda)}_{t})^{2}
                    = \frac{2t}{\lambda}
                    \underset{\lambda}{\longrightarrow} 0
                }$
            for ${(\omega_{0}(T)^{+},\:\infty) \ni \lambda \longrightarrow \infty}$
            (see \Cref{prop:characteristic-function:aux-poisson-process:sig:article-stochastic-raj-dahya}),
            and this convergence is clearly uniform in $t$
            on compact subsets of $\realsNonNeg$.
            To demonstrate that
                $(T^{(\lambda)})_{\lambda\in(\omega_{0}(T)^{+},\:\infty)}$
            forms a net of expectation-approximants for $T$,
            it thus suffices to show for
                $t\in\realsNonNeg$,
                $\lambda\in(\omega_{0}(T)^{+},\:\infty)$,
            and
                ${\theta \distributedAs \DistPoissAux{t}{\lambda}}$
            that
                $T^{(\lambda)}(t) = \Expected[T(\theta)]$.

            By the construction of auxiliary Poisson processes
            in \S{}\ref{sec:stochastic:distributions:sig:article-stochastic-raj-dahya},
            it suffices to prove the claim concretely for
                ${\theta \colonequals \theta_{t} = \tau_{<N_{t}}}$.
            For $n\in\naturalsPos$, the \emph{Phillips calculus} applied to $A$
            (see \exempli
                \cite[Lemma~VIII.1.12~{[$\ast$]}]{Dunfordschwartz1988BookLinOpI},
                \cite[Proposition~3.3.5]{Reissig2005abstractresolvent}%
            )
            yields

                \noparskip
                \begin{eqnarray*}
                    (\onematrix + \lambda^{-1}A^{(\lambda)})^{n}
                        &= &(\lambda \opResolvent{A}{\lambda})^{n}\\
                        &= &\lambda^{n}
                            \cdot
                            \displaystyle
                            \sotInt_{s=0}^{\infty}
                                \frac{s^{n-1}}{(n-1)!}
                                e^{-\lambda s}
                                T(s)
                                \:\dee s\\
                        &= &\displaystyle
                            \sotInt_{s=0}^{\infty}
                                T(s)
                                \:
                                \underbrace{
                                    \lambda
                                    \frac{(\lambda s)^{n-1}}{(n-1)!}
                                    e^{-\lambda s}
                                    \cdot
                                    \:\dee s
                                }_{=\Prob_{\tau_{<n}}(\dee s)},\\
                        &= &\displaystyle
                            \sotInt_{s \in \realsNonNeg}
                                T(s)
                                \:\Prob_{\tau_{<n}}(\dee s)\\
                        &\textoverset{defn}{=}
                            &\Expected[T(\tau_{<n})].\\
                \end{eqnarray*}

            \continueparagraph
            Moreover, since $\tau_{<0} = 0$ \almostSurely,
            one has $T(\tau_{<0}) = \onematrix$ \almostSurely
            and thus
                $
                    \Expected[T(\tau_{<0})]
                    = \onematrix
                    = (\onematrix + \lambda^{-1}A^{(\lambda)})^{0}
                $,
            using the convention of $S^{0} \colonequals \onematrix$
            for $S \in \BoundedOps{\BanachRaum}$.
            Thus

                \noparskip
                \begin{equation}
                \label{eq:1:\beweislabel}
                    (\onematrix + \lambda^{-1}A^{(\lambda)})^{n}
                    = \Expected[T(\tau_{<n})]
                \end{equation}

            \continueparagraph
            for all $n\in\naturalsZero$
            (\cf
                \cite[Lemma~2]{Chung1962exp}%
            ).
            Let $t\in\realsNonNeg$ be arbitrary.
            By the boundedness of $A^{(\lambda)}$
            one has
                $
                    T^{(\lambda)}(t)
                    = e^{t A^{(\lambda)}}
                    = e^{\lambda t (-\onematrix + (\onematrix + \lambda^{-1} A^{(\lambda)}))}
                    = e^{-\lambda t} e^{\lambda t (\onematrix + \lambda^{-1} A^{(\lambda)}))}
                $
            and

                \noparskip
                \begin{eqnarray*}
                    e^{-\lambda t} e^{\lambda t (\onematrix + \lambda^{-1} A^{(\lambda)}))}
                        &= &e^{-\lambda t}
                            \displaystyle
                            \sum_{n=0}^{\infty}
                                \frac{(\lambda t)^{n}}{n!}
                                \underbrace{
                                    (\onematrix + \lambda^{-1} A^{(\lambda)})^{n}
                                }_{\text{%
                                    $\colonequals\onematrix$ for $n=0$%
                                }}\\
                        &\eqcrefoverset{eq:1:\beweislabel}{=}
                        &\displaystyle
                            \sum_{n=0}^{\infty}
                                \underbrace{
                                    e^{-\lambda t}
                                    \frac{(\lambda t)^{n}}{n!}
                                }_{=\Prob[N_{t}=n]}
                                \Expected[T(\tau_{<n})]\\
                        &= &\Expected[T(\tau_{<N_{t}})].\\
                \end{eqnarray*}

            \continueparagraph
            Hence
                ${T^{(\lambda)}(t) = \Expected[T(\tau_{<N_{t}})] = \Expected[T(\theta)]}$.
    \end{proof}

\begin{rem}
    The setup of the auxiliary Poisson processes
    used to demonstrate that Yosida-approximants are expectation-approximants
    in \Cref{lemm:classical-examples-expectation-approximants:sig:article-stochastic-raj-dahya}
    bears some resemblance to \emph{expectation semigroups} in \emph{random evolutions}.
    Note that a random evolution involves a \emph{single Markov-process}:
    both
        for the number of events (referred to as \usesinglequotes{jumps})
        that occur in a time interval $[0,\:t)$,
    as well as
        for the durations between each jump used with the semigroup(s),
    which accumulate to a total duration
    which necessarily lies in $[0,\:t)$
    (\cf \cite[\S{}II.15.7]{Goldstein1985semigroups}).
    By contrast, the construction in \S{}\ref{sec:stochastic:distributions:sig:article-stochastic-raj-dahya}
    involves two independent processes:
        $N_{t}$ counts up the number of events that occur in $[0,\:t)$
        for the first PPP,
    whereas the total duration,
        $\theta_{t} = \sum_{i=0}^{N_{t}}\tau_{i}$,
    used as inputs to the semigroup,
    is determined by a second independent PPP
    and may well lie outside $[0,\:t)$.
\end{rem}

We now observe basic facts about expectation-approximants,
which provide alternative unified proofs of facts
we already knew about Hille- and Yosida-approximants
(\cf \S{}\ref{sec:introduction:results:sig:article-stochastic-raj-dahya}).
We first demonstrate that expectation-approximants
are themselves bona fide $\Cnought$-semigroups
and that the nets
approximate the original $\Cnought$-semigroup.
This result is likely well known (\cf \exempli \cite[Lemma~1]{Chung1962exp}).
For the sake of completeness
and the reader's convenience we present the proofs.

\begin{prop}
\makelabel{prop:expectation-approximants:basic:c0-semigroups-converege:sig:article-stochastic-raj-dahya}
    Let $T$ be a contractive $\Cnought$-semigroup on a Banach space $\BanachRaum$
    and
        $(T^{(\alpha)})_{\alpha \in \Lambda}$
    a net of expectation-approximants for $T$.
    Then
        $T^{(\alpha)}$ is a contractive $\Cnought$-semigroup
    for each $\alpha\in\Lambda$.
    Moreover,
        ${
            T^{(\alpha)}(t) \underset{\alpha}{\longrightarrow} T(t)
        }$
    \wrt the \topSOT-topology
    uniformly in $t$
    on compact subsets of $\realsNonNeg$.
\end{prop}

    \begin{proof}
        Let $(\Gamma^{(\alpha)})_{\alpha\in\Lambda}$
        be the distribution semigroups associated with
        the net of expectation-approximants.
        Further let
            ${\theta^{(\alpha)}_{t} \distributedAs \Gamma^{(\alpha)}(t)}$
        and let
            $\mu^{(\alpha)}_{t}$ and $(\sigma^{(\alpha)}_{t})^{2}$
            be the mean and variance of $\theta^{(\alpha)}_{t}$ respectively
        for each $\alpha \in \Lambda$ and $t\in\realsNonNeg$.

        \paragraph{$\Cnought$-semigroup:}
            Let $\alpha\in \Lambda$ be arbitrary.
            First observe that
                $
                    \norm{T^{(\alpha)}(t)}
                    = \norm{\Expected[T(\theta^{(\alpha)}_{t})]}
                    \leq \Expected[\norm{T(\theta^{(\alpha)}_{t})}]
                    \leq 1
                $
            for each $t\in\realsNonNeg$,
            whence each expectation-approximant is contractive.
            Towards the semigroup law,
            by definition of continuous semigroups of distributions,
                $
                    T^{(\alpha)}(0)
                    = \Expected[T(\theta^{(\alpha)}_{0})]
                    = \Expected[T(0)]
                    = \onematrix
                $
            since $\theta^{(\alpha)}_{0} \distributedAs \delta_{0}$,
            and for each $s,t\in\realsNonNeg$,
            letting
                $\tau_{1} \distributedAs \theta^{(\alpha)}_{s} \distributedAs \Gamma^{(\alpha)}(s)$
                and
                $\tau_{2} \distributedAs \theta^{(\alpha)}_{t} \distributedAs \Gamma^{(\alpha)}(t)$
            be independent,
            one has
                $\tau_{1} + \tau_{2} \distributedAs \theta^{(\alpha)}_{s+t} \distributedAs \Gamma^{(\alpha)}(s+t)$,
            whence
                $
                    T^{(\alpha)}(s)T^{(\alpha)}(t)
                    = \Expected[T(\tau_{1})]\Expected[T(\tau_{2})]
                    = \Expected[T(\tau_{1})T(\tau_{2})]
                    = \Expected[T(\tau_{1}+\tau_{2})]
                    = \Expected[T(\theta^{(\alpha)}_{s+t})]
                    = T^{(\alpha)}(s+t)
                $.
            Towards \topSOT-continuity of $T^{(\alpha)}$,
            it suffices to prove continuity in $0$.
            To this end, let $\xi\in\BanachRaum$ and $\eps > 0$ be arbitrary.
            Since $T$ is \topSOT-continuous,
                ${
                    \sup_{s\in[0,\:\delta]}
                        \norm{(T(s) - \onematrix)\xi}
                    < \eps
                }$
            for some $\delta > 0$.
            Thus

                \noparskip
                \begin{eqnarray*}
                    \norm{(T^{(\alpha)}(t) - T^{(\alpha)}(0))\xi}
                    &= &\normLong{\Big(\Expected[T(\theta^{(\alpha)}_{t})] - \onematrix\Big)\xi}\\
                    &= &\normLong{
                            \displaystyle
                            \int_{s\in\realsNonNeg}
                                (T(s) - \onematrix)\xi
                            \:\Prob_{\theta^{(\alpha)}_{t}}(\dee s)
                        }\\
                    &\leq &\displaystyle
                        \int_{s\in\realsNonNeg}
                            \norm{(T(s) - \onematrix)\xi}
                        \:\Prob_{\theta^{(\alpha)}_{t}}(\dee s)\\
                    &= &\displaystyle
                        \int_{s\in[0,\:\delta]}
                            \underbrace{
                                \normLong{(T(s) - \onematrix)\xi}
                            }_{<\eps}
                        \:\Prob_{\theta^{(\alpha)}_{t}}(\dee s)
                        + \displaystyle
                        \int_{s\in(\delta,\:\infty)}
                            \underbrace{
                                \normLong{(T(s) - \onematrix)\xi}
                            }_{\leq 2\norm{\xi}}
                        \:\Prob_{\theta^{(\alpha)}_{t}}(\dee s)\\
                    &\leq &\eps\,\Prob_{\theta^{(\alpha)}_{t}}[[0,\:\delta]]
                        + 2\norm{\xi}\,\Prob_{\theta^{(\alpha)}_{t}}[(\delta,\:\infty)]\\
                \end{eqnarray*}

            \continueparagraph
            for all $t\in\realsNonNeg$.
            By the Portmanteau theorem (see \cite[Theorem~17.20~v)]{Kechris1995BookDST}) and since by definition
                ${\theta^{(\alpha)}_{t} \underset{t}{\longrightarrow} \delta_{0}}$
            and
                $
                    \delta_{0}(\quer{(\delta,\:\infty)} \without \topInterior{(\delta,\:\infty)}
                    = \delta_{0}(\{\delta\})
                    = 0
                $,
            one has
                ${
                    \Prob_{\theta^{(\alpha)}_{t}}[(\delta,\:\infty)]
                    \longrightarrow
                    \Prob_{\delta_{0}}[(\delta,\:\infty)]
                    = 0
                }$
            for ${\realsPos \ni t \longrightarrow 0}$.
            From the above inequality and since $\eps > 0$ was arbitrarily chosen,
            it follows that
                ${T^{(\alpha)}(t)\xi \longrightarrow T^{(\alpha)}(0)\xi}$
            for ${\realsPos \ni t \longrightarrow 0}$.
            So $T^{(\alpha)}$ constitutes a contractive $\Cnought$-semigroup.

        \paragraph{Approximation:}
            Fix an arbitrary
                $\xi\in\BanachRaum$
                and
                compact subset $L \subseteq \realsNonNeg$.
            Let $\eps > 0$ be arbitrary.
            Without loss of generality, we can assume that $L=[0,\:a]$
            for some $a\in\realsPos$.
            By the \topSOT-continuity of $T$ and compactness of $[0,\:a+1]$,
            there exists $\delta \in (0,\:1)$ such that

                \noparskip
                \begin{equation}
                \label{eq:1:\beweislabel}
                    \sup_{s \in U}
                        \norm{(T(s) - T(t))\xi}
                    < \eps
                \end{equation}

            \continueparagraph
            for all $t \in L$,
            where ${U \colonequals \realsNonNeg \cap (t-\delta,\:t+\delta)}$.
            Since, by definition of expectation-approximants,
                ${\mu^{(\alpha)}_{t} \underset{\alpha}{\longrightarrow} t}$
            uniformly for $t \in L$,
            by \eqcref{eq:1:\beweislabel} it follows that
                ${\sup_{t \in L}\norm{(T(\mu^{(\alpha)}_{t}) - T(t))\xi} \leq \eps}$
            for sufficiently large indices $\alpha \in \Lambda$.
            Furthermore, by the Chebyshev-inequality (see \exempli \cite[Theorem~5.11]{Klenke2008probTheory}), one has that

                \noparskip
                \begin{equation}
                \label{eq:chebyshev:\beweislabel}
                    \Prob_{\theta^{(\alpha)}_{t}}[\realsNonNeg \without U]
                        \leq \delta^{-2}(\sigma^{(\alpha)}_{t})^{2}
                \end{equation}

            \continueparagraph
            for each $t \in L$.
            Since, by definition of expectation-approximants,
                ${(\sigma^{(\alpha)}_{t})^{2} \underset{\alpha}{\longrightarrow} 0}$
            uniformly for $t \in L$,
            by \eqcref{eq:chebyshev:\beweislabel} one has
                ${\sup_{t \in L}\Prob_{\theta^{(\alpha)}_{t}}[\realsNonNeg \without U] \leq \eps}$
            for sufficiently large indices $\alpha \in \Lambda$.
            For sufficiently large indices
                $\alpha \in \Lambda$ and all $t \in L$
            it follows that

                \noparskip
                \begin{eqnarray*}
                    \norm{(T^{(\alpha)}(t) - T(t))\xi}
                        &\leq
                            &\norm{(T(\mu^{(\alpha)}_{t}) - T(t))\xi}
                            + \norm{(T^{(\alpha)}(t) - T(\mu^{(\alpha)}_{t}))\xi}\\
                    &=
                        &\underbrace{
                            \norm{(T(\mu^{(\alpha)}_{t}) - T(t))\xi}
                        }_{\leq\eps}
                        + \normLong{
                            \Big(\Expected[T(\theta^{(\alpha)}_{t})] - T(\mu^{(\alpha)}_{t})\Big)\xi
                        }\\
                    &\leq &\eps + \normLong{
                            \displaystyle
                            \int_{s\in\realsNonNeg}
                                (T(s) - T(\mu^{(\alpha)}_{t}))\xi
                            \:\Prob_{\theta^{(\alpha)}_{t}}(\dee s)
                        }\\
                    &\leq &\eps
                        + \displaystyle
                        \int_{s\in\realsNonNeg}
                            \norm{(T(s) - T(\mu^{(\alpha)}_{t}))\xi}
                        \:\Prob_{\theta^{(\alpha)}_{t}}(\dee s)\\
                    &\leq &\begin{array}[t]{0l}
                            \eps + \displaystyle
                            \int_{s \in U}
                                \underbrace{
                                    \normLong{(T(s) - T(t))\xi}
                                    + \normLong{(T(t) - T(\mu^{(\alpha)}_{t}))\xi}
                                }_{\leq 2\eps}
                            \:\Prob_{\theta^{(\alpha)}_{t}}(\dee s)\\
                            + \displaystyle
                            \int_{s \in \realsNonNeg \without U}
                                \underbrace{
                                    \normLong{(T(s) - T(\mu^{(\alpha)}_{t}))\xi}
                                }_{\leq 2\norm{\xi}}
                            \:\Prob_{\theta^{(\alpha)}_{t}}(\dee s)\\
                        \end{array}\\
                    &\leq &\eps + 2\eps\,\Prob_{\theta^{(\alpha)}_{t}}[U]
                        + 2\norm{\xi}\,\underbrace{
                            \Prob_{\theta^{(\alpha)}_{t}}[\realsNonNeg \without U]
                        }_{\leq\eps}
                    \leq 3\eps + 2\norm{\xi}\eps.\\
                \end{eqnarray*}

            \continueparagraph
            Since $\xi$ and $\eps$ were arbitrarily chosen,
            it follows that
                ${T^{(\alpha)}(t) \longrightarrow T(t)}$
                \wrt the \topSOT-topology
                uniformly in $t$
                on compact subsets of $\realsNonNeg$.
    \end{proof}

\begin{prop}
\makelabel{prop:expectation-approximants:basic:commuting-contractive-families:sig:article-stochastic-raj-dahya}
    Let
        $d\in\naturals$,
        $\BanachRaum$ be a Banach space,
        and
        $\{T_{i}\}_{i=1}^{d}$ be a commuting family
        of contractive $\Cnought$-semigroups on $\BanachRaum$.
    Furthermore, let
        $(T^{(\alpha)}_{i})_{\alpha\in\Lambda_{i}}$
        be expectation-approximants for $T_{i}$
        for each $i\in\{1,2,\ldots,d\}$.
    Then for each
        $\boldsymbol{\alpha} \in \prod_{i=1}^{d} \Lambda_{i}$
    the family of approximants
        $\{T^{(\alpha_{i})}_{i}\}_{i=1}^{d}$
    is a commuting family of contractive $\Cnought$-semigroups.
    Moreover,
        $(\{T^{(\alpha_{i})}_{i}\}_{i=1}^{d})_{\boldsymbol{\alpha}\in\prod_{i=1}^{d} \Lambda_{i}}$
    converges to $\{T_{i}\}_{i=1}^{d}$
    \wrt the \topSOT-topology
    uniformly on compact subsets of $\realsNonNeg^{d}$.
\end{prop}

    \begin{proof}
        \paragraph{Commuting family:}
            By \Cref{prop:expectation-approximants:basic:c0-semigroups-converege:sig:article-stochastic-raj-dahya},
            the approximants $T^{(\alpha_{i})}_{i}$
            are contractive $\Cnought$-semigroups
            for each $i\in\{1,2,\ldots,d\}$.
            Let $\mathbf{t}\in\realsNonNeg^{d}$ be arbitrary.
            Per definition of expectation-approximants
            there exist $\realsNonNeg$-valued \randomvar's
                $\theta_{1},\theta_{1},\ldots,\theta_{d}$
            satisfying
                ${T^{(\alpha_{i})}_{i}(t_{i}) = \Expected[T_{i}(\theta_{i})]}$
            for each $i\in\{1,2,\ldots,d\}$.
            Without loss of generality, we may assume that the $\theta_{i}$ are independent \randomvar's.
            By independence and commutativity
                $
                    \Expected[T_{i}(\theta_{i})]
                    \Expected[T_{j}(\theta_{j})]
                    = \Expected[T_{i}(\theta_{i})T_{j}(\theta_{j})]
                    = \Expected[T_{j}(\theta_{j})T_{i}(\theta_{i})]
                    = \Expected[T_{j}(\theta_{j})]
                    \Expected[T_{i}(\theta_{i})]
                $
            for all $i,j\in\{1,2,\ldots,d\}$ with $i \neq j$.
            Thus $\{T^{(\alpha_{i})}_{i}\}_{i=1}^{d}$
            is a commuting family of contractive $\Cnought$-semigroups.

        \paragraph{Approximation:}
            Let $L \subseteq \realsNonNeg^{d}$ be an arbitrary compact subset.
            Without loss of generality, one may assume $L = \prod_{i=1}^{d}L_{i}$
            for some compact subsets $L_{i} \subseteq \realsNonNeg$, $i\in\{1,2,\ldots,d\}$.
            We prove by induction over $k\in\{1,2,\ldots,d\}$ that

                \noparskip
                \begin{equation}
                \label{eq:1:\beweislabel}
                    \sup_{\mathbf{t} \in \prod_{i=1}^{k}L_{i}}
                    \normLong{
                        \Big(
                            \prod_{i=1}^{k}
                                T^{(\alpha_{i})}_{i}(t_{i})
                            -
                            \prod_{i=1}^{k}
                                T_{i}(t_{i})
                        \Big)
                        \xi
                    }
                    \underset{\boldsymbol{\alpha}}{\longrightarrow} 0
                \end{equation}

            \continueparagraph
            for all $\xi\in\BanachRaum$.
            For $k=1$, this holds by \Cref{prop:expectation-approximants:basic:c0-semigroups-converege:sig:article-stochastic-raj-dahya}.
            Let $1 < k \leq d$ and assume that \eqcref{eq:1:\beweislabel} holds for $k-1$.
            Let $\xi\in\BanachRaum$ and $\eps > 0$ be arbitrary.
            Since $T_{k}$ is \topSOT-continuous and $L_{k}$ is compact,
            there is a finite subset $F \subseteq L_{k}$ such that
                $\min_{t^{\prime} \in F}\norm{(T_{k}(t) - T_{k}(t^{\prime}))\xi} < \eps$
            for each $t \in L_{k}$.
            Let $\mathbf{t} \in L$ be arbitrary
            and let $t^{\prime} \in F$ be such
            that $\norm{(T_{k}(t_{k}) - T_{k}(t^{\prime}))\xi} < \eps$.
            Since the approximants are all contractive,
            one obtains

                \noparskip
                \begin{eqnarray*}
                    \normLong{
                        \Big(
                            \displaystyle
                            \prod_{i=1}^{k}
                                T^{(\alpha_{i})}_{i}(t_{i})
                            -
                            \displaystyle
                            \prod_{i=1}^{k}
                                T_{i}(t_{i})
                        \Big)
                        \xi
                    }
                        &\leq
                            &\normLong{
                                \Big(
                                    \displaystyle
                                    \prod_{i=1}^{k-1}
                                        T^{(\alpha_{i})}_{i}(t_{i})
                                    -
                                    \displaystyle
                                    \prod_{i=1}^{k-1}
                                        T_{i}(t_{i})
                                \Big)
                                T_{k}(t^{\prime})
                                \xi
                            }\\
                            &&+ \underbrace{
                                \normLong{
                                    \displaystyle
                                    \prod_{i=1}^{k-1}
                                        T^{(\alpha_{i})}_{i}(t_{i})
                                    -
                                    \displaystyle
                                    \prod_{i=1}^{k-1}
                                        T_{i}(t_{i})
                                }
                            }_{\leq 2}
                            \underbrace{
                                \normLong{
                                    \Big(
                                        T_{k}(t^{\prime}) - T_{k}(t_{k})
                                    \Big)
                                    \xi
                                }
                            }_{< \eps}\\
                            &&+ \underbrace{
                                \normLong{
                                    \displaystyle
                                    \prod_{i=1}^{k-1}
                                        T^{(\alpha_{i})}_{i}(t_{i})
                                }
                            }_{\leq 1}
                            \normLong{
                                \Big(
                                    T^{(\alpha_{k})}_{k}(t_{k}) - T_{k}(t_{k})
                                \Big)
                                \xi
                            }\\
                        &\leq
                            &\displaystyle
                            \max_{t^{\prime\prime} \in F}
                                \normLong{
                                    \Big(
                                        \displaystyle
                                        \prod_{i=1}^{k-1}
                                            T^{(\alpha_{i})}_{i}(t_{i})
                                        -
                                        \displaystyle
                                        \prod_{i=1}^{k-1}
                                            T_{i}(t_{i})
                                    \Big)
                                    T_{k}(t^{\prime\prime})
                                    \xi
                                }\\
                        &&+ 2\eps + \norm{
                                (T^{(\alpha_{k})}_{k}(t_{k}) - T_{k}(t_{k}))
                                \xi
                            }\\
                \end{eqnarray*}

            \continueparagraph
            for each $\boldsymbol{\alpha} \in \prod_{i=1}^{k}\Lambda_{i}$.
            By
                induction
                (%
                    \eqcref{eq:1:\beweislabel} applied to $k-1$
                    and the finite set of vectors
                    $\{T_{k}(t^{\prime\prime})\xi \mid t^{\prime\prime} \in F\}$%
                )
                and
                \Cref{prop:expectation-approximants:basic:c0-semigroups-converege:sig:article-stochastic-raj-dahya}
                (applied to $T_{k}$),
            by taking $\limsup$ over $\boldsymbol{\alpha}$,
            the right-hand expression is bounded by $0 + 2\eps + 0$.
            Since $\eps > 0$ was arbitrarily chosen, it follows that \eqcref{eq:1:\beweislabel} holds.
            Hence the claim holds by induction.
    \end{proof}

We also obtain the following auxiliary results for simple modifications of $\Cnought$-semigroups:

\begin{prop}
\makelabel{prop:basic-operations:expectation-approximations:sig:article-stochastic-raj-dahya}
    Let
        $\BanachRaum$ be a Banach space
        and
        $T$ be a contractive $\Cnought$-semigroup on $\BanachRaum$.
    Furthermore, let
        $(T^{(\alpha)})_{\alpha \in \Lambda}$
    be a net of expectation-approximants for $T$
    with associated distribution semigroups
        $(\Gamma^{(\alpha)})_{\alpha \in \Lambda}$.
    Then
    for any (equivalently: for an arbitrary) \randomvar
        ${\theta\distributedAs\Gamma^{(\alpha)}(t)}$,
        $t\in\realsNonNeg$,
        and
        $\alpha \in \Lambda$

    \begin{kompaktenum}{\bfseries (i)}
        \item\punktlabel{1}
            $(T^{(\alpha)}(t))^{\prime} = \Expected[T(\theta)^{\prime}]$,
            provided $\BanachRaum$ is reflexive; and
        \item\punktlabel{2}
            $(T^{(\alpha)}(t))^{\ast} = \Expected[T(\theta)^{\ast}]$,
            if $\BanachRaum$ is a Hilbert space.
    \end{kompaktenum}

    \continueparagraph
    Furthermore, in the case of Hille-(\respectively Yosida-)approximants
        $(T^{(\lambda)})_{\lambda \in \realsPos}$
    it holds that

    \begin{kompaktenum}{\bfseries (i)}
    \setcounternach{enumi}{3}
        \item\punktlabel{3}
            $T^{(\lambda)}(rt) = \Expected[T(r\theta)]$;
    \end{kompaktenum}

    \continueparagraph
    for $r\in\realsPos$
    and any (equivalently: for an arbitrary) \randomvar
        ${\theta \distributedAs \DistPoissScale{r\lambda t}{\tfrac{1}{r\lambda}}}$
    (\respectively ${\theta \distributedAs \DistPoissAux{t}{r\lambda}}$).
\end{prop}

    \begin{proof}
        \paragraph{\punktcref{1}:}
            First observe that
                $(T^{(\alpha)}(t)^{\prime})_{t\in\realsNonNeg}$
                and
                $(T(t)^{\prime})_{t\in\realsNonNeg}$
            are $\Cnought$-semigroups
            (see \exempli \cite[Theorem~I.4.9]{Goldstein1985semigroups}).
            Let $\xi\in\BanachRaum^{\prime\prime} \cong \BanachRaum$
            and $\eta\in\BanachRaum^{\prime}$ be arbitrary.
            It holds that

                \noparskip
                \begin{eqnarray*}
                    \brkt{\xi}{T^{(\alpha)}(t)^{\prime}\eta}
                        &= &\brkt{T^{(\alpha)}(t)\xi}{\eta}\\
                        &= &\brkt{\Expected[T(\theta)]\xi}{\eta}\\
                        &= &\brktLong{
                                \Big(
                                    \displaystyle
                                    \sotInt_{s\in\realsNonNeg}
                                        T(s)
                                    \:\Prob_{\theta}(\dee s)
                                \Big)
                                \xi
                            }{\eta}\\
                        &= &\displaystyle
                            \int_{s\in\realsNonNeg}
                                \brkt{T(s)\xi}{\eta}
                            \:\Prob_{\theta}(\dee s)\\
                        &= &\displaystyle
                            \int_{s\in\realsNonNeg}
                                \brkt{\xi}{T(s)^{\prime}\eta}
                            \:\Prob_{\theta}(\dee s)\\
                        &= &\brktLong{\xi}{
                                \Big(
                                    \displaystyle
                                    \sotInt_{s\in\realsNonNeg}
                                        T(s)^{\prime}
                                    \:\Prob_{\theta}(\dee s)
                                \Big)
                                \eta
                            }\\
                        &= &\brkt{\xi}{\Expected[T(\theta)^{\prime}]\eta}.\\
                \end{eqnarray*}

            \continueparagraph
            It follows that $T^{(\alpha)}(t)^{\prime} = \Expected[T(\theta)^{\prime}]$.
            The expression in \punktcref{2} can be proved analogously.

        \paragraph{\punktcref{3}:}
            First observe that
                ${T_{r} \colonequals (T(rt^{\prime}))_{t^{\prime}\in\realsNonNeg}}$
            is a $\Cnought$-semigroup with generator $A_{r} = r A$.
            Let $\lambda\in\realsPos$.
            In the case of Hille-approximants, one has

                \noparskip
                $$
                    A_{r}^{(r\lambda)}
                    = (r\lambda) \cdot (
                        T_{r}(\tfrac{1}{r\lambda}) - \onematrix
                    )
                    = (r\lambda) \cdot (
                        T(\tfrac{1}{\lambda}) - \onematrix
                    )
                    = r A^{(\lambda)}
                $$

            \continueparagraph
            and in the case of Yosida-approximants

                \noparskip
                $$
                    A_{r}^{(r\lambda)}
                    = (r\lambda) \cdot (
                        (r\lambda) \opResolvent{A_{r}}{r\lambda}
                        - \onematrix
                    )
                    = r\lambda \cdot (
                        (r\lambda) \opResolvent{rA}{r\lambda}
                        - \onematrix
                    )
                    = rA^{(\lambda)}.
                $$

            \continueparagraph
            Thus in both cases
                $
                    T^{(r\lambda)}_{r}(t)
                    = e^{t A_{r}^{(r\lambda)}}
                    = e^{t r A^{(\lambda)}}
                    = T^{(\lambda)}(rt)
                $
            holds.
            Applying the properties of the
                $(r\lambda)$\textsuperscript{th}-Hille-approximant
            (\respectively $(r\lambda)$\textsuperscript{th}-Yosida-approximant)
            of $T_{r}$ thus yields

                \noparskip
                $$
                    T^{(\lambda)}(rt)
                    = T^{(r\lambda)}_{r}(t)
                    = \Expected[T_{r}(\theta)]
                    = \Expected[T(r\theta)]
                $$

            \continueparagraph
            for any (equivalently: for an arbitrary) \randomvar
                ${\theta \distributedAs \DistPoissScale{r\lambda t}{\tfrac{1}{r\lambda}}}$
            (\respectively ${\theta \distributedAs \DistPoissAux{t}{r\lambda}}$).
    \end{proof}



\subsection[Families of approximants as expectations]{Families of approximants as expectations}
\label{sec:stochastic:multi-parameter:sig:article-stochastic-raj-dahya}

\firstparagraph
In oder to prove
    \eqcref{it:3:thm:classification-unbounded:poly:sig:article-stochastic-raj-dahya}{}\ensuremath{\implies}{}\eqcref{it:4:thm:classification-unbounded:poly:sig:article-stochastic-raj-dahya}
    of
    \Cref{thm:classification-unbounded:poly:sig:article-stochastic-raj-dahya},
we shall rely on results about the expectations of products.

\begin{prop}
\makelabel{prop:expectation-approximants:product:sig:article-stochastic-raj-dahya}
    Let
        $d\in\naturals$,
        $\BanachRaum$ be a Banach space,
        and
        $\{T_{i}\}_{i=1}^{d}$ be a (not necessarily commuting) family of
        contractive $\Cnought$-semigroups on $\BanachRaum$.
    Further let
        $(T^{(\alpha)}_{i})_{\alpha \in \Lambda_{i}}$
    be a net of expectation-approximants for $T_{i}$
    with associated distribution semigroups
        $(\Gamma^{(\alpha)}_{i}(t))_{t\in\realsNonNeg,\alpha \in \Lambda_{i}}$
    for each $i \in \{1,2,\ldots,d\}$.
    Then for
        $\boldsymbol{\alpha}\in\prod_{i=1}^{d}\Lambda_{i}$,
        and
        $\boldsymbol{t} \in \realsNonNeg^{d}$
    it holds that

        \noparskip
        \begin{equation}
        \label{eq:yosida-stoch:prod:sig:article-stochastic-raj-dahya}
            \displaystyle
            \prod_{i=1}^{d}
                T^{(\alpha_{i})}_{i}(t_{i})
            = \Expected\Big[
                \displaystyle
                \prod_{i=1}^{d}
                    T_{i}(\theta_{i})
            \Big],
        \end{equation}

    \continueparagraph
    for any (equivalently: for an arbitrary) family of
    independent \randomvar's
        ${\theta_{i} \distributedAs \Gamma^{(\alpha_{i})}_{i}(t)}$,
        $i\in\{1,2,\ldots,d\}$.
\end{prop}

    \begin{proof}
        By definition of expectation-approximants we have
            $
                T^{(\alpha_{i})}_{i}(t_{i})
                = \Expected[
                    T_{i}(\theta_{i})
                ]
            $
        for each $i\in\{1,2,\ldots,d\}$
        and thus by independence
            $
                \prod_{i=1}^{d}
                    T^{(\alpha_{i})}_{i}(t_{i})
                = \prod_{i=1}^{d}
                    \Expected[T_{i}(\theta_{i})]
                = \Expected[
                    \prod_{i=1}^{d}
                        T_{i}(\theta_{i})
                ]
            $.
    \end{proof}

Restricting to the context of semigroups over Hilbert spaces yields a result,
which can be utilised with regular polynomial evaluations.
For $p \in \complex[X_{1},X_{1}^{-1},X_{2},X_{2}^{-1},\ldots,X_{d},X_{d}^{-1}]$
let the \highlightTerm{absolute degree of $p$}
denote the largest $n\in\naturalsZero$,
such that $X_{i}^{n}$ or $X_{i}^{-n}$ occurs in some monomial in $p$,
or else $0$ if $p=0$.
In particular, the absolute degree of $p$ is at most $1$
if and only if
the only powers of the $X_{i}$ that occur in monomials in $p$ are $\pm 1$.

\begin{prop}
\makelabel{prop:expectation-approximants:poly:sig:article-stochastic-raj-dahya}
    Let
        $d\in\naturals$,
        $\HilbertRaum$ be a Hilbert space,
        and
        $\{T_{i}\}_{i=1}^{d}$
        be a commuting family of contractive $\Cnought$-semigroups on $\HilbertRaum$.
    Further let
        $(T^{(\alpha)}_{i})_{\alpha\in\Lambda_{i}}$
    be a net of expectation-approximants for $T_{i}$
    with associated distribution semigroups
        $(\Gamma^{(\alpha)}_{i})_{\alpha\in\Lambda_{i}}$
    for each $i \in \{1,2,\ldots,d\}$.
    Then for
        $\boldsymbol{\alpha}\in\prod_{i=1}^{d}\Lambda_{i}$,
        $\boldsymbol{t} \in \realsNonNeg^{d}$
    and any regular polynomial
        $p \in \complex[X_{1},X_{1}^{-1},X_{2},X_{2}^{-1},\ldots,X_{d},X_{d}^{-1}]$
    with absolute degree at most $1$

        \noparskip
        \begin{equation}
        \label{eq:yosida-stoch:poly:sig:article-stochastic-raj-dahya}
            p(
                T^{(\lambda_{1})}_{1}(t_{1}),
                T^{(\lambda_{2})}_{2}(t_{2}),
                \ldots,
                T^{(\lambda_{d})}_{d}(t_{d})
            )
            = \Expected[
                p(
                    T_{1}(\theta_{1}),
                    T_{2}(\theta_{2}),
                    \ldots,
                    T_{d}(\theta_{d})
                )
            ]
        \end{equation}

    \continueparagraph
    for any (equivalently: for an arbitrary) family of
    independent \randomvar's
        ${\theta_{i} \distributedAs \Gamma^{(\alpha_{i})}_{i}(t_{i})}$,
        $i\in\{1,2,\ldots,d\}$.
\end{prop}

    \begin{proof}
        By linearity, it suffices to simply consider monomials of the form
            $p=\prod_{i=1}^{n}X_{i}^{n_{i}}$
        where $\mathbf{n} \in \{0, \pm 1\}^{d}$.
        Let
            $C_{1} \colonequals \supp(\mathbf{n}^{-})$,
            $C_{2} \colonequals \supp(\mathbf{n}^{+})$,
            and
            $K \colonequals \supp(\mathbf{n}) \subseteq \{1,2,\ldots,d\}$,
        where
            $\mathbf{n}^{-} \colonequals (n_{i}^{-})_{i=1}^{d}$
            and
            $\mathbf{n}^{+} \colonequals (n_{i}^{+})_{i=1}^{d}$.
        Then $\isPartition{(C_{1},C_{2})}{K}$.
        By taking regular polynomial evaluations, we thus have to prove that

            \noparskip
            $$
                \prod_{i \in C_{1}}T^{(\lambda_{i})}_{i}(t_{i})^{\ast}
                \cdot
                \prod_{j \in C_{2}}T^{(\lambda_{j})}_{j}(t_{j})
                =
                \Expected\Big[
                    \prod_{i \in C_{1}}T_{i}(\theta_{i})^{\ast}
                    \cdot
                    \prod_{j \in C_{2}}T_{j}(\theta_{j})
                \Big]
            $$

        \continueparagraph
        for any (equivalently: for an arbitrary) family of
        independent \randomvar's
            ${\theta_{i} \distributedAs \Gamma^{(\alpha_{i})}_{i}(t)}$,
            $i \in K$.
        This follows by applying
            \Cref{%
                prop:basic-operations:expectation-approximations:sig:article-stochastic-raj-dahya,%
                prop:expectation-approximants:product:sig:article-stochastic-raj-dahya,%
            }.
    \end{proof}

\begin{rem}
    For this paper we only need the statement in
        \Cref{prop:expectation-approximants:poly:sig:article-stochastic-raj-dahya}
    to hold for regular polynomials of absolute degree at most $1$.
    If we nonetheless sought to widen the scope of this result,
    observe that limitations arise for powers $n$ of $X_{i}$ with $\abs{n}\geq 2$,
    since then the application of
        \Crefit{prop:basic-operations:expectation-approximations:sig:article-stochastic-raj-dahya}{3}
    leads to random variables with different parameterisations.
\end{rem}

As an immediate consequence of \Cref{prop:expectation-approximants:poly:sig:article-stochastic-raj-dahya},
we obtain:

\begin{schattierteboxdunn}[backgroundcolor=leer,nobreak=true]
\begin{lemm}[Transfer result]
\makelabel{lemm:positivity-transfer:poly:sig:article-stochastic-raj-dahya}
    Let
        $d\in\naturals$,
        $\HilbertRaum$ be a Hilbert space,
        and
        $\{T_{i}\}_{i=1}^{d}$
        be a commuting family of contractive $\Cnought$-semigroups on $\HilbertRaum$.
    Further for each $i \in \{1,2,\ldots,d\}$ let
        $(T^{(\alpha)}_{i})_{\alpha \in \Lambda_{i}}$
    be a net of expectation-approximants for $T_{i}$.
    Then for any regular polynomial
        $p \in \complex[X_{1},X_{1}^{-1},X_{2},X_{2}^{-1},\ldots,X_{d},X_{d}^{-1}]$
    with absolute degree at most $1$,
    if
        $p(T_{1}(t_{1}),T_{1}(t_{2}),\ldots,T_{d}(t_{d}))$
        is a positive operator
        for all $\mathbf{t} \in \realsNonNeg^{d}$,
    then
        $p(T^{(\alpha_{1})}_{1}(t_{1}),T^{(\alpha_{1})}_{1}(t_{2}),\ldots,T^{(\alpha_{d})}_{d}(t_{d}))$
        is positive
    for all
        $\boldsymbol{\alpha} \in \prod_{i=1}^{d}\Lambda_{i}$
        and
        $\mathbf{t} \in \realsNonNeg^{d}$.
\end{lemm}
\end{schattierteboxdunn}




\section[First main result]{First main result: Expectation-approximants and polynomial bounds}
\label{sec:first-results:sig:article-stochastic-raj-dahya}


\firstparagraph
We can now prove \Cref{thm:classification-unbounded:poly:sig:article-stochastic-raj-dahya}.

\def\beweislabel{thm:classification-unbounded:poly:sig:article-stochastic-raj-dahya}

\begin{proof}[of \Cref{\beweislabel}]
    First observe by
        \Cref{prop:expectation-approximants:basic:commuting-contractive-families:sig:article-stochastic-raj-dahya}
    that the expectation-approximants constitute commuting families of $\Cnought$-semigroups,
    which converge \wrt the \topSOT-topology uniformly on compact subsets of $\realsNonNeg^{d}$
    to the original family $\{T_{i}\}_{i=1}^{d}$.
    In particular, the implication
        \punktcref{4}{}\ensuremath{\implies}{}\punktcref{1}
    holds by \Cref{thm:classification:dissipativity:sig:article-stochastic-raj-dahya}.
    The implications
        \punktcref{1}{}\ensuremath{\implies}{}\punktcref{2}{}\ensuremath{\implies}{}\punktcref{3}
    hold by \Cref{thm:classification-bounded:poly:sig:article-stochastic-raj-dahya},
    since these do not require the assumption of bounded generators.
    It remains to prove \punktcref{3}{}\ensuremath{\implies}{}\punktcref{4}.
    We first prove this under the assumption
    that the expectation-approximants have bounded generators.

    \paragraph{\punktcref{3}{}\ensuremath{\implies}{}\punktcref{4}, under boundedness assumption on approximants:}
        Let
            $
                \boldsymbol{\alpha}
                \in \prod_{i=1}^{d}\Lambda_{i}
            $
        be arbitrary.
        Let $K \subseteq \{1,2,\ldots,d\}$ be arbitrary.
        By assumption,

            \noparskip
            \begin{eqnarray*}
                p_{K}(T_{1}(t_{1}),T_{2}(t_{2}),\ldots,T_{d}(t_{d}))
                = \displaystyle
                    \sum_{\mathclap{
                        \isPartition{(C_{1},C_{2})}{K}
                    }}
                        \displaystyle
                        \prod_{i \in C_{1}}T_{i}(t_{i})^{\ast}
                        \displaystyle
                        \prod_{j \in C_{2}}T_{j}(t_{j})
                \geq \zeromatrix
            \end{eqnarray*}

        \continueparagraph
        for all
            $\mathbf{t} \in \realsNonNeg^{d}$.
        By \Cref{lemm:positivity-transfer:poly:sig:article-stochastic-raj-dahya},
        this positivity can be transferred to the family of expectation-approximants,
        \idest
            ${
                p_{K}(T^{(\alpha_{1})}_{1}(t_{1}),T^{(\alpha_{2})}_{2}(t_{2}),\ldots,T^{(\alpha_{d})}_{d}(t_{d}))
                \geq \zeromatrix
            }$
        for all
            $\mathbf{t} \in \realsNonNeg^{d}$.
        Since this holds for all $K \subseteq \{1,2,\ldots,n\}$
        and since by assumption the semigroups in
            $\{T^{(\alpha_{i})}_{i}\}_{i=1}^{d}$
        have bounded generators,
        by \Cref{thm:classification-bounded:poly:sig:article-stochastic-raj-dahya}
            $\{T^{(\alpha_{i})}_{i}\}_{i=1}^{d}$
        has a simultaneous regular unitary dilation.

        We have thus established the equivalence of
            \punktcref{1}, \punktcref{2}, \punktcref{3}, \punktcref{4}
        under the assumption that the expectation-approximants have bounded generators.
        Since by \Cref{lemm:classical-examples-expectation-approximants:sig:article-stochastic-raj-dahya}
        one can always use the Hille- or Yosida-approximants (which by construction have bounded generators),
        the equivalences
            \punktcref{1}{}\ensuremath{\iff}{}\punktcref{2}{}\ensuremath{\iff}{}\punktcref{3}
        hold in general. (\textdagger)

        \paragraph{\punktcref{3}{}\ensuremath{\implies}{}\punktcref{4}, without boundedness assumption on approximants:}
        By the above,
            \punktcref{4}{}\ensuremath{\implies}{}\punktcref{1}{}\ensuremath{\implies}{}\punktcref{2}{}\ensuremath{\implies}{}\punktcref{3}
        continue to hold.
        Exactly as argued above, \punktcref{3} implies that
            ${
                p_{K}(T^{(\alpha_{1})}_{1}(t_{1}),T^{(\alpha_{2})}_{2}(t_{2}),\ldots,T^{(\alpha_{d})}_{d}(t_{d}))
                \geq \zeromatrix
            }$
        for all
            $\mathbf{t} \in \realsNonNeg^{d}$,
            $K\subseteq\{1,2,\ldots,d\}$,
            and
            $\boldsymbol{\alpha} \in \prod_{i=1}^{d}\Lambda_{i}$.
        Let $\boldsymbol{\alpha} \in \prod_{i=1}^{d}\Lambda_{i}$ be arbitrary.
        By the general validity of \punktcref{3}{}\ensuremath{\implies}{}\punktcref{1} (see (\textdagger))
        applied to $\{T^{(\alpha_{i})}_{i}\}_{i=1}^{d}$,
        it follows that
            $\{T^{(\alpha_{i})}_{i}\}_{i=1}^{d}$
        has a simultaneous regular unitary dilation.
        Returning to the current context, this means that \punktcref{4} holds.
\end{proof}



\def\beweislabel{thm:classification-unbounded:poly:sig:article-stochastic-raj-dahya}
\begin{rem}
\label{rem:explicit-requirement-contractive-not-needed:sig:article-stochastic-raj-dahya}
    In \Cref{thm:classification-unbounded:poly:sig:article-stochastic-raj-dahya},
    since \punktcref{1}, \punktcref{2}, and \punktcref{3}
    each imply that the $T_{i}$ are contractive,
    for the equivalences
        \punktcref{1}{}\ensuremath{\iff}{}\punktcref{2}{}\ensuremath{\iff}{}\punktcref{3}
    it is not necessary to explicitly demand that the $T_{i}$ are contractive.
    To see that \punktcref{2} implies that the $T_{i}$ are contractive,
    consider polynomial bounds applied to the regular polynomials
    $X_{i}$ for each $i\in\{1,2,\ldots,d\}$.
    To see that \punktcref{3} implies that the $T_{i}$ are contractive,
    let $i\in\{1,2,\ldots,d\}$ be arbitrary.
    By considering the polynomial $p_{\{i\}}$,
    we have that
        $(\onematrix-T_{i}(t)^{\ast}) + (\onematrix-T_{i}(t)) \geq \zeromatrix$
    for all $t\in\realsNonNeg$.
    It readily follows that the generator $A_{i}$ of $T_{i}$
    is dissipative and thus that $T_{i}$ is contractive.
\end{rem}

\def\beweislabel{thm:classification-unbounded:poly:sig:article-stochastic-raj-dahya}
\begin{rem}
\label{rem:significance-of-result-other-dilation:sig:article-stochastic-raj-dahya}
    If we replace \emph{\usesinglequotes{simultaneous regular unitary dilations}}
    by \emph{\usesinglequotes{simultaneous unitary dilations}},
    classifications via bounds of algebraic expressions
    include \cite[Theorem~2.2]{LeMerdy1996DilMultiParam} in the continuous setting,
    and \cite[Corollaries~4.9]{Pisier2001bookCBmaps} in the discrete setting (\idest for tuples of commuting contractions).
    In the discrete setting,
    the existence of a simultaneous \emph{power dilation}
    is characterised by the \emph{complete boundedness}
    of polynomials defined on the operators.
    In the continuous setting,
    le~Merdy characterised the existence of a simultaneous unitary dilation
    by the complete boundedness of a certain functional calculus map
    generated by the resolvents of the semigroups
    and defined on an algebra of holomorphic functions.
    Adding to this picture, our result in
    \Cref{thm:classification-unbounded:poly:sig:article-stochastic-raj-dahya}
    (see also \Cref{rem:explicit-requirement-contractive-not-needed:sig:article-stochastic-raj-dahya})
    answers \Cref{qstn:polynomial:sig:article-stochastic-raj-dahya} positively
    and we now know in full generality
    that finite commuting families of $\Cnought$-semigroups over Hilbert spaces
    have a simultaneous regular unitary dilation if and only if
    they satisfy regular polynomial bounds.
    (In \S{}\ref{sec:functional-calculus:sig:article-stochastic-raj-dahya} a further characterisation
    of regular unitary dilations via the complete positivity of a functional calculus shall be presented,
    which is more general than the regular polynomial bounds and further adds to this picture.)
\end{rem}

\def\beweislabel{thm:classification-unbounded:poly:sig:article-stochastic-raj-dahya}
\begin{rem}
\label{rem:theorem-positivity-in-a-neighbourhood:sig:article-stochastic-raj-dahya}
    Consider a neighbourhood $U \subseteq \realsNonNeg^{d}$ of $\zerovector$
    and let (\punktref{3}') be the assertion of the positivity of the operators in \punktcref{3}
    for $\mathbf{t} \in U$ instead of for all $\mathbf{t}\in\realsNonNeg^{d}$.
    We show that \Cref{thm:classification-unbounded:poly:sig:article-stochastic-raj-dahya}
    holds with \punktcref{3} replaced by (\punktref{3}').
    For this it suffices to show that (\punktref{3}') implies \punktcref{1}
    (\idest that $\{T_{i}\}_{i=1}^{d}$ has a simultaneous regular unitary dilation).

    To this end, fix some ${a\in\realsPos}$ such that ${U \supseteq [0,\:a)^{d}}$
    and consider the subnet of Hille-approximants:
        $(T^{(\lambda)}_{i})_{\lambda\in(a^{-1},\:\infty)}$ for $T_{i}$
        for each ${i\in\{1,2,\ldots,d\}}$.
    Clearly, taking subnets does not affect the fact that these are expectation-approximants.
    Thus applying \Cref{thm:classification-unbounded:poly:sig:article-stochastic-raj-dahya}
    yields that $\{T_{i}\}_{i=1}^{d}$ has a simultaneous regular unitary dilation,
    if and only if
        $\{T^{(\lambda_{i})}_{i}\}_{i=1}^{d}$
    has a simultaneous regular unitary dilation
    for each ${\boldsymbol{\lambda}\in(a^{-1},\:\infty)^{d}}$.
    Since the Hille-approximants have bounded generators,
    by \Cref{thm:classification-bounded:poly:sig:article-stochastic-raj-dahya}
    this holds if and only if the generators
        $\{A^{(\lambda_{i})}_{i}\}_{i=1}^{d}$
    of the Hille-approximants are completely dissipative
    for each ${\boldsymbol{\lambda}\in(a^{-1},\:\infty)^{d}}$.
    By construction of the Hille-approximants, this holds if and only if

        \noparskip
        $$
            (-\tfrac{1}{2})^{\card{K}}
            \sum_{\isPartition{(C_{1},C_{2})}{K}}
                \Big(
                    \prod_{i\in C_{1}}
                        \Big(
                            \lambda_{i}
                            (T_{i}(\tfrac{1}{\lambda_{i}})-\onematrix)
                        \Big)
                \Big)^{\ast}
                \prod_{j\in C_{2}}
                    \Big(
                        \lambda_{j}
                        (T_{j}(\tfrac{1}{\lambda_{j}})-\onematrix)
                    \Big)
            \geq \zeromatrix
        $$

    \continueparagraph
    for all ${\boldsymbol{\lambda}\in(a^{-1},\:\infty)^{d}}$
    and all $K \subseteq \{1,2,\ldots,d\}$.
    This holds if and only if
        ${p_{K}(T_{1}(t_{1}), T_{1}(t_{2}),\ldots,T_{d}(t_{d})) \geq \zeromatrix}$
    for all ${\mathbf{t}\in(0,\:a)^{d}}$
    and all $K \subseteq \{1,2,\ldots,d\}$.
    This in turn is clearly implied by  (\punktref{3}').
\end{rem}

\def\beweislabel{thm:classification-unbounded:poly:sig:article-stochastic-raj-dahya}
\begin{rem}
\label{rem:other-approximants:sig:article-stochastic-raj-dahya}
    By
        \punktcref{1}{}\ensuremath{\iff}{}\punktcref{4}
        of
        \Cref{thm:classification-unbounded:poly:sig:article-stochastic-raj-dahya}
    as well as
        \Cref{lemm:classical-examples-expectation-approximants:sig:article-stochastic-raj-dahya},
    we have answered
        \Cref{qstn:classical-approximants:dilation:sig:article-stochastic-raj-dahya}
    positively
    for a large class of naturally definable examples:
    The simultaneous regular unitary dilatability
        of a commuting contractive family
    is characterised by
    the simultaneous regular unitary dilatability
        of families of semigroups
    in any given net of expectation-approximants,
    \exempli the nets of Hille- and Yosida-approximants.
    It would be interesting to know whether this characterisation holds
    for other classically defined approximants,
    such as the approximants that occur in Kendall's formula
    and the semigroup version of the Post-Widder theorem
    (\cf
        \cite[Theorem~10.4.3~and~11.6.6]{Hillephillips1957faAndSg},
        \cite[Theorems~2--3]{Chung1962exp}%
    ).
    In these cases, in place of the stochastic methods used in the present paper,
    other techniques such as product formulae used with path integrals,
    \exempli Chernoff approximations
    (see \cite{Chernoff1968article,Chernoff1974book,Butko2020chernoff}),
    may be better suited.
\end{rem}

\def\beweislabel{thm:classification-unbounded:poly:sig:article-stochastic-raj-dahya}
\begin{rem}
\label{rem:complete-dissipativity:sig:article-stochastic-raj-dahya}
    Let $\{A_{i}\}_{i=1}^{d}$ be the generators of a commuting family
        $\{T_{i}\}_{i=1}^{d}$
    of $\Cnought$-semigroups on a Hilbert space $\HilbertRaum$.
    Consider now the subset
        $D \subseteq \HilbertRaum$,
    of elements $\xi$ which lie in the domain of
        $A_{k_{n}} \cdot \ldots \cdot A_{k_{2}} \cdot A_{k_{1}}$
    and such that the value of
        ${(A_{k_{n}} \cdot \ldots \cdot A_{k_{2}} \cdot A_{k_{1}})\xi}$
    does not depend on the order of the $k_{i}$
    for injective sequences
        $(k_{i})_{i=1}^{n} \subseteq \{1,2,\ldots,d\}$,
        $n\in\{1,2,\ldots,d\}$.
    It can be shown that $D$ is a dense linear subspace of $\HilbertRaum$
    (see \Cref{prop:dense-subspace-multi-param:sig:article-stochastic-raj-dahya}).
    One may thus extend the notion of complete dissipativity to
    families of generators $\{A_{i}\}_{i=1}^{d}$
    (without the boundedness assumption),
    by demanding that

        \noparskip
        $$
            (-\tfrac{1}{2})^{\card{K}}
            \sum_{\isPartition{(C_{1},C_{2})}{K}}
                \brktLong{
                    \Big(
                    \prod_{j\in C_{2}}
                        A_{j}
                    \Big)
                    \xi
                }{
                    \Big(
                        \prod_{i\in C_{1}}
                            A_{i}
                    \Big)
                    \xi
                }
            \geq 0
        $$

    \continueparagraph
    for all ${\xi \in D}$ and ${K \subseteq \{1,2,\ldots,d\}}$.

    Appealing to condition \eqcref{it:3:thm:classification-unbounded:poly:sig:article-stochastic-raj-dahya}
    of \Cref{thm:classification-unbounded:poly:sig:article-stochastic-raj-dahya},
    if $\{T_{i}\}_{i=1}^{d}$ has a simultaneous regular unitary dilation,
    then by taking limits of the positive expressions
        $
            \frac{1}{2^{K}\prod_{i=1}^{d}t_{i}}
            \brktLong{
                p_{K}(
                    T_{1}(t_{1}),
                    T_{2}(t_{2}),
                    \ldots,
                    T_{d}(t_{d})
                )\xi
            }{\xi}
        $
    for ${\realsPos \ni t_{i} \longrightarrow 0}$
    successively for each ${i \in \{1,2,\ldots,d\}}$
    and for each ${\xi \in D}$,
    we obtain that $\{A_{i}\}_{i=1}^{d}$ is completely dissipative
    by the above definition.
    However, it is unclear whether the reverse implication holds.
    The approach used in \cite[Theorem~1.1]{Dahya2023dilation}
    to link complete dissipativity to
    a previously known condition%
    \footnote{%
        \emph{Brehmer positivity},
        see
        \cite[Theorem~3.2]{Ptak1985}.
    }
    which characterises the existence of simultaneous regular unitary dilations in general,
    relies on asymptotic expressions,
    which in turn rely on the boundedness of the generators.
    Hence an alternative approach is needed for the unbounded setting.
    It would thus be of interest to know whether the above (or an alternative)
    definition of complete dissipativity can be shown to be equivalent
    to any (and thus all) of the conditions in \Cref{thm:classification-unbounded:poly:sig:article-stochastic-raj-dahya}.
\end{rem}




\section[Functional calculi associated with dilations]{Functional calculi associated with dilations}
\label{sec:functional-calculus:sig:article-stochastic-raj-dahya}

\firstparagraph
We now leave the setting of commuting families and turn our attention
to classical dynamical systems modelled by
\topSOT-continuous homomorphisms between topological monoids
and bounded operators over a Hilbert space.
In this section we provide characterisations of unitary and regular unitary dilations
via \emph{functional calculi} defined on (subalgebras of) certain $C^{\ast}$-algebras
related to topological groups.
We thus begin in \S{}\ref{sec:functional-calculus:harmonic-analysis:sig:article-stochastic-raj-dahya}
by recalling the \onetoone-correspondence between ${}^{\ast}$\=/representations of group $C^{\ast}$-algebras
and unitary representations of topological groups.
Then in \S{}\ref{sec:functional-calculus:continuous:sig:article-stochastic-raj-dahya}
we shall use the Wittstock-Haagerup result,
which involves extending and then dilating \emph{completely bounded} maps.
We build on this to generalise le~Merdy's approach to characterise unitary dilatability.
Finally, in \S{}\ref{sec:functional-calculus:discrete:sig:article-stochastic-raj-dahya}
we shall use Averson's result and Stinespring's theorem,
which involves extending and then dilating \emph{completely positive} maps.
We build on this to obtain a characterisation of regular unitary dilations
similar to that of Sz.-Nagy and Foias
(see \cite[Theorem~I.7.1~b]{Nagy1970}).

Throughout this section, $(G,\cdot,e)$ (or simply: $G$)
shall denote a locally compact topological group%
\footnote{%
    Since we are only concerned with continuous maps between $G$
    and other Hausdorff topological groups
    (\exempli the group of unitaries on a Hilbert space under the $\topSOT$-topology),
    it is not important to assume that $G$ be Hausdorff
    (\cf \cite[\S{}1.2]{Deitmar2014bookHarmonicAn}).
    Nonetheless, all of our examples
    (see \Cref{sec:introduction:monoids:sig:article-stochastic-raj-dahya})
    are Hausdorff.
}
and $M$ shall denote a (closed) submonoid of $G$
so that
    $(M,\cdot,e)$ (or simply $M$),
equipped with the relative topology,
comprises a (locally compact) topological monoid.
We furthermore let $\lambda_{G}$ denote a left-invariant Haar-measure on $G$
and express integrals of $G$ via ${\int_{x \in G}\:\cdot\:\dee x}$.


\subsection[Abstract harmonic analysis]{Abstract harmonic analysis}
\label{sec:functional-calculus:harmonic-analysis:sig:article-stochastic-raj-dahya}

\firstparagraph
In order to study dilations of classical dynamical systems on Hilbert spaces,
we shall make use of a fundamental relationship between unitary representations
and ${}^{\ast}$\=/representations of $C^{\ast}$-algebras.
We recall these facts here.
For a detailed exposition,
see \exempli
    \cite[\S{}3.2 and \S{}7.1]{Folland2015bookHarmonicAnalysis},
    \cite[\S{}3.3]{Deitmar2014bookHarmonicAn}.
Let $G$ be a locally compact group
(for which we can fix a left-invariant Haar measure)
and let ${\Delta(\cdot) \colon G\to(\realsPos,\cdot,1)}$
be the \emph{modular function},
which is a continuous homomorphism.
Then $L^{1}(G)$ forms a Banach ${}^{\ast}$\=/algebra
under the convolution operation (viewed as \usesinglequotes{multiplication})
and involution defined by

    \noparskip
    $$
        (f_{1}\ast f_{2})(x)
        \colonequals
        \int_{y \in G}
            f_{1}(y)f_{2}(y^{-1}x)
        \:\dee y
    $$

\continueparagraph
and

    \noparskip
    $$
        f^{\ast}(x)
        \colonequals
            (f(x^{-1}))^{\ast}
            \Delta(x^{-1})
    $$

\continueparagraph
respectively,
for $f, f_{1},f_{2}\in L^{1}(G)$, $x\in G$.
It can then be shown, for any Hilbert space $\HilbertRaum$,
that there is a natural \onetoone-correspondence between
    non-degenerate%
    \footnote{%
        \idest there is no $\xi\in\HilbertRaum\without\{\zerovector\}$
        such that
            $\pi(f)\xi=\zerovector$
        for all $f \in L^{1}(G)$.
        This is the case, \exempli for irreducible representations.
    }
    ${}^{\ast}$\=/representations $\pi$ of $(L^{1}(G),\ast,{}^{\ast})$
    on $\HilbertRaum$
and \topSOT-continuous unitary representations $U$ of $G$
    on $\HilbertRaum$
given by

    \noparskip
    \begin{equation}
    \label{eq:correspondence-unitary-to-star-repr:shift:sig:article-stochastic-raj-dahya}
        U(x)\pi(f) = \pi(L_{x}f) = \pi(f(x^{-1}\cdot ))
    \end{equation}

\continueparagraph
and

    \noparskip
    \begin{equation}
    \label{eq:correspondence-unitary-to-star-repr:inner-product:sig:article-stochastic-raj-dahya}
        \brkt{\pi(f)\xi}{\eta}
        = \int_{x \in G}
            f(x)\brkt{U(x)\xi}{\eta}
        \:\dee x
    \end{equation}

\continueparagraph
for $f\in L^{1}(G)$, $x\in G$, $\xi,\eta\in\HilbertRaum$.

Let $\hat{G}$ denote the set of all
    irreducible ${}^{\ast}$\=/representations of $L^{1}(G)$
    (equivalently: irreducible unitary representations of $G$),
    up to unitary equivalence.
One can then equip
    $(L^{1}(G),\ast,{}^{\ast})$
with the norm

    \noparskip
    $$
        \norm{f}_{\ast}
        \colonequals
            \sup_{[\pi]\in\hat{G}}
                \norm{\pi(f)}
    $$

\continueparagraph
for $f\in L^{1}(G)$.
This renders $(L^{1}(G),\ast,{}^{\ast})$
a dense ${}^{\ast}$\=/subalgebra of a $C^{\ast}$-algebra,
which is referred to as the
\highlightTerm{group $C^{\ast}$-algebra} for $G$,
and is denoted $C^{\ast}(G)$.
Note also that $\norm{f}_{\ast} \leq \norm{f}$
holds for all $f \in L^{1}(G)$.

\begin{conv}
    For simplicity we shall view $L^{1}(G)$ as a subset of $C^{\ast}(G)$.
    We shall interchangeably denote the group $C^{\ast}$-algebra for $G$
    as $C^{\ast}(G)$ and as $\quer{L^{1}(G)}$.
    Equipping $G$ with the discrete topology yields a locally compact group
    with the counting measure as the Haar measure.
    In this case, the $L^{1}$-space is simply $\ell^{1}(G)$.
    We shall thus denote the group $C^{\ast}$-algebra associated
    with the discretised $G$ via $\quer{\ell^{1}(G)}$.
\end{conv}

As a $C^{\ast}$-algebra, $C^{\ast}(G)$ has a unique unital extension.
Note that $L^{1}(G)$ contains a unital element
if and only if $G$ is discrete,
in which case, the unit is $\delta_{e}$.
For this reason,
we shall denote the unital extension of $C^{\ast}(G)$ by

    \noparskip
    $$
        \complex\cdot\delta_{e} + C^{\ast}(G)
    $$

\continueparagraph
in case $G$ is continuous.
If $G$ is discrete, we have that
    $\complex\cdot\delta_{e} + C^{\ast}(G) = C^{\ast}(G)$.
Otherwise the above extension can be understood as
    $\complex\cdot\delta_{e} \oplus C^{\ast}(G)$.

Relying on the density of $L^{1}(G)$ in $C^{\ast}(G)$,
the above correspondence can be restated.
The following is a slightly adapted version of the proofs in \cite{Folland2015bookHarmonicAnalysis}:

\begin{prop}[Correspondence between ${}^{\ast}$- and unitary representations]
\makelabel{prop:correspondence-unitary-to-star:sig:article-stochastic-raj-dahya}
    Let
        $G$ be a locally compact topological group
        and
        $\HilbertRaum$ a Hilbert space.
    Then for every
        \topSOT-continuous unitary representation $U$ of
        $G$ on $\HilbertRaum$
    there exists a
        non-degenerate ${}^{\ast}$\=/representation $\pi = \pi_{U}$ of
        $C^{\ast}(G)$ on $\HilbertRaum$,
    such that
        \eqcref{eq:correspondence-unitary-to-star-repr:shift:sig:article-stochastic-raj-dahya}
        and
        \eqcref{eq:correspondence-unitary-to-star-repr:inner-product:sig:article-stochastic-raj-dahya}
    hold.
    And for every
        non-degenerate ${}^{\ast}$\=/representation $\pi$ of
        $C^{\ast}(G)$ on $\HilbertRaum$,
    there exists an
        \topSOT-continuous unitary representation $U = U_{\pi}$ of
        $G$ on $\HilbertRaum$,
    such that
        \eqcref{eq:correspondence-unitary-to-star-repr:inner-product:sig:article-stochastic-raj-dahya}
    holds.
    The constructions
        ${U\mapsto\pi_{U}}$ and ${\pi\mapsto U_{\pi}}$
    establish a \onetoone-correspondence
    between \topSOT-continuous unitary representations of $G$ on $\HilbertRaum$
    and non-degenerate ${}^{\ast}$\=/representations of $C^{\ast}(G)$ on $\HilbertRaum$.
\end{prop}

    \begin{proof}
        For the first claim, by
            \cite[Theorem~3.9]{Folland2015bookHarmonicAnalysis},
        there exists a non-degenerate ${}^{\ast}$\=/representation $\pi$ of
            $L^{1}(G)$ on $\HilbertRaum$
        satisfying
            \eqcref{eq:correspondence-unitary-to-star-repr:shift:sig:article-stochastic-raj-dahya}
            and
            \eqcref{eq:correspondence-unitary-to-star-repr:inner-product:sig:article-stochastic-raj-dahya}.
        Using Zorn's lemma, one can decompose
            $\HilbertRaum$
        into closed $\pi$-invariant subspaces with cyclic vectors
            $\HilbertRaum=\bigoplus_{i}\HilbertRaum_{i}$,
        thereby obtaining irreducible ${}^{\ast}$\=/representations $\pi_{i}$
        of $L^{1}(G)$ on each $\HilbertRaum_{i}$.
        By construction of the $\norm{\cdot}_{\ast}$-norm,
        each $\pi_{i}$ and thus $\pi$ itself are contractive.
        It follows that $\pi$ can be extended
        to a bounded linear operator between
            $\quer{L^{1}(G)} = C^{\ast}(G)$
        and $\BoundedOps{\HilbertRaum}$,
        which we may also call $\pi$.
        Since the algebraic operations in $C^{\ast}(G)$ are continuous,
        $\pi$ remains a ${}^{\ast}$\=/representation.

        Towards the second claim, by density, $\pi$ is a non-degenerate ${}^{\ast}$\=/representation
        of $L^{1}(G)$ on $\HilbertRaum$,
        and thus by \cite[Theorem~3.11]{Folland2015bookHarmonicAnalysis},
        there exists an \topSOT-continuous unitary representation $U$ of
            $G$ on $\HilbertRaum$,
        satisfying
            \eqcref{eq:correspondence-unitary-to-star-repr:inner-product:sig:article-stochastic-raj-dahya}.

        Towards the final claim, we first show that $\pi_{U_{\pi}} = \pi$
        for each ${}^{\ast}$\=/representation $\pi$ of $C^{\ast}(G)$ on $\HilbertRaum$.
        Using
            \eqcref{eq:correspondence-unitary-to-star-repr:inner-product:sig:article-stochastic-raj-dahya}
        yields
            $
                \brkt{\pi_{U_{\pi}}(f)xi}{\eta}
                = \int_{x \in G}f(x)\brkt{U_{\pi}(x)\xi}{\eta}\:\dee x
                = \brkt{\pi(f)\xi}{\eta}
            $
        for all $f\in L^{1}(G)$ and $\xi,\eta\in\HilbertRaum$.
        By the density of $L^{1}(G)$ in $C^{\ast}(G)$
        and continuity of $\pi$, $\pi_{U_{\pi}}$,
        it follows that $\pi_{U_{\pi}} = \pi$.
        To show that $U_{\pi_{U}} = U$
        for each \topSOT-continuous unitary representation $U$ of
            $G$ on $\HilbertRaum$,
        using
            \eqcref{eq:correspondence-unitary-to-star-repr:inner-product:sig:article-stochastic-raj-dahya}
        yields
            $
                \int_{K}\brkt{U_{\pi_{U}}(x)\xi}{\eta}\:\dee x
                = \brkt{\pi_{U}(\einser_{K})xi}{\eta}
                = \int_{K}\brkt{U(x)\xi}{\eta}\:\dee x
            $
        for all compact $K \subseteq G$
        and all $\xi,\eta\in\HilbertRaum$.
        By the \topWOT-continuity of $U$, $U_{\pi_{U}}$,
        it follows that $U_{\pi_{U}} = U$.
    \end{proof}



\subsection[The Phillips--le~Merdy functional calculus]{The Phillips--le~Merdy functional calculus}
\label{sec:functional-calculus:continuous:sig:article-stochastic-raj-dahya}

\firstparagraph
For a locally compact topological group $G$ and closed submonoid $M \subseteq G$,
we consider

    \noparskip
    $$
        \LpPlusCpct{1}{M}{G}
        \colonequals
        \{
            f \in L^{1}(G)
            \mid
            \quer{\supp}(f) \subseteq M,
            ~\text{compact}
        \},
    $$

\continueparagraph
which is a subalgebra of the convolution algebra $(L^{1}(G),\ast)$
and thus of the unital $C^{\ast}$-algebra ${\complex\cdot\delta_{e} + C^{\ast}(G)}$.
For an \topSOT-continuous homomorphism
    ${T \colon M \to \BoundedOps{\HilbertRaum}}$
on $\HilbertRaum$, consider the map
    ${\funcCalcCts \colon \complex\cdot\delta_{e} + \LpPlusCpct{1}{M}{G} \to \BoundedOps{\HilbertRaum}}$
defined by

    \noparskip
    \begin{equation}
    \label{eq:functional-calc:cts:sig:article-stochastic-raj-dahya}
        \funcCalcCts(c\delta_{e} + f)
            = c\onematrix
            + \sotInt_{x \in \quer{\supp}(f)}
                f(x)T(x)
            \:\dee x
    \end{equation}

\continueparagraph
for $f\in\LpPlusCpct{1}{M}{G}$, $c\in\complex$.%
\footnote{%
    Note that since $T$ is $\topSOT$-continuous,
    one has that $T(\cdot)\restr{K}\xi$
    has a norm-compact and thus separable image
    for any compact subset $K \subseteq M$.
    Thus
        ${\quer{\supp}(f) \ni x \mapsto f(x)T(x)\xi \in \HilbertRaum}$
    is Bochner-integrable for each $f\in\LpPlusCpct{1}{M}{G}$.
    In particular, we do not need to demand the separability of $\HilbertRaum$.
}
It is easy to verify that $\funcCalcCts$ is a linear unital map
which satisfies
    $
        \funcCalcCts(f \ast g)
        = \funcCalcCts(f)
        \funcCalcCts(g)
    $
for $f,g\in\LpPlusCpct{1}{M}{G}$.
We shall refer to this unital homomorphism
as the \highlightTerm{Phillips--le~Merdy calculus}
associated with $T$.

\begin{rem}
\makelabel{rem:continuous-functional-calculus-le-merdy-phillips:sig:article-stochastic-raj-dahya}
    Consider the case $(G,M)=(\reals^{d},\realsNonNeg^{d})$.
    Let
        $\HilbertRaum$ be a Hilbert space
        and
        ${T \colon \realsNonNeg^{d} \to \BoundedOps{\HilbertRaum}}$
        be an \topSOT-continuous contractive homomorphism on $\HilbertRaum$.
    Further let
        $\{T_{i}\}_{i=1}^{d}$
    be the corresponding commuting family of contractive $\Cnought$-semigroups.
    Then \eqcref{eq:functional-calc:cts:sig:article-stochastic-raj-dahya}
    can be naturally extended
    to
        ${
            \mathcal{B}
            \colonequals
            \complex\cdot\delta_{\zerovector} \oplus \mathcal{B}_{0}
        }$,
    where
        ${
            \mathcal{B}_{0}
            \colonequals
            \{f\in L^{1}(\reals^{d}) \mid \supp(f) \subseteq \realsNonNeg^{d}\}
        }$.
    This extension of $\funcCalcCts$ to $\mathcal{B}$
    is essentially a restriction of the \emph{Phillips calculus}
    (see \exempli
        \cite[Lemma~VIII.1.12~{[$\ast$]}]{Dunfordschwartz1988BookLinOpI},
        \cite[Proposition~3.3.5]{Reissig2005abstractresolvent}%
    )
    applied to the generators $\{A_{i}\}_{i=1}^{d}$ of $\{T_{i}\}_{i=1}^{d}$.
    Consider now arbitrary
        $\mathbf{n}\in\naturalsPos^{d}$
    and
        $\boldsymbol{\lambda}\in\realsPos^{d}$,
    and let
        $f\in \mathcal{B}_{0}$
    be defined by
    ${
        f(\mathbf{t})
        \colonequals
        \prod_{i=1}^{d}
            \lambda_{i}\frac{(\lambda_{i}t_{i})^{n_{i}-1}}{(n-1)!}
            e^{-\lambda_{i}t_{i}}
    }$
    for $\mathbf{t}\in\realsNonNeg^{d}$.
    One has
        ${
            (\Fourier f)(\boldsymbol{\omega})
            = \prod_{i=1}^{d}
                (\frac{\lambda_{i}}{\lambda_{i} + \iunit\omega_{i}})^{n_{i}}
        }$
    for $\boldsymbol{\omega}\in\reals^{d}$,
    where ${\Fourier \colon L^{1}(\reals^{d}) \to \Cts_{0}{\reals^{d}}}$
    denotes the Fourier transform.
    Then
        ${
            \funcCalcCts(f)
            = \prod_{i=1}^{d}
                \sotInt_{s=0}^{\infty}
                    \lambda_{i}\frac{(\lambda_{i}s)^{n_{i}-1}}{(n-1)!}
                    e^{-\lambda_{i}s}
                    T_{i}(s)
                    \:\dee s
            = \prod_{i=1}^{d}
                (\lambda_{i} \opResolvent{A_{i}}{\lambda_{i}})^{n_{i}}
        }$.
    Thus, in the case of $(G,M)=(\reals^{d},\realsNonNeg^{d})$,
    the Phillips calculus is a common extension of $\funcCalcCts$
    and (upon application of the Fourier transform)
    the one defined by le~Merdy in \cite[Definition~2.1]{LeMerdy1996DilMultiParam}.
\end{rem}

The main result in this subsection is thus inspired by le~Merdy's characterisation of simultaneous unitary dilations.
By simplifying his approach and utilising results from abstract harmonic analysis we obtain a generalisation.
Before presenting this, we need the following approximation.

\begin{prop}
\makelabel{prop:func-calc-cts:basic:shift-unit:sig:article-stochastic-raj-dahya}
    Let $G$ be a locally compact topological group and
    $M \subseteq G$ a closed submonoid.
    Further let
        ${T \colon M \to \BoundedOps{\HilbertRaum}}$
    be an \topSOT-continuous homomorphism.
    If $M$ is $e$-joint,
    then there exists a net
        $(f_{i})_{i} \subseteq \LpPlusCpct{1}{M}{G}$
    such that
        $\norm{f_{i}}_{\ast} \leq 1$ for all $i\in I$,
        ${\norm{f_{i} \ast g - g}_{\ast} \underset{i}{\longrightarrow} 0}$
    for all $g \in L^{1}(G)$,
    and
        $\funcCalcCts(f_{i}) \underset{i}{\longrightarrow} \onematrix$
    \wrt the \topSOT-topology.
\end{prop}

    \begin{proof}
        Let $\mathcal{N}$ be the filter of all compact neighbourhoods of
        the group identity $e \in G$.
        By compactness and $e$-jointedness,
            $0 < \lambda_{G}(K \cap M) \leq \lambda_{G}(K) < \infty$
        for all $K \in \mathcal{N}$.
        Thus
            $f_{K} \colonequals \frac{1}{\lambda_{G}(K \cap M)}\einser_{K \cap M}$
        is a well-defined element of $L^{1}(G)$
        with $\quer{\supp}(f) = K \cap M \subseteq M$
        for each $K\in\mathcal{N}$,
        and moreover
            $
                \norm{f}_{\ast}
                \leq \norm{f}_{1}
                = 1
            $.
        Hence it suffices to consider the net
            $(f_{K})_{K\in\mathcal{N}}$
        directly ordered by reverse inclusion.

        Let $g \in L^{1}(G)$ be arbitrary.
        Then
            $\norm{f_{K} \ast g - g}_{\ast} \leq \norm{f_{K} \ast g - g}_{1}$
        for each $h\in\realsPos$.
        Since convolution is $L^{1}$-continuous
        (see \exempli \cite[Proposition~2.40a)]{Folland2015bookHarmonicAnalysis}) and $\Cts_{c}{G}$ is dense in $L^{1}(G)$,
        it suffices to prove that
            ${\norm{f_{K} \ast g - g}_{1} \underset{K}{\longrightarrow}  0}$
        for each $g\in\Cts_{c}{G}$.
        This is a straightforward matter that can be derived
        using uniform continuity arguments.

        Towards the final claim, one has
            $
                \norm{\funcCalcCts(f_{K})\xi - \xi}
                = \normLong{
                    \Big(
                        \frac{1}{\lambda_{G}(K \cap M)}
                        \sotInt_{x \in K \cap M}
                            T(x)
                        \:\dee x
                    \Big)
                    \xi
                    -
                    \xi
                }
                \leq
                    \frac{1}{\lambda_{G}(K \cap M)}
                    \int_{x \in K \cap M}
                        \norm{T(x)\xi - T(e)\xi}
                    \:\dee x
                \leq
                    \sup_{x \in K \cap M}
                        \norm{(T(x) - T(e))\xi}
            $
        for each $\xi\in\HilbertRaum$,
        whereby the latter expression
        converges to $0$ since $T$ is \topSOT-continuous.
    \end{proof}

\begin{schattierteboxdunn}[backgroundcolor=leer,nobreak=true]
\begin{lemm}[Generalisation of le~Merdy's characterisation of dilations]
\makelabel{lemm:cts-functional-calculus:sig:article-stochastic-raj-dahya}
    Let $G$ be a locally compact topological group and
    $M \subseteq G$ an $e$-joint closed submonoid.
    Further let
        ${T \colon M \to \BoundedOps{\HilbertRaum}}$
    be an \topSOT-continuous homomorphism.
    Then $T$ has a unitary dilation if and only if $\funcCalcCts$ is completely bounded
    with $\normCb{\funcCalcCts} \leq 1$.
\end{lemm}
\end{schattierteboxdunn}

    \begin{proof}
        Via the GNS-construction, we may view
        the unital $C^{\ast}$-algebra,
            ${\mathcal{A} \colonequals \complex\cdot\delta_{e} + C^{\ast}(G)}$
        as a unital $C^{\ast}$-subalgebra of $\BoundedOps{\HilbertRaum_{0}}$
        for some Hilbert space $\HilbertRaum_{0}$.

        \paragraph{Necessity:}
            Suppose that $(\HilbertRaum_{1},U,r)$ is a unitary dilation of $T$.
            By the correspondence between unitary representations
            and non-degenerate ${}^{\ast}$\=/representations
            in abstract harmonic analysis
            (see \Cref{prop:correspondence-unitary-to-star:sig:article-stochastic-raj-dahya}),
            there exists a non-degenerate ${}^{\ast}$\=/representation
                ${
                    \pi \colon
                    C^{\ast}(G)
                    \to
                    \BoundedOps{\HilbertRaum_{1}}
                }$
            such that \eqcref{eq:correspondence-unitary-to-star-repr:inner-product:sig:article-stochastic-raj-dahya} holds.
            Now, $\pi$ can be extended to
                ${
                    \tilde{\pi} \colon
                    \complex\cdot\delta_{e} + C^{\ast}(G)
                    \to
                    \BoundedOps{\HilbertRaum_{1}}
                }$
            defined by
                ${\tilde{\pi}(c\delta_{e} + f) \colonequals c\onematrix + \pi(f)}$
                for $f\in C^{\ast}(G)$ and $c\in\complex$.
            By non-degeneracy, $\tilde{\pi}$ is a well-defined%
            \footnote{%
                in particular, if $C^{\ast}(G)$ already contains the $\delta_{e}$,
                then by non-degeneracy $\pi(\delta_{e})=\onematrix$
                must hold.
            }
            unital ${}^{\ast}$\=/algebra representation.
            Using the unitary dilation and the ${}^{\ast}$\=/representation,
            one obtains

                \noparskip
                \begin{eqnarray*}
                    \brkt{\funcCalcCts(c\delta_{e} + f)\xi}{\eta}
                    &= &\brktLong{
                            \Big(
                                c\onematrix
                                +
                                \displaystyle
                                \sotInt_{x \in \quer{\supp}(f)}
                                    f(x)T(x)
                                \:\dee x
                            \Big)
                            \xi
                        }{\eta}\\
                    &= &c\brkt{\xi}{\eta}
                        +
                        \displaystyle
                        \int_{x\in\quer{\supp}(f)}
                            f(x)
                            \underbrace{
                                \brkt{T(x)\xi}{\eta}
                            }_{=\brkt{r^{\ast}U(x)r\xi}{\eta}}
                        \:\dee x\\
                    &= &c\brkt{r\xi}{r\eta}
                        +
                        \displaystyle
                        \int_{x\in\quer{\supp}(f)}
                            f(x)
                            \brkt{U(x)r\xi}{r\eta}
                        \:\dee x\\
                    &\eqcrefoverset{eq:correspondence-unitary-to-star-repr:inner-product:sig:article-stochastic-raj-dahya}{=}
                        &\brkt{c\onematrix\,r\xi}{r\eta}
                        +
                        \brkt{\pi(f)r\xi}{r\eta}\\
                    &= &\brkt{r^{\ast}\tilde{\pi}(c\cdot\delta_{e} + f)r\xi}{\eta}
                \end{eqnarray*}

            \continueparagraph
            for all
                $f \in \LpPlusCpct{1}{M}{G} \subseteq L^{1}(G)$,
                $c\in\complex$,
                $\xi,\eta\in\HilbertRaum$.
            Thus $\funcCalcCts(a) = r^{\ast}\tilde{\pi}(a)r$
            for all $a\in \complex\cdot\delta_{e} + \LpPlusCpct{1}{M}{G}$.
            Let $n\in\naturals$ and $\mathbf{a} = (a_{ij})_{ij} \in M_{n}(\complex\cdot\delta_{e} + \LpPlusCpct{1}{M}{G})$.
            Then
                $
                    (\funcCalcCts\otimes\id_{M_{n}})(\mathbf{a})
                    = (\funcCalcCts(a_{ij}))_{ij}
                    = (r^{\ast}\tilde{\pi}(a_{ij})r)_{ij}
                    = (r \otimes \onematrix_{M_{n}})^{\ast}
                        (\tilde{\pi} \otimes \id_{M_{n}})(\mathbf{a})
                    (r \otimes \onematrix_{M_{n}})
                $.
            Now, since $\tilde{\pi} \otimes \id_{M_{n}}$
            is a unital ${}^{\ast}$\=/algebra representation,
            it is necessarily contractive.
            It follows that $\funcCalcCts\otimes\id_{M_{n}}$ is contractive for each $n\in\naturals$.
            Thus $\normCb{\funcCalcCts} \leq 1$.

        \paragraph{Sufficiency:}
            If $\funcCalcCts$ is completely bounded with $\normCb{\funcCalcCts} \leq 1$,
            then, since $\funcCalcCts$ is a (contractive!) unital homomorphism
            defined on the unital subalgebra
                ${
                    \complex\cdot\delta_{e} + \LpPlusCpct{1}{M}{G}
                    \subseteq \mathcal{A}
                    \subseteq \BoundedOps{\HilbertRaum_{0}}
                }$,
            one may apply the dilation theorem in \cite[Theorem~4.8]{Pisier2001bookCBmaps}%
            \footnote{%
                This result is based on the factorisation of completely bounded maps
                (see \cite[Theorem~4.8]{Pisier2001bookCBmaps}),
                which builds on Stinespring's theorem.
                Pisier attributes this to the independent work of
                    Wittstock and Haagerup as well as Paulsen,
                and cites
                \cite{Wittstock1981ArticleDilationDe,Wittstock1984inCollDilationEn}
                as well as an unpublished work from Haagerup
                (\cf the comments before Theorem~3.6 in \cite{Pisier2001bookCBmaps}).
            }
            and obtain
                a Hilbert space $\HilbertRaum_{1}$,
                an isometry $r\in\BoundedOps{\HilbertRaum}{\HilbertRaum_{1}}$,
                and
                a ${}^{\ast}$\=/algebra representation
                ${\pi \colon \BoundedOps{\HilbertRaum_{0}} \to \BoundedOps{\HilbertRaum_{1}}}$
            such that

                \noparskip
                \begin{equation}
                \label{eq:1:\beweislabel}
                    \funcCalcCts(a) = r^{\ast}\,\pi(a)\,r
                \end{equation}

            \continueparagraph
            for all $a\in\complex\cdot\delta_{e} + \LpPlusCpct{1}{M}{G}$.
            Since $\mathcal{A}$ is unital,
            we can replace $\HilbertRaum_{1}$
            with the $\pi$-invariant subspace $\quer{\pi(\mathcal{A})r\HilbertRaum}$,
            which contains $r\HilbertRaum$.
            In particular, one can assume that $\pi$ is a unital
            (and thus non-degenerate) ${}^{\ast}$\=/representation
            of $\mathcal{A}$ on $\HilbertRaum_{1}$
            and that $\pi(\mathcal{A})r\HilbertRaum$ is dense in $\HilbertRaum_{1}$.
            Since $L^{1}(G)$ is $\norm{\cdot}_{\ast}$-dense in $C^{\ast}(G)$,
            it follows that
                $(\complex\cdot\onematrix + \pi(L^{1}(G)))r\HilbertRaum$
            is dense in $\HilbertRaum_{1}$.

            By the correspondence between unitary representations
            and non-degenerate ${}^{\ast}$\=/representations
            in abstract harmonic analysis
            (see \Cref{prop:correspondence-unitary-to-star:sig:article-stochastic-raj-dahya}),
            there exists a (unique) \topSOT-continuous unitary representation
                ${U \colon G \to \BoundedOps{\HilbertRaum_{1}}}$
            such that

                \noparskip
                \begin{equation}
                \label{eq:2:\beweislabel}
                    U(x)\pi(f) = \pi(L_{x}f) = \pi(f(x^{-1}\cdot))
                \end{equation}

            \continueparagraph
            for all $x \in G$ and all $f \in L^{1}(G)$.
            Our goal is to show that $(\HilbertRaum_{1},U,r)$ is a unitary dilation of $T$.

            Let $x \in M$ and $f \in \LpPlusCpct{1}{M}{G}$ be arbitrary.
            Then
                $L_{x}f = f(x^{-1}\cdot) \in \LpPlusCpct{1}{M}{G}$,
            since
                $\quer{\supp}(f(x^{-1}\cdot)) = x \cdot \quer{\supp}(f) \subseteq M$.
            Applying the construction of $\funcCalcCts$ yields

            \noparskip
            \begin{eqnarray*}
                r^{\ast}\,U(x)\,\pi(f)\,r
                    &\eqcrefoverset{eq:2:\beweislabel}{=}
                        &r^{\ast}\,\pi(\Fourier (f(x^{-1}\cdot)))\,r\\
                    &\eqcrefoverset{eq:1:\beweislabel}{=}
                        &\funcCalcCts(f(x^{-1}\cdot))\\
                    &= &\displaystyle
                        \sotInt_{y \in \quer{\supp}(f(x^{-1}\cdot))}
                            f(x^{-1}y) T(y)
                            \:\dee y\\
                    &= &\displaystyle
                        \sotInt_{y \in \quer{\supp}(f)}
                            f(y) T(xy)
                            \:\dee y\\
                    &= &T(x)
                        \cdot
                        \displaystyle
                        \sotInt_{y \in \quer{\supp}(f)}
                            f(y) T(y)
                            \:\dee y\\
                    &= &T(x)\funcCalcCts(f).\\
            \end{eqnarray*}

            Consider now the net
                $(f_{i})_{i \in I} \subseteq \LpPlusCpct{1}{M}{G}$
            constructed
            in \Cref{prop:func-calc-cts:basic:shift-unit:sig:article-stochastic-raj-dahya},
            for which it holds that
                ${\funcCalcCts(f_{i}) \underset{i}{\longrightarrow} \onematrix}$
            \wrt the \topSOT-topology.
            By construction,
                $\norm{f_{i}}_{\ast} \leq 1$
            for each $i\in I$,
            and since $\pi$ is a ${}^{\ast}$\=/algebra representation,
            it follows that
                $\norm{\pi(f_{i})} \leq \norm{f_{i}}_{\ast} \leq 1$
            for all $i\in I$.
            We now claim that
                ${\pi(f_{i})r \underset{i}{\longrightarrow} r}$
            \wrt the \topWOT-topology.
            Since $(\pi(f_{i}))_{i \in I}$ is uniformly bounded
            and ${(\complex\cdot\onematrix + \pi(L^{1}(G)))r\HilbertRaum}$
            is dense in $\HilbertRaum_{1}$,
            it suffices to show that
                ${
                    \brkt{\pi(f_{i})r\xi}{ar\eta}
                    \underset{i}{\longrightarrow}
                    \brkt{r\xi}{ar\eta}
                }$
            for
                $a\in\complex\cdot\onematrix + \pi(L^{1}(G))$
                and
                $\xi,\eta\in\HilbertRaum$.
            To this end, consider
                an arbitrary $\eta\in\HilbertRaum$
            and
                $a = c\cdot\onematrix + \pi(g)$
            for arbitrary $c\in\complex$, $g\in L^{1}(G)$.
            By the properties of the construction in
                \Cref{prop:func-calc-cts:basic:shift-unit:sig:article-stochastic-raj-dahya}
            one has
                ${\funcCalcCts(f_{i}) \underset{i}{\longrightarrow} \onematrix}$
                \wrt the \topSOT-topology
            as well as
                ${\norm{g^{\ast} \ast f_{i} - g^{\ast}}_{\ast} \underset{i}{\longrightarrow} 0}$
            and thus
                ${\pi(g^{\ast} \ast f_{i}) \underset{i}{\longrightarrow} \pi(g^{\ast}) = \pi(g)^{\ast}}$
            in norm.
            Hence

                \noparskip
                \begin{eqnarray*}
                    \brkt{\pi(f_{i})r\xi}{ar\eta}
                        &= &c^{\ast}\brkt{\pi(f_{i})r\xi}{r\eta}
                            + \brkt{\pi(f_{i})r\xi}{\pi(g)r\eta}\\
                        &= &c^{\ast}\brkt{r^{\ast}\pi(f_{i})r\xi}{\eta}
                            + \brkt{\pi(g)^{\ast}\pi(f_{i})r\xi}{r\eta}\\
                        &\eqcrefoverset{eq:1:\beweislabel}{=}
                            &c^{\ast}\brkt{\funcCalcCts(f_{i})\xi}{\eta}
                            + \brkt{\pi(g^{\ast} \ast f_{i})r\xi}{r\eta}\\
                        &\underset{i}{\longrightarrow}
                            &c^{\ast}\brkt{\onematrix\,\xi}{\eta}
                            + \brkt{\pi(g)^{\ast}r\xi}{r\eta}
                        = \brkt{r\xi}{ar\eta},\\
                \end{eqnarray*}

            \continueparagraph
            from which the claim follows.
            Taking weak limits in the above computation
            applied to the $f_{i}$ thus yields
                $
                    r^{\ast}\,U(x)\,r = T(x)\cdot\onematrix
                $
            for all $x \in M$.
            Hence $(\HilbertRaum_{1},U,r)$ is a unitary dilation of $T$.
    \end{proof}



\subsection[The discrete functional calculus]{The discrete functional calculus}
\label{sec:functional-calculus:discrete:sig:article-stochastic-raj-dahya}

\firstparagraph
For a (not necessarily locally compact!) topological group $G$, we consider

    \noparskip
    $$
        c_{00}(G)
        =
        \{
            f \in \ell^{1}(G)
            \mid
            \supp(f)~\text{finite}
        \},
    $$

\continueparagraph
which is a ${}^{\ast}$\=/subalgebra of the convolution algebra $(\ell^{1}(G),\ast)$
and thus of the unital $C^{\ast}$-algebra $\quer{\ell^{1}(G)}$.
Let $M \subseteq$ be an arbitrary submonoid
and suppose that $(G,M,\cdot^{+})$ is a positivity structure
(see \Cref{defn:positivity-structure-monoids:sig:article-stochastic-raj-dahya}).
For a (not necessarily \topSOT-continuous) homomorphism
    ${T \colon M \to \BoundedOps{\HilbertRaum}}$
on $\HilbertRaum$, consider the map
    ${\funcCalcDiscr \colon c_{00}(G) \to \BoundedOps{\HilbertRaum}}$
defined by

    \noparskip
    \begin{equation}
    \label{eq:functional-calc:discr:sig:article-stochastic-raj-dahya}
        \funcCalcDiscr(f)
            = \sum_{x \in \supp(f)}
                f(x)T(x^{-})^{\ast}T(x^{+})
    \end{equation}

\continueparagraph
for $f\in c_{00}(G)$.
One can readily check that $\funcCalcDiscr$ is a
linear self-adjoint unital map
which satisfies
    $
        \funcCalcDiscr(f \ast g)
        = \funcCalcDiscr(f)
        \funcCalcDiscr(g)
    $
for $f,g\in c_{00}(G)$.
We shall refer to this linear self-adjoint unital map
as the \highlightTerm{discrete functional calculus}
associated with $T$.

\begin{schattierteboxdunn}[backgroundcolor=leer,nobreak=true]
\begin{lemm}[Characterisation of regular dilations \`{a} la Sz.-Nagy]
\makelabel{lemm:discr-functional-calculus:sig:article-stochastic-raj-dahya}
    Let $(G,M,\cdot^{+})$ be a positivity structure
    where $G$ is a topological group
    and $M \subseteq G$ is a submonoid.%
    \footref{ft:disc-func-calc:G-not-locally-compact:sig:article-stochastic-raj-dahya}
    Further let
        ${T \colon M \to \BoundedOps{\HilbertRaum}}$
    be an \topSOT-continuous homomorphism.
    Then $T$ has a regular unitary dilation if and only if $\funcCalcDiscr$ is completely positive.
\end{lemm}
\end{schattierteboxdunn}

\footnotetext[ft:disc-func-calc:G-not-locally-compact:sig:article-stochastic-raj-dahya]{
    Note that we neither require $G$ to be locally compact
    nor $M$ to be a measurable subset in this theorem!
}

Parts of the proof of \Cref{lemm:discr-functional-calculus:sig:article-stochastic-raj-dahya}
are similar to \cite[Theorem~I.7.1~b)]{Nagy1970}.
However, there are two main differences.
Firstly, Sz.-Nagy works with extensions of $T$ to all of $G$,
without explicitly defining this
(except in the special cases of $G=\reals^{d}$ and $G=\integers^{d}$).
Secondly, our approach relies on Stinespring's dilation theorem for $C^{\ast}$-algebras,
whilst Sz.-Nagy's approach is more directly connected the theory of unitary representations.

    \def\beweislabel{lemm:discr-functional-calculus:sig:article-stochastic-raj-dahya}
    \begin{proof}[of \Cref{\beweislabel}]
        Let $\mathcal{A} \colonequals \quer{\ell^{1}(G)}$
        be the (unital) group $C^{\ast}$-algebra
        for the discretised version of $G$.
        Note that $c_{00}(G)$ is unital and self-adjoint,
        and thus constitutes an \emph{operator system}
        (\cf \cite[Chapter~2, p.~9]{Paulsen1986bookCBmapsAndDilations}).

        \paragraph{Necessity:}
            Suppose that $(\HilbertRaum_{1},U,r)$ is a regular unitary dilation of $T$.
            By the correspondence between unitary representations
            and non-degenerate ${}^{\ast}$\=/representations
            in abstract harmonic analysis
            (see \Cref{prop:correspondence-unitary-to-star:sig:article-stochastic-raj-dahya}),
            there exists a non-degenerate ${}^{\ast}$\=/representation
                ${
                    \pi \colon
                    \mathcal{A}
                    \to
                    \BoundedOps{\HilbertRaum_{1}}
                }$
            such that \eqcref{eq:correspondence-unitary-to-star-repr:inner-product:sig:article-stochastic-raj-dahya} holds.
            Using the regular unitary dilation and the ${}^{\ast}$\=/representation,
            one obtains

                \noparskip
                \begin{eqnarray*}
                    \brkt{\funcCalcDiscr(f)\xi}{\eta}
                    &= &\brktLong{
                            \Big(
                                \displaystyle
                                \sum_{x \in \supp(f)}
                                    f(x)
                                    \underbrace{
                                        T(x^{-})^{\ast}T(x^{+})
                                    }_{=r^{\ast}U(x)r}
                            \Big)
                            \xi
                        }{\eta}\\
                    &= &\displaystyle
                        \sum_{x \in \supp(f)}
                            f(x)
                            \brkt{U(x)r\xi}{r\eta}\\
                    &\eqcrefoverset{eq:correspondence-unitary-to-star-repr:inner-product:sig:article-stochastic-raj-dahya}{=}
                        &\brkt{\pi(f)r\xi}{r\eta}\\
                    &= &\brkt{r^{\ast}\,\pi(f)\,r\xi}{\eta}
                \end{eqnarray*}

            \continueparagraph
            for all
                $f \in c_{00}(G)$,
                $\xi,\eta\in\HilbertRaum$.
            Thus $\funcCalcDiscr(a) = r^{\ast}\,\pi(a)\,r$
            for all $a\in c_{00}(G)$.
            For $n\in\naturals$ and positive matrices
                $\mathbf{a} = (a_{ij})_{ij} \in M_{n}(c_{00}(G))$
            it follows that
                $
                    (\funcCalcDiscr \otimes \id_{M_{n}})(\mathbf{a})
                        = (\funcCalcDiscr(a_{ij}))_{ij}
                        = \Big(
                            r^{\ast}\,\pi(a_{ij})\,r
                        \Big)_{ij}
                        = (r \otimes \onematrix_{M_{n}})^{\ast}
                            (\pi \otimes \id_{M_{n}})(\mathbf{a})
                            (r \otimes \onematrix_{M_{n}})
                $,
            which is positive,
            since
                ${\pi \otimes \id_{M_{n}}}$
            is a ${}^{\ast}$\=/representation of
            the $C^{\ast}$-algebra $M_{n}(\mathcal{A})$
            and thus positive.
            Thus $\funcCalcDiscr$ is completely positive.

        \paragraph{Sufficiency:}
            Since $c_{00}(G) \subseteq \mathcal{A}$ is an operator system,
            Averson's extension theorem
            (see \cite[Theorem~7.5]{Paulsen1986bookCBmapsAndDilations})
            yields an extension of $\funcCalcDiscr$ to a completely positive map
            between $\mathcal{A}$ and $\BoundedOps{\HilbertRaum}$.
            Stinespring's dilation theorem
            (see \cite[Theorem~1 and \S{}3.~Remarks]{Stinespring1955dilation})
            applied to this yields
                a Hilbert space $\HilbertRaum_{1}$,
                an isometry $r\in\BoundedOps{\HilbertRaum}{\HilbertRaum_{1}}$,
                and
                a unital ${}^{\ast}$\=/representation ${\pi \colon \mathcal{A}\to\BoundedOps{\HilbertRaum_{1}}}$,
            such that

                \noparskip
                \begin{equation}
                \label{eq:1:\beweislabel}
                    \funcCalcDiscr(a) = r^{\ast}\,\pi(a)\,r
                \end{equation}

            \continueparagraph
            holds for all $a\in c_{00}(G)$.
            Since $c_{00}(G)$ is a dense, unital ${}^{\ast}$\=/subalgebra of $\mathcal{A}$,
            we can replace $\HilbertRaum_{1}$ by the $\pi$-invariant closed subspace
                $\quer{c_{00}(G)r\HilbertRaum}$,
            which contains $r\HilbertRaum$.

            Since for all $x,y \in G$
            one has
                $\delta_{x}\ast\delta_{y} = \delta_{xy}$
            and
                $\delta_{x}\in c_{00}(G)\subseteq\mathcal{A}$
            are unitary,%
            \footnote{%
                since
                    $\delta_{x}^{\ast}\ast\delta_{x} = \delta_{x^{-1}x} = \delta_{e}$
                and
                    $\delta_{x}\ast\delta_{x}^{\ast} = \delta_{xx^{-1}} = \delta_{e}$.
            }
            and since $\pi$ is a unital ${}^{\ast}$\=/representation,
            it follows that ${U \colon G \to \BoundedOps{\HilbertRaum}}$
            defined by $U(x) \colonequals \pi(\delta_{x})$
            is a unitary homomorphism of $G$ on $\HilbertRaum$.
            Moreover by \eqcref{eq:1:\beweislabel}

                \noparskip
                \begin{equation}
                \label{eq:2:\beweislabel}
                    T(x^{-})^{\ast}T(x^{+})
                    = \funcCalcDiscr(\delta_{x})
                    = r^{\ast}\,\pi(\delta_{x})\,r
                    = r^{\ast}\,U(x)\,r
                \end{equation}

            \continueparagraph
            for all $x \in G$.
            To show that $(\HilbertRaum_{1},U,r)$ is a regular unitary dilation of $T$,
            it thus remains to show that $U$ is $\topSOT$-continuous.

            To this end, first note that
                ${
                    (\BoundedOps{\HilbertRaum},\topSOT)
                    \times (\BoundedOps{\HilbertRaum},\topSOT)
                    \ni (R, S)
                    \mapsto
                    S^{\ast}R \in (\BoundedOps{\HilbertRaum},\topWOT)
                }$
            is continuous.%
            \footnote{%
                This follows directly from the
                observation that
                    ${
                        (\BoundedOps{\HilbertRaum},\topSOT)^{2}
                        \ni (R,S)
                        \mapsto
                        \brkt{S^{\ast}R\xi}{\eta}
                        = \brkt{R\xi}{S\eta}
                        \in\complex
                    }$
                is continuous for each $\xi,\eta\in\HilbertRaum$,
                which in turn holds,
                since the inner product
                    ${
                        \brkt{}{}
                        \colon (\HilbertRaum,\norm{\cdot})^{2} \to \complex
                    }$
                is continuous (by the Cauchy-Schwarz inequality).
                Note that we do not need to restrict the operators to bounded subsets
                of $\BoundedOps{\HilbertRaum}$
                (\cf \cite[Lemma~3.1]{Ptak1985}).
            }
            So since ${T \colon M \to (\BoundedOps{\HilbertRaum},\topSOT)}$
            and
            $\cdot^{+}$ (and thus $\cdot^{-}$)
            are continuous,
            it follows that
                ${
                    G \ni x
                    \mapsto
                    r^{\ast}U(x)r
                    \eqcrefoverset{eq:2:\beweislabel}{=}
                    T(x^{-})^{\ast}T(x^{+})
                    \in \BoundedOps{\HilbertRaum}
                }$
            is \topWOT-continuous.
            This implies that

                \noparskip
                $$
                    G \ni x
                    \mapsto
                    \begin{array}[t]{0l}
                        \brkt{
                            U(x)
                            \pi(\delta_{y})r\xi
                        }{
                            \pi(\delta_{z})r\eta
                        }\\
                        = \brkt{
                                U(x)
                                U(y)
                                r\xi
                            }{
                                U(z)
                                r\eta
                            }\\
                        = \brkt{
                                r^{\ast}
                                U(z^{-1}xy)
                                r\xi
                            }{
                                r\eta
                            }\\
                    \end{array}
                $$

            \continueparagraph
            is continuous
            for all $\xi,\eta\in\BoundedOps{\HilbertRaum}$, $y,z\in G$,
            which in turn entails the continuity of
                ${
                    G \ni x \mapsto \brkt{U(x)\pi(f)r\xi}{\pi(g)r\eta}
                }$
            for all $\xi,\eta\in\BoundedOps{\HilbertRaum}$, $f,g \in c_{00}(G)$.
            Since $U$ is unitary-valued and $\pi(c_{00}(G))r\HilbertRaum$ is dense in $\HilbertRaum_{1}$,
            it follows that $U$ is a \topWOT- and thus indeed an \topSOT-continuous
            unitary representation of $G$ on $\HilbertRaum_{1}$.
    \end{proof}

\begin{rem}
    The \topSOT-continuity of $T$ was only used
    to prove the \topSOT-continuity of the unitary representation.
    Without this assumption, the above proof shows that
    $T$ has a (not necessarily \topSOT-continuous) regular unitary dilation
    if and only if $\funcCalcDiscr$ is completely positive.
\end{rem}




\section[Second main result]{Second main results: Unitary approximants}
\label{sec:second-results:sig:article-stochastic-raj-dahya}

\firstparagraph
The functional calculi presented in \S{}\ref{sec:functional-calculus:sig:article-stochastic-raj-dahya}
to characterise unitary and regular unitary dilations respectively,
provide us the means to study topological approximations of classical dynamical systems.
We exploit these results to prove
\Cref{%
    thm:unitary-approx:weak:sig:article-stochastic-raj-dahya,%
    thm:unitary-approx:regular-weak:sig:article-stochastic-raj-dahya,%
}.


\firstparagraph
\def\beweislabel{thm:unitary-approx:weak:sig:article-stochastic-raj-dahya}
\begin{proof}[of \Cref{\beweislabel}]
    The implications \punktcref{2}{}\ensuremath{\implies}{}\punktcref{3} are clear
    irrespective of the assumptions on $\dim(\HilbertRaum)$.

    \paragraph{\punktcref{3}{}\ensuremath{\implies}{}\punktcref{1}:}
        Let
            $(U^{(\alpha)})_{\alpha\in\Lambda}$,
        be a net of \topSOT-continuous unitary representations of $G$ on $\HilbertRaum$.
        Suppose that $(U^{(\alpha)}\restr{M})_{\alpha\in\Lambda}$
        approximates $T$ in the \emph{uniform weak} sense.
        We make use of the \emph{Phillips--le~Merdy calculi}
            ${
                \funcCalcCts,\funcCalcCts^{(\alpha)}
                \colon
                \LpPlusCpct{1}{M}{G} \to \BoundedOps{\HilbertRaum}
            }$
        associated with $T$
        and each $U^{(\alpha)}\restr{M}$ respectively
        (see \S{}\ref{sec:functional-calculus:continuous:sig:article-stochastic-raj-dahya}).
        By the characterisation in
            \Cref{lemm:cts-functional-calculus:sig:article-stochastic-raj-dahya},
            $\normCb{\funcCalcCts^{(\alpha)}} \leq 1$
        for each $\alpha\in\Lambda$,
        and, in order to show that $T$ has a unitary dilation,
        it suffices to show that
            $\funcCalcCts$
        is completely bounded
        with $\normCb{\funcCalcCts} \leq 1$.
        To this end, first observe that for each
            $
                a = c\cdot\delta_{e} + f
                \in \complex\cdot\delta_{e} + \LpPlusCpct{1}{M}{G}
                \equalscolon \mathcal{A}
            $,
        uniform weak convergence yields

            \noparskip
            \begin{eqnarray*}
                \brkt{
                    \funcCalcCts^{(\alpha)}(a)
                    \xi
                }{\eta}
                &= &\brktLong{
                        \Big(
                            c\onematrix
                            +
                            \displaystyle
                            \sotInt_{x \in \quer{\supp}(f)}
                                f(x)
                                U^{(\alpha)}(x)
                            \:\dee x
                        \Big)
                        \xi
                    }{\eta}\\
                &= &c\brkt{\xi}{\eta}
                    +
                    \displaystyle
                    \int_{x \in \quer{\supp}(f)}
                        f(x)
                        \brkt{U^{(\alpha)}(x)\xi}{\eta}
                    \:\dee x\\
                &\underset{\alpha}{\longrightarrow}
                    &c\brkt{\xi}{\eta}
                    +
                    \displaystyle
                    \int_{x \in \quer{\supp}(f)}
                        f(x)
                        \brkt{T(x)\xi}{\eta}
                    \:\dee x\\
                &= &\brkt{
                    \funcCalcCts(a)
                    \xi
                }{\eta}\\
            \end{eqnarray*}

        \continueparagraph
        for all $\xi,\eta\in\HilbertRaum$.
        Thus
            ${
                \funcCalcCts^{(\alpha)}(a)
                \underset{\alpha}{\longrightarrow}
                \funcCalcCts(a)
            }$
        \wrt the $\topWOT$-topology
        for each $a\in \mathcal{A}$.
        It follows that
            ${
                (\funcCalcCts^{(\alpha)} \otimes \id_{M_{n}})(\mathbf{a})
                \underset{\alpha}{\longrightarrow}
                (\funcCalcCts \otimes \id_{M_{n}})(\mathbf{a})
            }$
        \wrt the $\topWOT$-topology
        for
            $n\in\naturals$
            and matrices
            $\mathbf{a} = (a_{ij})_{ij} \in M_{n}(\mathcal{A})$.
        Since
            $\funcCalcCts^{(\alpha)} \otimes \id_{M_{n}}$
        is a contraction for each $\alpha\in\Lambda$ and each $n\in\naturals$,
        it follows that
            $\funcCalcCts \otimes \id_{M_{n}}$
        is a contraction for each $n\in\naturals$.
        Thus $\funcCalcCts$ is completely bounded
        with $\normCb{\funcCalcCts} \leq 1$.

    \paragraph{\punktcref{1}{}\ensuremath{\implies}{}\punktcref{2}, under cardinality assumption:}
        By assumption, $G$ contains a dense subset $D \subseteq G$
        and $\dim(\HilbertRaum) \geq \max\{\aleph_{0},\card{D}\}$.
        Without loss of generality, one can replace $D$ by a dense subgroup of $G$.
        Let
            $(\HilbertRaum_{1},U,r)$
        be a regular unitary dilation of $T$.
        Let $B \subseteq \HilbertRaum$ be an orthonormal basis (ONB) for $\HilbertRaum$
        and
            $\kappa \colonequals \card{B} = \dim(\HilbertRaum) \geq \max\{\aleph_{0},\card{D}\}$
        and consider

            \noparskip
            $$
                \HilbertRaum_{0} \colonequals \quer{\linspann}\{
                    U(x)r\xi
                    \mid
                    x \in D,
                    \xi \in B
                \},
            $$

        \continueparagraph
        which is a $U$-invariant subspace.
        By the above cardinality assumptions
        and elementary computations with infinite cardinals
        (see \exempli \cite[\S{}I.10]{Kunen2011BookSetTh}),
        one has cardinality
            $\card{\HilbertRaum_{0}} = \card{B} = \kappa$.
        It follows that
            $(\HilbertRaum_{0},U\restr{\HilbertRaum_{0}},r)$
        is a regular unitary dilation of $T$.
        Furthermore, since $r$ is isometric, we have
            $
                \kappa
                = \dim(\HilbertRaum)
                = \dim(\ran(r))
                \leq \dim(\HilbertRaum_{0})
                = \kappa
            $.
        Thus $\dim(\HilbertRaum_{0}) = \kappa = \dim(\HilbertRaum)$.
        So without loss of generality, one may assume that $\HilbertRaum_{0} = \HilbertRaum$.
        It follows that there exists
            an isometry ${r\in\BoundedOps{\HilbertRaum}}$
            and
            an $\topSOT$-continuous
            unitary representation $U$ of $G$ on $\HilbertRaum$
        such that

            \noparskip
            \begin{equation}
            \label{eq:dil:\beweislabel}
                T(x)
                =
                r^{\ast}\,U(x)\,r
            \end{equation}

        \continueparagraph
        for all $x \in M$.

        Now let
            $P \subseteq \BoundedOps{\HilbertRaum}$
        be the index set consisting of finite projections on $\HilbertRaum$,
        directly ordered by $p \succeq q$ $:\Leftrightarrow$ $\ran(p) \supseteq \ran(q)$.
        Let $p \in P$ be arbitrary
        and let $F_{p} \subseteq \HilbertRaum$ be a finite ONB for $\ran(p)$.
        Since $r$ is an isometry,
            $\tilde{F}_{p} \colonequals \{r\,e \mid e\in F_{p}\}$
        is also a finite orthonormal family of vectors.
        Let
            $B_{p},\tilde{B}_{p}^{\prime} \subseteq \HilbertRaum$
        be ONBs extending $F_{p},\tilde{F}_{p}$ respectively.
        Since $\HilbertRaum$ is infinite dimensional
        and $F_{p},\tilde{F}_{p}$ are finite,
        one has
            $
                \card{B_{p} \setminus F_{p}}
                = \dim(\HilbertRaum)
                = \card{\tilde{B}_{p} \setminus \tilde{F}_{p}}
            $.
        Thus there exists a bijection
            ${f \colon B_{p} \setminus F_{p} \to \tilde{B}_{p} \setminus \tilde{F}_{p}}$.
        Thus
            $g \colonequals r\restr{F_{p}} \cup f$
        is a bijection between $B_{p}$ and $\tilde{B}_{p}$.
        This extends uniquely to a unitary operator
            $w_{p} \in \BoundedOps{\HilbertRaum}$.
        By construction,
            ${w_{p}\restr{F_{p}} = r\restr{F_{p}}}$
        and thus by linearity

            \noparskip
            \begin{equation}
            \label{eq:unitary-squash:\beweislabel}
                w_{p}p = rp
            \end{equation}

        \continueparagraph
        for each $p \in P$.
        Finally, set

            \noparskip
            $$
                U^{(p)} \colonequals w_{p}^{\ast}\,U(\cdot)\,w_{p}
            $$

        \continueparagraph
        for each $p \in P$,
        which are clearly \topSOT-continuous unitary representations of $G$ on $\HilbertRaum$.
        We demonstrate that the net
            $(U^{(p)}\restr{M})_{p \in P}$
        of \topSOT-continuous homomorphisms,
        is an \emph{exact weak approximation} of $T$.
        To this end, let
            $\xi,\eta\in\HilbertRaum$ be arbitrary.
        Let $p_{0}\in\BoundedOps{\HilbertRaum}$
        be the projection onto $\linspann\{\xi,\eta\}$.
        For each $p \in P$ with $p \succeq p_{0}$ one has
        that
            $p\xi = \xi$
            and
            $p\eta = \eta$.
        By \eqcref{eq:unitary-squash:\beweislabel},
            $
                w_{p}^{\ast}r \xi
                = w_{p}^{\ast} r p \xi
                = w_{p}^{\ast} w_{p} p \xi
                = p \xi
                = \xi
            $
        and similarly
            $
                w_{p}^{\ast}r \eta
                = \eta
            $.
        The dilation yields

            \noparskip
            \begin{eqnarray*}
                \brkt{
                    T(x)
                    \xi
                }{\eta}
                    &\eqcrefoverset{eq:dil:\beweislabel}{=} &\brkt{
                            r^{\ast}\,U(x)\,r
                            \xi
                        }{\eta}\\
                    &= &\brkt{
                            r^{\ast}
                            w_{p}
                            U^{(p)}(x)
                            w_{p}^{\ast}
                            r
                            \xi
                        }{\eta}\\
                    &= &\brkt{
                            U^{(p)}(x)
                            w_{p}^{\ast}
                            r
                            \xi
                        }{w_{p}^{\ast}r\eta}\\
                    &= &\brkt{
                            U^{(p)}(x)
                            \xi
                        }{\eta}\\
            \end{eqnarray*}

        \continueparagraph
        for all $x \in M$
        and $p \succeq p_{0}$.
        Hence
            $(U^{(p)}\restr{M})_{p \in P}$
        is an exact weak approximation of $T$.
\end{proof}

\def\beweislabel{thm:unitary-approx:regular-weak:sig:article-stochastic-raj-dahya}
\begin{proof}[of \Cref{\beweislabel}]
    The implications \punktcref{2}{}\ensuremath{\implies}{}\punktcref{3}{}\ensuremath{\implies}{}\punktcref{4} are clear
    irrespective of the assumptions on $\dim(\HilbertRaum)$.

    \paragraph{\punktcref{4}{}\ensuremath{\implies}{}\punktcref{1}:}
        Let
            $(U^{(\alpha)})_{\alpha\in\Lambda}$,
        be a net of \topSOT-continuous unitary representations of $G$ on $\HilbertRaum$.
        Suppose that $(U^{(\alpha)}\restr{M})_{\alpha\in\Lambda}$
        approximates $T$ in the \emph{pointwise regular weak} sense.
        We now make use of the \emph{discrete functional calculi}
            ${
                \funcCalcDiscr,\funcCalcDiscr^{(\alpha)}
                \colon
                c_{00}(G) \to \BoundedOps{\HilbertRaum}
            }$
        associated with $T$
        and each $U^{(\alpha)}\restr{M}$ respectively
        (see \S{}\ref{sec:functional-calculus:discrete:sig:article-stochastic-raj-dahya}).
        By the characterisation in
            \Cref{lemm:discr-functional-calculus:sig:article-stochastic-raj-dahya},
        each
            $\funcCalcDiscr^{(\alpha)}$
        is completely positive
        and, in order to show that $T$ has a regular unitary dilation,
        it suffices to show that
            $\funcCalcDiscr$
        is completely positive.
        To this end, first observe for $f\in c_{00}(G)$,
        that pointwise regular weak convergence yields

            \noparskip
            \begin{eqnarray*}
                \brkt{
                    \funcCalcDiscr^{(\alpha)}(f)
                    \xi
                }{\eta}
                &= &\displaystyle
                    \sum_{x \in \supp(f)}
                        f(x)
                        \brkt{
                            U^{(\alpha)}(x^{-})^{\ast}U^{(\alpha)}(x^{+})
                            \xi
                        }{\eta}\\
                &\underset{\alpha}{\longrightarrow}
                    &\displaystyle
                    \sum_{x \in \supp(f)}
                        f(x)
                        \brkt{
                            T(x^{-})^{\ast}T(x^{+})
                            \xi
                        }{\eta}\\
                &= &\brkt{
                    \funcCalcDiscr(f)
                    \xi
                }{\eta}\\
            \end{eqnarray*}

        \continueparagraph
        for all $\xi,\eta\in\HilbertRaum$.
        Thus
            ${
                \funcCalcDiscr^{(\alpha)}(f)
                \underset{\alpha}{\longrightarrow}
                \funcCalcDiscr(f)
            }$
        \wrt the $\topWOT$-topology
        for each $f\in c_{00}(G)$.
        It follows that
            ${
                (\funcCalcDiscr^{(\alpha)} \otimes \id_{M_{n}})(\mathbf{a})
                \underset{\alpha}{\longrightarrow}
                (\funcCalcDiscr \otimes \id_{M_{n}})(\mathbf{a})
            }$
        \wrt the $\topWOT$-topology
        for
            $n\in\naturals$
            and matrices
            $\mathbf{a} = (a_{ij})_{ij} \in M_{n}(c_{00}(G))$.
        Since
            $\funcCalcDiscr^{(\alpha)} \otimes \id_{M_{n}}$
        is positive for each $\alpha$ and each $n\in\naturals$,
        it follows that
            $\funcCalcDiscr \otimes \id_{M_{n}}$
        is positive for each $n\in\naturals$.
        Thus $\funcCalcDiscr$ is completely positive.

    \paragraph{\punktcref{1}{}\ensuremath{\implies}{}\punktcref{2}, under cardinality assumption:}
        The proof is analogous to proof of
            \eqcref{it:1:thm:unitary-approx:weak:sig:article-stochastic-raj-dahya}{}\ensuremath{\implies}{}\eqcref{it:2:thm:unitary-approx:weak:sig:article-stochastic-raj-dahya}
            of
            \Cref{thm:unitary-approx:weak:sig:article-stochastic-raj-dahya}.
        Relying on the cardinality assumptions, the same arguments as above
        yield a regular unitary dilation of $T$ of the form $(\HilbertRaum,U,r)$,
        \idest

            \noparskip
            \begin{equation}
            \label{eq:dil:\beweislabel}
                T(x^{-})^{\ast}T(x^{+})
                = r^{\ast}\,U(x)\,r
            \end{equation}

        \continueparagraph
        for all $x \in G$.
        The net
            $(w_{p})_{p \in P}$
            of unitary operators
        and the net
            $(U^{(p)} \colonequals w_{p}^{\ast}U(\cdot)w_{p})_{p \in P}$
            of \topSOT-continuous unitary representations of $G$ on $\HilbertRaum$
        are constructed as above.
        For $\xi,\eta\in\HilbertRaum$,
        letting $p_{0}$ be the projection onto $\linspann\{\xi,\eta\}$,
        one has again
            $w_{p}^{\ast}r\xi=\xi$
            and
            $w_{p}^{\ast}r\eta=\eta$
        for $p \succeq p_{0}$.
        The regular dilation yields

            \noparskip
            \begin{eqnarray*}
                \brkt{
                    T(x^{-})^{\ast}
                    T(x^{+})
                    \xi
                }{\eta}
                    &\eqcrefoverset{eq:dil:\beweislabel}{=}
                        &\brkt{
                            r^{\ast}\,U(x)\,r
                            \xi
                        }{\eta}\\
                    &= &\brkt{
                            r^{\ast}
                            w_{p}
                            U^{(p)}(x)
                            w_{p}^{\ast}
                            r
                            \xi
                        }{\eta}\\
                    &= &\brkt{
                            U^{(p)}(x)
                            w_{p}^{\ast}
                            r
                            \xi
                        }{w_{p}^{\ast}r\eta}\\
                    &= &\brkt{
                            U^{(p)}(x)
                            \xi
                        }{\eta}\\
            \end{eqnarray*}

        \continueparagraph
        for all $x \in G$
        and $p \succeq p_{0}$.
        Hence
            $(U^{(p)}\restr{M})_{p \in P}$
        is an \emph{exact regular weak approximation} of $T$.
\end{proof}



As an immediate application of \Cref{thm:unitary-approx:regular-weak:sig:article-stochastic-raj-dahya},
we demonstrate an infinite class of commuting systems
which admit no regular weak unitary approximations.
The following examples demonstrate in particular,
that the problem of unitary approximability (in the \emph{regular} case)
of a commuting system cannot be reduced to the unitary approximability of strict subsystems.

\begin{cor}
\makelabel{cor:counter-examples-unitary-approx:sig:article-stochastic-raj-dahya}
    Let
        $d\in\naturals$ with $d \geq 2$
        and
        $\HilbertRaum$ be an infinite dimensional Hilbert space.
    Then there exists an infinite class
    of commuting families
        $\{T_{i}\}_{i=1}^{d}$
    of contractive $\Cnought$-semigroups on $\HilbertRaum$
    whose generators have strictly negative spectral bounds,%
    \footnote{%
        The \highlightTerm{spectral bound} of a linear operator
            ${A \colon \dom(A)\subseteq\HilbertRaum\to\HilbertRaum}$
        is given by
            $\sup\{\Re \lambda \mid \lambda\in\opSpectrum{A}\}$
        (\cf \cite[Definition~1.12]{EngelNagel2000semigroupTextBook}).
    }
    such that
        $\{T_{i}\}_{i \in C}$
    has an exact regular weak unitary approximation
    for each $C \subsetneq \{1,2,\ldots,d\}$,
    whilst
        $\{T_{i}\}_{i=1}^{d}$
    has no pointwise regular weak unitary approximation.
\end{cor}

The class of semigroups can be constructed as in
\cite[Proposition~5.3]{Dahya2023dilation}.
For the reader's convenience, we sketch the construction.
We apply the characterisation in
\Cref{thm:unitary-approx:regular-weak:sig:article-stochastic-raj-dahya}
to $(G,M) = (\reals^{d},\realsNonNeg^{d})$
(see \Cref{e.g.:monoids:positivity:sig:article-stochastic-raj-dahya})
as well as the \onetoone-correspondence between
between \topSOT-continuous homomorphisms defined over $\realsNonNeg^{d}$
and commuting families of $\Cnought$-semigroups
discussed in \S{}\ref{sec:introduction:sig:article-stochastic-raj-dahya}.

    \begin{proof}[of \Cref{cor:counter-examples-unitary-approx:sig:article-stochastic-raj-dahya}]
        By assumption, one can find orthonormal closed subspaces
            $\HilbertRaum_{1},\HilbertRaum_{2}\subseteq\HilbertRaum$
        with
            $0 < \dim(\HilbertRaum_{2}) \leq \dim(\HilbertRaum_{1})$
        such that
            $\HilbertRaum=\HilbertRaum_{1}\bigoplus\HilbertRaum_{2}$.
        Working \wrt this partition,
        and letting
            $\alpha \in (\frac{1}{\sqrt{d}},\:\frac{1}{\sqrt{d-1}})$
        be arbitrary,
        we consider for each
            $i\in\{1,2,\ldots,d\}$
        bounded operators of the form

            \noparskip
            $$
                A_{i} = -\onematrix
                +
                \begin{smatrix}
                    \zeromatrix &2\alpha V_{i}\\
                    \zeromatrix &\zeromatrix\\
                \end{smatrix}
            $$

        \continueparagraph
        where the
            $V_{i} \in \BoundedOps{\HilbertRaum_{2}}{\HilbertRaum_{1}}$
        can be chosen to be any isometries.
        One can show that $\{A_{i}\}_{i=1}^{d}$ is a commuting family of
        dissipative operators whose spectra are each given by $\{-1\}$.
        Thus, $\{T_{i} \colonequals (e^{t A_{i}})_{t\in\realsNonNeg}\}_{i=1}^{d}$
        is a commuting family of contractive $\Cnought$-semigroups,
        whose (bounded) generators have strictly negative spectral bounds.
        As in \cite[Proposition~5.3]{Dahya2023dilation},
        it can be shown that
            $\{A_{i}\}_{i \in C}$ is completely dissipative
            for each $C \subsetneq \{1,2,\ldots,d\}$,
            whilst
            $\{A_{i}\}_{i=1}^{d}$ is not completely dissipative.
        By the characterisation for semigroups with bounded generators
            (\Cref{thm:classification:dissipativity:sig:article-stochastic-raj-dahya}),
        it follows that
            $\{T_{i}\}_{i \in C}$ has a simultaneous regular unitary dilation
            for each $C \subsetneq \{1,2,\ldots,d\}$,
            whilst
            $\{T_{i}\}_{i=1}^{d}$ does not.
        By \Cref{thm:unitary-approx:regular-weak:sig:article-stochastic-raj-dahya},
        the claim follows.
    \end{proof}

\begin{rem}
\label{rem:conjecture-ordinary-unitary-dilation:sig:article-stochastic-raj-dahya}
    It is well known that commuting systems of $d=2$ contractive $\Cnought$-semigroups
    have simultaneous unitary dilations
    (see
        \cite{Slocinski1974},
        \cite[Theorem~2]{Slocinski1982},
        and
        \cite[Theorem~2.3]{Ptak1985}%
    ).
    Hence by
        \Cref{%
            thm:unitary-approx:weak:sig:article-stochastic-raj-dahya,%
            thm:unitary-approx:regular-weak:sig:article-stochastic-raj-dahya,%
        }
    and \Cref{cor:counter-examples-unitary-approx:sig:article-stochastic-raj-dahya},
    the topologies defined by
        exact weak
        (\respectively uniform weak)
        convergence
    are in general%
    \footnote{%
        \viz for commuting families of $d \geq 2$ semigroups
        over infinite dimensional Hilbert spaces.
    }
    strictly weaker than their \emph{regular} counterparts.
    In particular,
    the characterisations in
        \Cref{%
            thm:unitary-approx:weak:sig:article-stochastic-raj-dahya,%
            thm:unitary-approx:regular-weak:sig:article-stochastic-raj-dahya,%
        }
    now provide a sharp topological distinction
    between \emph{unitary dilations}
    and \emph{regular unitary dilations}.
\end{rem}

At the start of this paper we mentioned two different ways to treat irreversible systems:
via embeddings into or approximations by reversible systems.
The characterisations in
    \Cref{%
        thm:unitary-approx:weak:sig:article-stochastic-raj-dahya,%
        thm:unitary-approx:regular-weak:sig:article-stochastic-raj-dahya,%
    },
demonstrate that, under modest conditions,
these are in fact equivalent
for respective choices of dilations and approximations.
This holds for the commutative setting
and for a large class of classical dynamical systems,
including dynamical systems consisting of
families of $\Cnought$-semigroups on Hilbert spaces
satisfying a semigroup version of the CCR in Weyl-form
(%
    see
    \Cref{e.g.:non-commuting-family-heisenberg-d-C:sig:article-stochastic-raj-dahya},
    \Cref{e.g.:monoids:e-joint:sig:article-stochastic-raj-dahya},
    and
    \Cref{e.g.:monoids:positivity:sig:article-stochastic-raj-dahya}%
).




\setcounternach{section}{1}
\appendix



\section[Elements in the domains of generators]{Elements in the domains of generators}
\label{app:generators:sig:article-stochastic-raj-dahya}

\firstparagraph
In this appendix we construct elements in the domains of generators of $\Cnought$-semigroups,
and then extend this to the situation of commuting families of $\Cnought$-semigroups.
For the main ideas here, we refer the reader to
    \cite[Proposition G.2.5]{HytNeervanMarkLutz2016bookVol2},
    \cite[Theorem~I.5.8]{EngelNagel2000semigroupTextBook},
    \cite[Theorem~A.8]{Dahya2022complmetrproblem}.

Let $u_{t}$ denote a \randomvar concentrated and uniformly distributed on $[0,\:t]$,
or shorter: $u_{t} \distributedAs U([0,\:t])$ for each $t\in\realsNonNeg$.
Consider a $\Cnought$-semigroup $T$ over a Banach space.
By \topSOT-continuity and the uniform boundedness principle,
we have that $\{T(t) \mid t \in V\}$ is uniformly norm-bounded
for compact subsets $V \subseteq \realsNonNeg$ of $0$.
It follows that the expectation
    $\Expected[T(u_{t})]$
exists and satisfies the bounds
    $
        \norm{\Expected[T(u_{t})]}
        \leq \sup_{s\in[0,\:t]}\norm{T(s)}
        < \infty
    $
for all $t\in\realsPos$.
Hence we also have that
    $\{\Expected[T(u_{t})] \mid t \in V\}$
is uniformly norm-bounded
for bounded subsets $V \subseteq \realsNonNeg$.


\begin{prop}
\makelabel{prop:app:dense-subspace-one-param:sig:article-stochastic-raj-dahya}
    Let $T$ be a $\Cnought$-semigroup on a Banach space $\BanachRaum$
    with generator $A$.
    Further let $S \in \BoundedOps{\BanachRaum}$ commute with $T(t)$ for all $t\in\realsNonNeg$.
    Then
        $
            S\,\Expected[T(u_{t})] \xi \in \opDomain{A}
        $
    with
        $
            A\,S\,\Expected[T(u_{t})]\xi = S \,\frac{1}{t}(T(t) - \onematrix)\xi
        $
    for all $\xi\in\BanachRaum$ and $t\in\realsPos$.
\end{prop}

    \begin{proof}
        Relying on basic properties of Bochner-integrals,
        one can show for $h\in(0,\:t)$ that

            $$
                T(h)\,\Expected[T(u_{t})]
                =
                \Expected[T(u_{t})]
                + \frac{1}{t}(T(t) - \onematrix)
                    \cdot
                    h\,\Expected[T(u_{h})]
            $$

        \continueparagraph
        (see also the proof of \cite[Proposition 1.1.4]{Butzer1967semiGrApproximationsBook}).
        Thus by commutativity

            $$
                \frac{1}{h}(T(h)-\onematrix) \cdot S\,\Expected[T(u_{t})]
                = S\,\frac{1}{h}(T(h)\,\Expected[T(u_{t})] - \Expected[T(u_{t})])
                = S\,\frac{1}{t}(T(t) - \onematrix)
                    \,\Expected[T(u_{h})]
            $$

        \continueparagraph
        for sufficiently small $h\in\realsPos$.
        Now, since $T$ is $\topSOT$-continuous and $T(0)=\onematrix$,
        one has
            ${
                \Expected[T(u_{h})]
                \overset{\tinytopSOT}{\longrightarrow}
                \onematrix
            }$
        for ${\realsPos \ni h \longrightarrow 0}$
        (\cf the subsequent paragraph below (1.1.7) in \cite{Butzer1967semiGrApproximationsBook}).
        It follows that
            ${
                \frac{1}{h}(T(h)-\onematrix)
                \,S\,\Expected[T(u_{t})]
                    \overset{\tinytopSOT}{\longrightarrow}
                        S\,\frac{1}{t}(T(t) - \onematrix)
            }$
        for ${\realsPos \ni h\longrightarrow 0}$.
        This proves the claim.
    \end{proof}

\begin{prop}
\makelabel{prop:dense-subspace-multi-param:sig:article-stochastic-raj-dahya}
    Let $d\in\naturals$,
    and $\{T_{i}\}_{i=1}^{d}$ be
    a commuting family of $\Cnought$-semigroups
    over a Banach space $\BanachRaum$
    with generators $\{A_{i}\}_{i=1}^{d}$.
    For each
        $\mathbf{t} \in \realsPos^{d}$
    and $\xi\in\BanachRaum$
    define

        $$
            \tilde{\xi}_{\mathbf{t}}
            \colonequals
                \Expected[T_{1}(u_{t_{1}})]
                \Expected[T_{2}(u_{t_{2}})]
                \ldots
                \Expected[T_{d}(u_{t_{d}})]
                \xi.
        $$

    \continueparagraph
    Then
        $
            D_{0} \colonequals \{
                \tilde{\xi}_{\mathbf{t}}
                \mid
                \xi \in \BanachRaum,
                \mathbf{t}\in\realsNonNeg^{d}
            \}
        $
    spans a dense linear subspace $D$ of $\BanachRaum$.
    Furthermore, for injective sequences
        $(k_{i})_{i=1}^{n} \subseteq \{1,2,\ldots,d\}$,
        $n\in\{1,2,\ldots,d\}$
    one has
        $
            D \subseteq \opDomain{
                A_{k_{n}} \cdot \ldots \cdot A_{k_{2}} \cdot A_{k_{1}}
            }
        $.
    And for each $\xi \in D$, the value of
        ${(A_{k_{n}} \cdot \ldots \cdot A_{k_{2}} \cdot A_{k_{1}})\xi}$
    does not depend on the order of the $k_{i}$.
\end{prop}

    \begin{proof}
        As explained in the proof of \Cref{prop:app:dense-subspace-one-param:sig:article-stochastic-raj-dahya},
        one has for each $i\in\{1,2,\ldots,d\}$
        that
            $
                \Expected[T_{i}(u_{t})]
                = \frac{1}{t}\cdot\sotInt_{s\in[0,\:t]}T(s)\:\dee s
                \overset{\tinytopSOT}{\longrightarrow}
                \onematrix
            $
        for ${\realsPos \ni t \longrightarrow 0}$.
        As explained above,
            $\{ \Expected[T_{i}(u_{t})] \mid t \in (0,\:h)\}$
        is uniformly bounded for all $h\in\realsPos$.
        Since multiplication is \topSOT-continuous on bounded subsets of $\BoundedOps{\BanachRaum}$,
        it follows that
            ${
                \tilde{\xi}_{\mathbf{t}}
                = \Expected[T_{1}(u_{t_{1}})]
                    \Expected[T_{2}(u_{t_{2}})]
                    \ldots
                    \Expected[T_{d}(u_{t_{d}})]
                    \xi
                \longrightarrow
                    \xi
            }$
        in norm for
            ${\realsPos^{d} \ni \mathbf{t} \longrightarrow \zerovector}$.
        Thus $D_{0}$ is dense in $\BanachRaum$.

        Towards the final claim, it suffices to consider elements of $D_{0}$.
        So consider
            $\tilde{\xi}_{\mathbf{t}} \in D_{0}$
        for some
            $\mathbf{t} \in \realsPos^{d}$
            and
            $\xi\in\BanachRaum$
        and let
            $(k_{i})_{i=1}^{n} \subseteq \{1,2,\ldots,d\}$
            be an injective sequence
            for some
            $n\in\{1,2,\ldots,d\}$.
        Letting
            $\tau_{i} \distributedAs U([0,\:t_{i}])$, $i\in\{1,2,\ldots,d\}$
            be independent \randomvar's,
        one has
            $
                \tilde{\xi}_{\mathbf{t}}
                = \Expected[T_{1}(\tau_{1})]
                \Expected[T_{2}(\tau_{2})]
                \ldots
                \Expected[T_{d}(\tau_{d})]
                    \xi
            $.
        Also, by independence and commutativity one obtains
            $
                \Expected[T_{i}(\tau_{i})]
                \Expected[T_{j}(\tau_{j})]
                = \Expected[T_{i}(\tau_{i})T_{j}(\tau_{j})]
                = \Expected[T_{j}(\tau_{j})T_{i}(\tau_{i})]
                = \Expected[T_{j}(\tau_{j})]
                \Expected[T_{i}(\tau_{i})]
            $
        as well as
            $
                \Expected[T_{i}(\tau_{i})]
                    \,\frac{1}{t_{j}}(T(t_{j}) - \onematrix)
                = \Expected[
                        T_{i}(\tau_{i})\frac{1}{t_{j}}
                        \,(T(t_{j}) - \onematrix)
                    ]
                = \Expected[
                        \frac{1}{t_{j}}(T(t_{j}) - \onematrix)\,T_{i}(\tau_{i})
                    ]
                = \frac{1}{t_{j}}(T(t_{j}) - \onematrix)
                    \,\Expected[T_{i}(\tau_{i})]
            $
        for all $i,j\in\{1,2,\ldots,d\}$ with $i \neq j$.
        Since these operators commute, one may apply
        \Cref{prop:app:dense-subspace-one-param:sig:article-stochastic-raj-dahya}
        and obtain by induction over $n$ that
            $
                \tilde{\xi}_{\mathbf{t}}
                =
                    \Expected[T_{1}(\tau_{1})]
                    \Expected[T_{2}(\tau_{2})]
                    \ldots
                    \Expected[T_{d}(\tau_{d})]
                        \xi
                \in
                \opDomain{A_{k_{n}} \cdot \ldots \cdot A_{k_{2}} \cdot A_{k_{1}}}
            $
        with

            \noparskip
            \begin{equation}
            \label{eq:1:\beweislabel}
                (A_{k_{n}} \cdot \ldots \cdot A_{k_{2}} \cdot A_{k_{1}})
                \tilde{\xi}_{\mathbf{t}}
                = \Big(
                    \displaystyle
                    \prod_{i=1}^{d}
                    \begin{cases}
                        \frac{1}{t_{i}}(T(t_{i}) - \onematrix)
                            &\colon  &i\in\{k_{1},k_{2},\ldots,k_{n}\}\\
                        \Expected[T_{i}(u_{i})]
                            &\colon &\text{otherwise}\\
                    \end{cases}
                \Big)
                \:\xi.
            \end{equation}

        \continueparagraph
        In particular, by \eqcref{eq:1:\beweislabel} the value is independent of the order of the $k_{i}$.
    \end{proof}







\firstparagraph
\paragraph{Acknowledgement.}
The author is grateful to
    Tanja Eisner
    for her patient and detailed feedback,
to
    Markus Haase
    for pointing out useful literature on Bochner-integrals,
to
    Ulrich Groh
    for discussions about positive operators,
and to the referee for their constructive feedback.


\bibliographystyle{siam}
\def\bibname{References}
\bgroup

\egroup


\addresseshere
\end{document}
